\documentclass[11pt]{article}
\usepackage[margin=1in]{geometry}
\usepackage{amsmath,amssymb,amsthm,mathtools}
\usepackage{bm}
\usepackage{enumitem}
\usepackage{tikz}
\usepackage{xcolor}
\usetikzlibrary{arrows.meta,calc,decorations.markings,decorations.pathreplacing,patterns,positioning}
\definecolor{HillBlue}{RGB}{35,92,145}
\definecolor{HillTeal}{RGB}{24,137,135}
\definecolor{HillOrange}{RGB}{213,112,42}
\definecolor{HillCore}{RGB}{217,111,72}
\usepackage{hyperref}
\numberwithin{equation}{section}
\newtheorem{theorem}{Theorem}[section]
\newtheorem{proposition}[theorem]{Proposition}
\newtheorem{remark}[theorem]{Remark}

\newtheorem{lemma}[theorem]{Lemma}
\newtheorem{corollary}[theorem]{Corollary}

\newcommand{\R}{\mathbb{R}}
\newcommand{\T}{\mathbb{T}}
\newcommand{\Z}{\mathbb{Z}}
\newcommand{\N}{\mathbb{N}}
\newcommand{\E}{\mathcal{E}}
\newcommand{\diver}{\operatorname{div}}
\newcommand{\supp}{\operatorname{supp}}

\newcommand{\Id}{\mathrm{Id}}

\newcommand{\dd}{\mathrm{d}}
\title{Flexibility for the Three-Dimensional Navier--Stokes Equations via Moving Hill Vortices}
\author{
	Quoc-Hung Nguyen\thanks{State Key Laboratory of Mathematical Sciences,
	Academy of Mathematics and Systems Science, Chinese Academy of Sciences,
	Beijing 100190, China; Institute of Mathematics, Academy of Mathematics
	and Systems Science, the Chinese Academy of Sciences, Beijing 100190,
	China. Email: \texttt{qhnguyen@amss.ac.cn}.}
	\and
	Zexi Wang\thanks{Institute of Mathematics, Academy of Mathematics and Systems Science,
	Chinese Academy of Sciences, 100190 Beijing, PR China.
	Email: \texttt{wangzexi25@mails.ucas.ac.cn}.}
}
\date{}

\begin{document}
	\maketitle

	\begin{abstract}
		We construct weak solutions of the three-dimensional incompressible
		Navier--Stokes equations on the torus.  The convex-integration scheme is
		based on the moving-dipole construction of Bru\`e, Colombo, and
		Kumar~\cite{BrueColomboKumar2024}.  For the explicit exponent
		\[
		\bar p=\frac65+5\times10^{-5},
		\]
		and for any two mean-zero, divergence-free vector fields in
		\(L^2(\mathbb T^3)\), we construct a weak solution whose traces at times
		\(0\) and \(1\) approximate the prescribed fields arbitrarily well and
		which satisfies
		\[
		u\in C([0,1];L^2(\mathbb T^3)),
		\qquad
		\nabla u\in C([0,1];L^{\bar p}(\mathbb T^3)).
		\]
		Exploiting the time-locality of the iteration, we also obtain exact
		nonuniqueness for a dense set of initial data in
		\(L^2_\sigma(\mathbb T^3)\).
		The principal perturbations are localized, rescaled copies of Hill's
		spherical vortex.
		The Hill scaling preserves both the kinetic-energy scale and
		the \(L^{6/5}\)-scale of the velocity gradient.  The construction uses
		localization of the potential exterior flow, long-orbit averaging
		of moving vortex cores, an auxiliary source correction, and a temporal
		corrector compatible with uniform-in-time Sobolev control.
	\end{abstract}
	
	\section{Introduction}
	
	\subsection{The problem and the main result}
	
	We consider the incompressible Navier--Stokes equations on the
	three-dimensional torus:
	\begin{equation}\label{eq:intro-NS}
		\begin{cases}
			\partial_tu+\diver(u\otimes u)-\nu\Delta u+\nabla p=0,
			& (x,t)\in\T^3\times(0,1),\\
			\diver u=0,
		\end{cases}
	\end{equation}
	where \(u:\T^3\times[0,1]\to\R^3\) is the velocity, \(p\) is the
	pressure, and \(\nu>0\) is fixed.
	We take \(\T^3=(\R/\Z)^3\) with normalized Lebesgue
	measure \(|\T^3|=1\).  Throughout the paper, all velocity fields have
	zero spatial mean.

	We write
	\[
	L^2_\sigma(\T^3)
	:=
	\left\{
	v\in L^2(\T^3;\R^3):
	\diver v=0,\quad \int_{\T^3}v(x)\,\dd x=0
	\right\}.
	\]
	A finite-energy weak solution of \eqref{eq:intro-NS} is a vector field
	\[
	u\in C([0,1];L^2_\sigma(\T^3))
	\]
	such that
	\begin{align}\label{eq:intro-weak}
		&\int_0^1\!\!\int_{\T^3}
		\left[
		u\cdot\partial_t\varphi
		+(u\otimes u):\nabla\varphi
		+\nu u\cdot\Delta\varphi
		\right]\,\dd x\,\dd t
		+\int_{\T^3}u(x,0)\cdot\varphi(x,0)\,\dd x=0
	\end{align}
	for every
	\(\varphi\in C_c^\infty(\T^3\times[0,1);\R^3)\) satisfying
	\(\diver\varphi=0\).  The pressure can then be recovered as a
	distribution.  No energy inequality is included in this definition.
	
	Our main result is the following.
	
	\begin{theorem}[Flexibility]\label{thm:main}
		Let
		\[
		\bar p=\frac65+5\times10^{-5}.
		\]
		For every \(\varepsilon>0\) and every
		\(u^{(0)},u^{(1)}\in L^2_\sigma(\T^3)\), there exists a weak solution
		\((u,p)\) of \eqref{eq:intro-NS}, in the sense of
		\eqref{eq:intro-weak}, such that
		\[
		u\in C([0,1];L^2(\T^3)),
		\qquad
		\nabla u\in C([0,1];L^{\bar p}(\T^3)).
		\]
		Moreover,
		\[
		\|u(\cdot,0)-u^{(0)}\|_{L^2(\T^3)}<\varepsilon,
		\qquad
		\|u(\cdot,1)-u^{(1)}\|_{L^2(\T^3)}<\varepsilon.
		\]
	\end{theorem}

	The numerical value of \(\bar p\) is not claimed to be optimal.  It is
	an explicit exponent for which all inequalities in the parameter
	selection are strict.  The threshold \(6/5\) comes from
	the Hill scaling used in the present construction.  Also,
	Theorem~\ref{thm:main} is not a statement in the Leray--Hopf class:
	the regularity \(\nabla u\in C_tL_x^{\bar p}\), with
	\(\bar p\) close to \(6/5\), does not imply
	\(u\in L_t^2H_x^1\), and the solutions constructed here are not asserted
	to satisfy the energy inequality.
	
	\subsection{Interpretation and consequences of the theorem}

	Theorem~\ref{thm:main} can be restated as a density result for pairs of
	endpoint traces.
	Let \(\mathcal S_{\bar p}\) denote the set of all weak
		solutions of \eqref{eq:intro-NS} satisfying
		\(u\in C([0,1];L^2_\sigma(\T^3))\) and
		\(\nabla u\in C([0,1];L^{\bar p}(\T^3))\).
	
	Then
	\begin{equation}\label{eq:intro-trace-density}
		\overline{
			\left\{
			\bigl(u(\cdot,0),u(\cdot,1)\bigr):
			u\in\mathcal S_{\bar p}
			\right\}
		}^{\,L^2\times L^2}
		=
		L^2_\sigma(\T^3)\times L^2_\sigma(\T^3).
	\end{equation}
	Thus weak trajectories in this class connect arbitrarily small
	neighborhoods of any two prescribed states.  No compatibility condition
	is imposed on the endpoint targets.
	
	This is an approximate, rather than exact, controllability statement.
	Theorem~\ref{thm:main} does not assert that both prescribed traces are
	attained exactly.  The density relation~\eqref{eq:intro-trace-density}
	states only that the set of endpoint pairs of weak solutions is dense
	in the product energy topology.  Nevertheless, the construction contains
	additional information that is not visible in this density statement:
	the iteration is local in time.
	Starting from two Reynolds flows that agree near \(t=0\), the successive
	perturbations can be chosen so that their limits continue to agree on a
	nontrivial initial time interval.  Choosing two distinct terminal targets
	then gives the following exact same-data consequence.

	\begin{theorem}[Nonuniqueness]\label{thm:nonuniqueness}
		There exists a dense set
		\(\mathcal D\subset L^2_\sigma(\T^3)\) such that every
		\(u_{\rm in}\in\mathcal D\) is the initial datum of at least two
		distinct weak solutions \(u^{(1)},u^{(2)}\in\mathcal S_{\bar p}\).
		More precisely, the two solutions may be chosen so that
		\[
			u^{(1)}(\cdot,t)=u^{(2)}(\cdot,t)
			\quad\text{for every }t\in[0,t_*]
		\]
		for some \(t_*>0\), while
		\(u^{(1)}(\cdot,1)\ne u^{(2)}(\cdot,1)\).
	\end{theorem}
Thus the same scheme yields exact nonuniqueness on a dense subset of the
	finite-energy data. The theorem does not claim nonuniqueness for every
	prescribed element of \(L^2_\sigma(\T^3)\), and the resulting solutions
	are not asserted to satisfy the energy inequality.
	
	The flexibility is incompatible with the Leray--Hopf energy inequality
	for suitable target pairs.  Indeed, take \(u^{(0)}=0\), \(u^{(1)}\neq0\), and
	\(0<\varepsilon<\|u^{(1)}\|_{L^2}/3\).  The solution supplied by the theorem
	satisfies
	\[
	\|u(\cdot,0)\|_{L^2}<\varepsilon,
	\qquad
	\|u(\cdot,1)\|_{L^2}>
	\|u^{(1)}\|_{L^2}-\varepsilon>2\varepsilon,
	\]
	and therefore cannot obey a nonincreasing kinetic-energy inequality.
	Hence the solution does not belong to the Leray--Hopf class.
	
	The theorem gives the following small improvement over \(L_x^2\).  Since
	\[
	\bar p=\frac{24001}{20000},
	\qquad
	\bar p^\ast:=\frac{3\bar p}{3-\bar p}
	=\frac{72003}{35999}
	=2+\frac{5}{35999},
	\]
	Poincar\'e's inequality and Sobolev embedding give
	\begin{equation*}
		u\in C([0,1];W^{1,\bar p}(\T^3))
		\hookrightarrow
		C([0,1];L^{\bar p^\ast}(\T^3)),
		\qquad \bar p^\ast>2.
	\end{equation*}
	This is a fixed-time spatial gain, rather than an estimate obtained by
	integrating a stronger norm in time.  The class nevertheless remains
	far below the classical uniqueness range.
	
	\subsection{Background and related work}
	
	Leray's classical construction~\cite{Leray1934} on \(\R^3\), and
	Hopf's construction~\cite{Hopf1951} in the periodic and bounded-domain
	settings, give global weak solutions in
	\[
	 u\in L_t^\infty L_x^2\cap L_t^2\dot H_x^1
	\]
	satisfying the energy inequality.
	Uniqueness holds under the
	Ladyzhenskaya--Prodi--Serrin conditions
	\[
	 u\in L_t^sL_x^q,
	 \qquad \frac2s+\frac3q\leq1,\qquad q>3,
	\]
	with the usual endpoint refinements; see
	\cite{Prodi1959,Serrin1962,Ladyzhenskaya1967}, and
	\cite{EscauriazaSereginSverak2003} for the endpoint \(q=3\).
	Quantitative refinements are discussed in~\cite{Tao2019}.
	The global regularity
	problem for smooth finite-energy data remains open~\cite{Fefferman2000}.
	For an introduction to the mathematical analysis of the Euler and
	Navier--Stokes equations, including the classical well-posedness theory
	and the Leray--Hopf framework, we refer to the book of Bedrossian and
	Vicol~\cite{BedrossianVicol2022}.  Further treatments of weak, mild, and
	local-energy solutions can be found in
	\cite{LemarieRieusset2002,LemarieRieusset2007,Tsai2018}.
	Our result concerns flexibility in a distributional class with finite
	kinetic energy and a uniform spatial-gradient bound below the
	energy-dissipation level; the energy inequality is not imposed.

	For comparison, we recall the Cauchy theory in critical spaces.
	Kato~\cite{Kato1984} proved local well-posedness in \(L^3(\R^3)\) and
	global existence for small data, while Koch and
	Tataru~\cite{KochTataru2001} extended the small-data theory to
	\(BMO^{-1}\).  Bourgain and Pavlovi\'c~\cite{BourgainPavlovic2008}
	proved norm inflation in \(\dot B^{-1,\infty}_\infty\), and
	Germain~\cite{Germain2008} studied the associated second-iterate
	instability.  Jia and \v{S}ver\'ak constructed forward self-similar
	solutions from large \((-1)\)-homogeneous data and formulated spectral
	conditions for same-data nonuniqueness~\cite{JiaSverak2014,JiaSverak2015};
	Guillod and \v{S}ver\'ak~\cite{GuillodSverak2023} investigated this
	bifurcation scenario numerically.

	Convex integration originates in the geometric flexibility of
	Nash and Kuiper~\cite{Nash1954,Kuiper1955} and
	Gromov~\cite{Gromov1973,Gromov1986}.
	For incompressible Euler, the modern theory developed from the early
	constructions of Scheffer and
	Shnirelman~\cite{Scheffer1993,Shnirelman1997,Shnirelman2000} and the
	differential-inclusion framework of De Lellis and
	Sz\'ekelyhidi~\cite{DeLellisSzekelyhidi2009,DeLellisSzekelyhidi2013}; see
	also~\cite{DaneriSzekelyhidi2017}.  After a series of partial results,
	including~\cite{Buckmaster2015,
	BuckmasterDeLellisIsettSzekelyhidi2015}, Isett~\cite{Isett2018} proved the
	flexible part of Onsager's conjecture~\cite{Onsager1949} by constructing
	non-energy-conserving Euler solutions in \(C_tC_x^\alpha\) for every
	\(\alpha<1/3\).  Together with the energy-conservation result of
	Constantin, E, and Titi~\cite{ConstantinETiti1994} above the exponent
	\(1/3\), this resolved Onsager's conjecture.  Buckmaster, De Lellis,
	Sz\'ekelyhidi, and Vicol~\cite{BuckmasterDeLellisSzekelyhidiVicol2019}
	subsequently gave a considerably shorter proof of the flexible part.
	They also constructed solutions with any prescribed smooth positive
	energy profile and established an \(h\)-principle in \(C^\beta_{t,x}\)
	for every \(\beta<1/3\).
	Intermittent and
	localized variants were developed in
	\cite{BuckmasterMasmoudiNovackVicol2023,NovackVicol2023,
	GiriKwonNovack2024,GiriKwonNovackStrong}; related developments include the
	two-dimensional Newton--Nash scheme~\cite{GiriRadu2024} and the
	construction with prescribed generalized helicity~\cite{GiriKwonNovack2026}.
	For surveys, see~\cite{TaoNotes2019,BuckmasterVicolHydrodynamics}.
	The relation between convex integration, intermittency, and
	phenomenological models of turbulence is discussed by Buckmaster and
	Vicol in~\cite{BuckmasterVicolTurbulence2019}.
	At the level of the Euler--Reynolds system, these schemes successively
	replace a coarse stress by the quadratic self-interaction of a rapidly
	oscillating perturbation.  Intermittency adds spatial concentration to
	this mechanism and gives estimates in Lebesgue spaces below
	\(L^2\), which are used to control the viscous error.

	For the viscous problem, spatial concentration makes high-frequency
	errors small in spaces below \(L^2\).
	Buckmaster and Vicol~\cite{BuckmasterVicol2019} constructed finite-energy
	weak solutions on \(\T^3\) with arbitrary prescribed smooth nonnegative
	energy profiles.  For some \(\beta>0\), these solutions satisfy
	\[
	 v\in C\bigl([0,T];H^\beta(\T^3)\cap
	 W^{1,1+\beta}(\T^3)\bigr).
	\]
	In particular, their construction implies nonuniqueness in the class of
	finite-energy weak solutions.  Their building blocks are intermittent
	Beltrami waves: stationary
	high-frequency flows modulated by concentrated Dirichlet kernels.  The
	resulting sub-\(L^2\) estimates make both the Reynolds and viscous errors
	perturbative, while introducing several separated spatial and
	temporal scales.

	Buckmaster, Colombo, and Vicol~\cite{BuckmasterColomboVicol2022}
	combined convex integration with gluing.  Starting from two strong
	Navier--Stokes solutions, they produced a weak solution agreeing with the
	first near the initial time and with the second near the terminal time,
	with the same type of uniform-in-time regularity and smoothness outside a
	fractal set of singular times.  Their result also concerns
	time-uniform regularity.  In our setting, however, the endpoint states are arbitrary finite-energy
	targets and are approximated rather than joined through prescribed strong
	trajectories.
	For comparison, Leray--Hopf solutions have a singular time set of
	Hausdorff dimension at most \(1/2\)~\cite{Leray1934,Scheffer1977}, while
	the space--time partial-regularity theorem of Caffarelli, Kohn, and
	Nirenberg~\cite{CaffarelliKohnNirenberg1982} controls the parabolic
	Hausdorff measure of the singular set for suitable weak solutions.  The
	solutions in~\cite{BuckmasterColomboVicol2022} are not suitable weak
	solutions, so these partial-regularity results do not apply directly.
	Buckmaster, Colombo, and Vicol instead proved that the singular set of
	times has Hausdorff dimension strictly less than (1).

	Cheskidov and Luo~\cite{CheskidovLuo2022} proved nonuniqueness in
	\(L_t^sL_x^\infty\) for every \(1\le s<2\).  More precisely, for every
	\(1\le s<2\) and \(1\le q<\infty\), they constructed non-Leray--Hopf weak
	solutions in
	\[
	 L_t^sL_x^\infty\cap L_t^1W_x^{1,q}.
	\]
	Their construction combines gluing, temporal concentration, spatially
	intermittent Mikado flows, and a temporally oscillating corrector.  The
	required smallness is obtained in time-integrated norms.
	Theorem~\ref{thm:main}, by contrast, requires estimates that are uniform
	at each time, so the principal perturbations must also provide sufficient
	spatial concentration.

	In two dimensions, the same authors proved an \(L^2\)-critical
	nonuniqueness result~\cite{CheskidovLuo2023}.  Their solutions are
	uniformly continuous in \(L^p\) for every \(p<2\).  This conclusion is
	uniform in time, as is ours, but their construction uses space--time
	intermittent accelerating jets, whereas the present construction uses
	moving Hill vortices.

	Related results concern Navier--Stokes equations with fractional
	dissipation.  Luo and Titi~\cite{LuoTiti2020} proved nonuniqueness for
	the hyperviscous equation with dissipation \((-\Delta)^\alpha\) whenever
	\(\alpha<5/4\), showing the sharpness of the J.-L.\ Lions exponent.
	Nonuniqueness results in the hypodissipative regime were obtained
	in~\cite{ColomboDeLellisDeRosa2018,DeRosa2019}.  In these works the
	dissipation exponent is varied.  Here the standard Laplacian is fixed,
	and the aim is to obtain time-uniform Sobolev regularity while retaining
	enough spatial concentration to control the viscous error.

	An inviscid result with a related endpoint statement is due to Novack
	and Vicol~\cite{NovackVicol2023}.  For every \(0<\beta<1/2\), their Euler
	solutions connect arbitrary divergence-free \(L^2\) states up to a small
	endpoint error while remaining in
	\[
	 C_t\left(H_x^\beta\cap
	 L_x^{\frac{2-2\beta}{1-2\beta}}\right).
	\]
	Their construction uses a scheme with higher-order Reynolds stresses,
	which are corrected through a combinatorial placement of intermittent
	pipe flows with optimal relative intermittency.  Although the endpoint
	statement is close to Theorem~\ref{thm:main}, their equation is inviscid.
	In the present
	setting, the Laplacian of a concentrated perturbation creates an
	additional principal constraint and requires a different choice of scales.

	A distinct route to nonuniqueness is based on spectral instability.
	Albritton, Bru\`e, and Colombo constructed two suitable Leray--Hopf
	solutions on \(\R^3\) with the same force and zero initial
	data~\cite{AlbrittonBrueColombo2022}.  Their argument starts from an unstable self-similar
	vortex-ring profile and constructs a second trajectory on its unstable
	manifold.  Thus it gives exact nonuniqueness for a forced Cauchy problem,
	rather than endpoint flexibility for the unforced equation.  An
	inner--outer gluing argument later extended the
	construction to the torus and smooth bounded domains
	\cite{AlbrittonBrueColombo2023}.  This instability mechanism is based on
	Vishik's construction for the forced two-dimensional Euler
	equations~\cite{Vishik2018I,Vishik2018II}; see
	also~\cite{AlbrittonEtAl2024} for a detailed presentation.  In a
	different direction, Bru\`e and De
	Lellis~\cite{BrueDeLellis2023} constructed forced classical
	Navier--Stokes solutions exhibiting anomalous dissipation in the
	vanishing-viscosity limit.
	For a different mechanism producing anomalous diffusion through
	fractal homogenization, see Armstrong and Vicol~\cite{ArmstrongVicol2025}.

	At critical regularity, Coiculescu and
	Palasek~\cite{CoiculescuPalasek2026} proved exact nonuniqueness for the
	unforced equation from a particular datum in
	\(BMO^{-1}(\T^3)\setminus L^2(\T^3)\).  Their nested-Mikado cascade builds
	on the dyadic scenario of~\cite{PalasekDyadic2025} and is not a
	convex-integration scheme.  It branches from one particular
	infinite-energy datum and shows that smallness is essential in the
	Koch--Tataru theory; this is complementary to the finite-energy endpoint
	flexibility obtained here.  A recent preprint by Hou, Wang, and
	Yang~\cite{HouWangYang2026} presents what the authors describe as a
	computer-assisted proof of nonuniqueness for Leray--Hopf solutions of the
	unforced three-dimensional Navier--Stokes equations on \(\R^3\).  This
	result is not used in the present paper.  Ionescu, Jia, and
	Palasek~\cite{IonescuJiaPalasek2026} obtain a conditional result based on
	a candidate axisymmetric profile whose required existence remains open.
	For an overview of instability and convex integration, see
	Colombo~\cite{ColomboICM2026}.
	These results concern three different forms of
	nonuniqueness.  Instability arguments branch from a specially constructed
	coherent profile; critical-space cascades branch from a particular rough
	datum; convex integration prescribes a sequence of stresses at each stage.
	The resulting convex-integration solutions generally lie below the
	Leray--Hopf class.
	Theorem~\ref{thm:main} should be read in this sense: it gives density of
	attainable endpoint traces in the finite-energy space, while the
	nonuniqueness statement deduced later concerns a dense set of initial
	data, not every datum and not the Leray--Hopf class.

	The two-dimensional Euler equation provides relevant background for
	the moving-profile construction.  Yudovich's theorem~\cite{Yudovich1963}
	gives global well-posedness for bounded vorticity, while uniqueness for
	unforced solutions with vorticity merely in \(L^p\), \(p<\infty\), remains
	open.  At the borderline of integrability,
	Bru\`e and Colombo~\cite{BrueColombo2023} first constructed nonunique weak
	Euler solutions with uniformly bounded kinetic energy and vorticity in the
	Lorentz space \(L^{1,\infty}\).  Buck and
	Modena~\cite{BuckModena2024,BuckModena2026} subsequently extended this line
	to the real Hardy spaces \(H^p\), first for \(2/3<p<1\), and later, with
	compact spatial support, for the full range \(0<p<1\).
	Related results for transport equations appear in
	\cite{MS18,MS19,MS20,BrueColomboDeLellis2021,
		BrueColomboKumarTransport2024}; their sharpness is measured against the
	renormalization and well-posedness theories of DiPerna and
	Lions~\cite{DiPernaLions1989} and Ambrosio~\cite{Ambrosio2004}.

	The present construction is based most directly on the work of Bru\`e,
	Colombo, and Kumar~\cite{BrueColomboKumar2024}.  For
	\[
	 1<p<1+\frac{1}{6500},
	\]
	they construct a dense set of initial data in
	\(L^2(\T^2)\cap W^{1,p}(\T^2)\) giving rise to infinitely many
	nonconservative weak Euler solutions in \(C_tL_x^2\) with vorticity in
	\(C_tL_x^p\).
	Their principal perturbation is a Lamb--Chaplygin dipole concentrated in
	a moving ball; see~\cite[\S165]{Lamb1932} and
	\cite{MeleshkoVanHeijst1994} for classical accounts of this structure.
	The Reynolds stress is decomposed into
	rank-one directions, with only one direction activated on each short time
	interval.  The radius and speed of the dipole are coupled to the
	coefficient of that stress component, while its motion along a long
	rational trajectory converts a localized source into the directional
	flux required for cancellation.  This dynamical use of a traveling
	coherent structure is also used in the present paper.
	The passage from that construction to three-dimensional Navier--Stokes is
	not formal.  A planar dipole carries scalar vorticity and travels in a
	preferred direction, whereas Hill's vortex is an axisymmetric
	three-dimensional velocity field whose gradient must be estimated in all
	directions.  Moreover, rescaling the core simultaneously changes its
	energy, its gradient, its transport speed, and the size of the viscous
	defect.  Balancing these effects is what produces the exponent near
	\(6/5\) and distinguishes the present iteration from its Euler precursor.

	For SQG, the weak existence theory was developed in
	\cite{Resnick1995,Marchand2008}.  Buckmaster, Shkoller, and
	Vicol~\cite{BuckmasterShkollerVicol2019} constructed the first nonunique
	weak solutions of the unforced equation, and Isett and
	Ma~\cite{IsettMa2021} later gave a direct scalar-level construction.  In
	related work, Bru\`e, Jin, and Nguyen~\cite{BrueJinNguyen2026} develop a
	moving-profile convex-integration scheme for the inviscid SQG equation on
	\(\T^2\).  For the explicit, nonoptimized exponent
	\(\bar p=4/3+10^{-5}\), their construction approximately connects any two
	prescribed mean-zero states in \(L^{\bar p}(\T^2)\) by a momentum weak
	solution in \(C([0,1];L^{\bar p}(\T^2))\).  It also gives nonuniqueness for
	a dense subset of the mean-zero space \(L^{\bar p}(\T^2)\).  The scheme
	uses localized traveling SQG profiles whose centers move along rational
	directions, and a bilinear null-form estimate controls the derivative loss
	caused by the nonlocal relation \(u=\Lambda v\).

	The present paper follows the moving-profile strategy introduced by Bru\`e,
	Colombo, and Kumar~\cite{BrueColomboKumar2024} for the two-dimensional
	Euler equations.  In passing to the three-dimensional viscous setting, the
	Lamb--Chaplygin dipoles used there are replaced by Hill's spherical vortex,
	control of the scalar vorticity is replaced by control of the full velocity
	gradient, and viscosity introduces an additional error at every stage of
	the iteration.  The related SQG work is a separate adaptation of the same
	general strategy and is not used in the proof of the present paper.
	
	\subsection{Moving Hill vortices and the critical scaling}
	
	Our construction uses a dynamical cancellation.  Hill's vortex is an exact
	traveling Euler flow, so transport and self-interaction cancel before
	antidivergence.  After localization, motion along a rational orbit converts
	the remaining source into a directional stress after time averaging.  The
	Hill scaling preserves the \(L^2\)-norm and the \(L^{6/5}\)-gradient norm,
	and disjoint time scheduling removes interactions between principal
	directions.
	
	Figures~\ref{fig:intro-time-partition}--\ref{fig:intro-orbit-averaging}
	illustrate this three-dimensional version of the mechanism in
	\cite{BrueColomboKumar2024}.
	
	\begin{figure}[htbp]
		\centering
		\begin{tikzpicture}[
			x=1.42cm,
			y=1cm,
			font=\small,
			direction/.style={-{Stealth[length=2.1mm,width=1.45mm]},
				line width=.75pt,HillBlue},
			cutoff/.style={draw=HillTeal,line width=.55pt,
				fill=HillTeal!18,rounded corners=1.2pt}
			]
			\foreach \i/\ang in {1/12,2/78,3/150,4/35,5/105,6/205,
				7/-24,8/58,9/137}{
				\pgfmathsetmacro{\xl}{\i-1}
				\pgfmathsetmacro{\xm}{\i-.5}
				\fill[HillBlue!7] (\xl,0) rectangle (\i,.58);
				\draw[HillBlue!60,line width=.45pt] (\xl,0) rectangle (\i,.58);
				\node[font=\scriptsize] at (\xm,.29)
				{$\mathcal T^{k}_{\i}$};
				\node[font=\scriptsize] at (\xm,1.48) {$\xi_{\i}$};
				\draw[direction]
				(\xm,1.08) ++(\ang:-.23) -- ++(\ang:.46);
				\path[cutoff]
				(\xl+.05,-.28) -- (\xl+.14,-.28) --
				(\xl+.23,-.08) -- (\xl+.77,-.08) --
				(\xl+.86,-.28) -- (\xl+.95,-.28) -- cycle;
			}
			\draw[-{Stealth[length=2.2mm]},line width=.7pt]
			(-.18,.72) -- (9.24,.72) node[right] {$t$};
			\draw[line width=.8pt] (0,.66)--(0,.80);
			\draw[line width=.8pt] (9,.66)--(9,.80);
			\node[above=2pt,font=\scriptsize] at (0,.78)
			{$k\tau_{q+1}$};
			\node[above=2pt,font=\scriptsize] at (9,.78)
			{$(k+1)\tau_{q+1}$};
			\draw[decorate,decoration={brace,mirror,amplitude=4.5pt},
			line width=.55pt]
			(0,-.48) -- (9,-.48)
			node[midway,below=6pt] {$\mathcal T^k$ (length $\tau_{q+1}$)};
			\draw[HillTeal!80,-{Stealth[length=1.7mm]},line width=.5pt]
			(1.55,-.78) -- (1.48,-.19);
			\node[anchor=east,font=\scriptsize,align=right,HillTeal!70!black]
			at (1.55,-.80)
			{$\zeta_i^k=1$ on the shorter\\
				interval $\bar{\mathcal T}_i^k$};
			\node[anchor=west,font=\scriptsize,align=left]
			at (6.35,-.82)
			{at most one active Hill block\\at each time};
			\draw[-{Stealth[length=1.7mm]},line width=.5pt]
			(6.30,-.76) -- (6.05,.17);
		\end{tikzpicture}
		\caption{One time cell is divided into nine subintervals with disjoint
			cutoffs.  Thus at most one block, associated with \(\xi_i\), is active
			at a given time.  The arrows are schematic projections of the nine
			directions in \(\T^3\).}
		\label{fig:intro-time-partition}
	\end{figure}
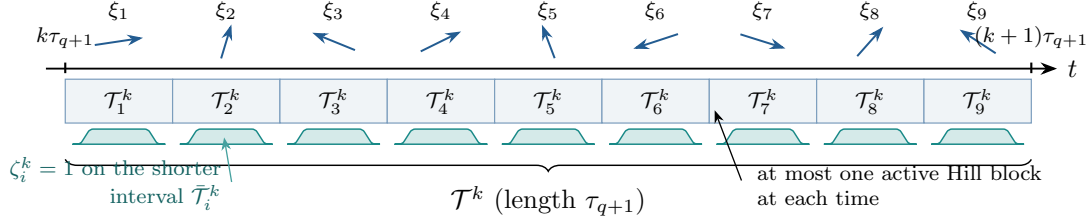

	\paragraph{The Hill profile.}
	
	Hill's spherical vortex~\cite{Hill1894,Lamb1932} is an axisymmetric,
	no-swirl traveling Euler flow whose azimuthal vorticity is supported in a
	ball.  Its orbital and linear stability under axisymmetric perturbations
	were studied in \cite{Choi2024,ProtasElcrat2016}; these results are not used
	below.

	After rotation and normalization, the profile \(H_e\) satisfies
	\begin{equation*}
		(e\cdot\nabla)H_e
		+\diver(H_e\otimes H_e)+\nabla P_e=0,
		\qquad
		\diver H_e=0.
	\end{equation*}
	Its exterior potential-doublet form permits localization without changing
	the core; see Figure~\ref{fig:intro-hill-geometry}.
	
	\begin{figure}[htbp]
		\centering
		\begin{tikzpicture}[
			font=\small,
			panel/.style={draw=black!45,rounded corners=1.5pt,
				fill=black!1,line width=.45pt},
			flowright/.style={draw=HillBlue!82!black,line width=.62pt,
				postaction={decorate},
				decoration={markings,mark=at position .58 with
					{\arrow{Stealth[length=1.7mm,width=1.15mm]}}}},
			flowleft/.style={draw=HillBlue!82!black,line width=.62pt,
				postaction={decorate},
				decoration={markings,mark=at position .58 with
					{\arrowreversed{Stealth[length=1.7mm,width=1.15mm]}}}},
			coreflow/.style={draw=HillOrange!88!black,line width=.7pt,
				postaction={decorate},
				decoration={markings,mark=at position .25 with
					{\arrow{Stealth[length=1.7mm,width=1.15mm]}}}}
			]
			\begin{scope}[xshift=-3.25cm,x=1.05cm,y=1.05cm]
				\path[panel] (-2.35,-2.05) rectangle (2.35,2.05);
				\draw[black!25,densely dashed] (0,-1.92)--(0,1.92);
				\node[anchor=south west,font=\scriptsize,black!65]
				at (.05,1.67) {symmetry axis};
				\foreach \A in {.78,1.12,1.48,1.88,2.22}{
					\draw[flowright]
					plot[domain=5:175,samples=90,variable=\t]
					({\A*sin(\t)^3},{\A*sin(\t)^2*cos(\t)});
					\draw[flowleft]
					plot[domain=5:175,samples=90,variable=\t]
					({-\A*sin(\t)^3},{\A*sin(\t)^2*cos(\t)});
				}
				\fill[white] (0,0) circle (3.2pt);
				\fill[black] (0,0) circle (1.5pt);
				\draw[-{Stealth[length=1.8mm]},black!65,line width=.45pt]
				(-1.2,-1.55) -- (-.12,-.10);
				\node[anchor=east,font=\scriptsize,align=right]
				at (-1.22,-1.57) {point singularity};
				\node[font=\scriptsize,fill=white,inner sep=1.5pt]
				at (1.15,1.48) {$H=\nabla\!\left(\dfrac{z}{2r^3}\right)$};
				\node at (0,-2.34) {\textup{(a)} Exterior potential doublet};
			\end{scope}
			\begin{scope}[xshift=3.25cm,x=1.05cm,y=1.05cm]
				\path[panel] (-2.35,-2.05) rectangle (2.35,2.05);
				\draw[black!25,densely dashed] (0,-1.92)--(0,1.92);
				\foreach \A in {.78,1.12,1.48,1.88,2.22}{
					\draw[flowright]
					plot[domain=5:175,samples=90,variable=\t]
					({\A*sin(\t)^3},{\A*sin(\t)^2*cos(\t)});
					\draw[flowleft]
					plot[domain=5:175,samples=90,variable=\t]
					({-\A*sin(\t)^3},{\A*sin(\t)^2*cos(\t)});
				}
				\fill[HillCore!13] (0,0) circle (1);
				\draw[HillOrange!85!black,line width=.9pt] (0,0) circle (1);
				\node[HillOrange!80!black,font=\Large] at (-.50,0) {$\odot$};
				\node[HillOrange!80!black,font=\Large] at ( .50,0) {$\otimes$};
				\draw[coreflow]
				(-.06,.76) .. controls (-.56,.70) and (-.82,.38) .. (-.82,.03)
				.. controls (-.82,-.35) and (-.50,-.67) .. (-.08,-.75);
				\draw[coreflow]
				(.06,-.76) .. controls (.56,-.70) and (.82,-.38) .. (.82,-.03)
				.. controls (.82,.35) and (.50,.67) .. (.08,.75);
				\draw[-{Stealth[length=2.2mm,width=1.5mm]},
				HillOrange!88!black,line width=1.05pt]
				(1.55,-.55)--(1.55,.85)
				node[above,font=\scriptsize] {translation $-e$};
				\node[font=\scriptsize,fill=HillCore!6,inner sep=1.2pt]
				at (0,-.42) {vorticity core};
				\node[font=\scriptsize,anchor=west]
				at (.74,-1.02) {$r=1$};
				\node at (0,-2.34) {\textup{(b)} Hill's spherical vortex};
			\end{scope}
			\node[align=center,font=\scriptsize,HillBlue!75!black]
			at (0,2.52) {the exterior fields coincide};
			\draw[{Stealth[length=1.5mm]}-{Stealth[length=1.5mm]},
			HillBlue!65,line width=.55pt]
			(-.58,2.20)--(.58,2.20);
			\node[font=\scriptsize,align=center] at (0,-2.88)
			{$\odot$ / $\otimes$: azimuthal vorticity pointing out of / into
				the meridional plane};
		\end{tikzpicture}
		\caption{Meridional schematic of the Hill building block.  The
			potential doublet in panel~\textup{(a)} is singular at the origin.
			Hill's spherical vortex in panel~\textup{(b)} replaces this singularity
			by a spherical vorticity core.  Outside the unit sphere, the velocity
			is unchanged and equals \(H=\nabla(z/(2r^3))\).  Rotation about the
			symmetry axis gives an azimuthal vorticity field supported in the core,
			and the vortex translates in the direction \(-e\).}
		\label{fig:intro-hill-geometry}
	\end{figure}
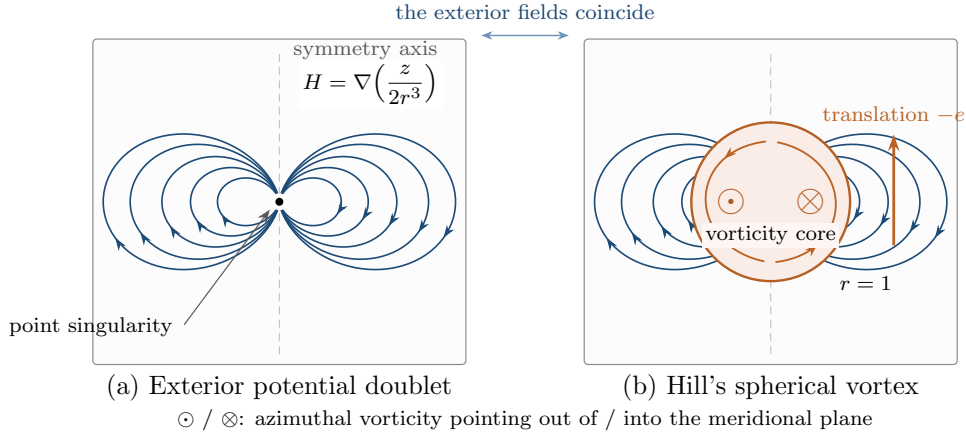

	For \(R>0\), define
	\begin{equation*}
		H_{R,e}(x)
		:=
		R^{-1}H_e\!\left(\frac{x}{R^{2/3}}\right),
		\qquad
		P_{R,e}(x)
		:=
		R^{-2}P_e\!\left(\frac{x}{R^{2/3}}\right).
	\end{equation*}
	The core radius is \(R^{2/3}\), the velocity amplitude is \(R^{-1}\), and
	\begin{equation*}
		\|H_{R,e}\|_{L^q}\lesssim_q R^{\frac2q-1},
		\qquad
		\|\nabla H_{R,e}\|_{L^p}\lesssim_p R^{\frac2p-\frac53}.
	\end{equation*}
	Here \(1<q\le\infty\) and \(1\le p\le\infty\).
	Hence the energy and the \(L^{6/5}\)-gradient norm are invariant, in
	agreement with \(W^{1,6/5}(\T^3)\hookrightarrow L^2(\T^3)\).  For
	\(p>6/5\), however,
	\[
	\|\nabla H_{R,e}\|_{L^p}
	=
	R^{-\theta(p)}\|\nabla H_e\|_{L^p},
	\qquad
	\theta(p):=\frac53-\frac2p>0,
	\]
	so concentration creates a loss that must be offset by the amplitude and
	averaging parameters.
	
	This Hill threshold differs from the Navier--Stokes scaling threshold.
	Under
	\[
	u(x,t)\mapsto u_\lambda(x,t)
	:=\lambda u(\lambda x,\lambda^2t),
	\]
	one has
	\[
	\|\nabla u_\lambda\|_{L^\infty(\lambda^{-2}I;L_x^p)}
	=\lambda^{2-3/p}
	\|\nabla u\|_{L^\infty(I;L_x^p)}.
	\]
	Thus the scale-invariant exponent is \(3/2\).  Our theorem crosses the
	Hill threshold \(6/5\) but remains supercritical for Navier--Stokes.
	
	On \(\T^3\), cutoff and mollification give a smooth compactly supported
	profile \(\widetilde H_{R,e}\).  For
	\[
	V^p(x,t)
	=
	A(t)\widetilde H_{R(t),e}(x-X(t)),
	\]
	we introduce the Helmholtz correction
	\[
	V^c=-\nabla\Delta^{-1}\diver V^p,
	\qquad
	V=V^p+V^c.
	\]
	Then \(V\) is divergence free.  Choosing the intrinsic trajectory
	\[
	X'(t)=-\frac{A(t)}{R(t)}e,
	\]
	cancels the leading transport and Euler terms.  With
	\[
	S_{R,e}:=\frac1R\widetilde H_{R,e},
	\qquad
	\int_{\T^3}S_{R,e}(x)\,\dd x=-2\pi e.
	\]
	The block then satisfies schematically
	\begin{equation}\label{eq:intro-block-identity}
		\partial_tV+\diver(V\otimes V)-\nu\Delta V+\nabla P
		=
		\frac{\dd}{\dd t}\bigl(A(t)R(t)\bigr)
		S_{R(t),e}(x-X(t))
		+\diver F,
	\end{equation}
	where \(F\) contains the remaining block errors.

	The viscous contribution is written as the symmetric stress
	\[
	-\nu\Delta V
	=
	\diver\!\left[-\nu\bigl(\nabla V+(\nabla V)^T\bigr)\right].
	\]
	It is estimated at an auxiliary exponent \(1<p_0<6/5\), where
	\(2/p_0-5/3>0\); the limiting gradient is instead measured at
	\(\bar p>6/5\).
	
	\subsection{Strategy of the construction}
	
	A smooth Reynolds flow encodes the endpoint data, and an iteration removes
	its defect with localized Hill vortices moving along rational trajectories.
	At each stage, the stress is frozen on short time cells and decomposed into
	finitely many rank-one tensors.  A moving Hill block treats each tensor,
	while a temporal corrector converts the resulting time-averaged cancellation
	into a pointwise identity.

	\subsubsection*{The zeroth Reynolds flow and the iteration}

	The iteration is formulated for the Navier--Stokes--Reynolds system
	\begin{equation*}
		\partial_tu_q+\diver(u_q\otimes u_q)
		-\nu\Delta u_q+\nabla p_q
		=
		\diver\E_q,
		\qquad
		\diver u_q=0.
	\end{equation*}
	After mollifying \(u_q\), we add the Hill blocks and a temporal corrector.
	The principal blocks cancel the frozen part of \(\E_q\); all remaining
	terms, including mollification, modulation, viscosity, and interactions
	with the coarse velocity, are placed in \(\E_{q+1}\).
	The estimates give \(\E_q\to0\) in \(L_t^\infty L_x^1\), summable
	\(C_tL_x^2\) increments, and a common
	\(C_t^{\alpha_0}L_x^{\bar p}\) gradient bound.  Hence
	\[
	u\in C_tL_x^2,
	\qquad
	\nabla u\in C_tL_x^{\bar p}.
	\]

	The endpoint fields enter only at the initial stage.  Given
	\(u^{(0)},u^{(1)}\in L^2_\sigma(\T^3)\), we mollify them at a common scale
	\(\ell\) and choose a smooth cutoff \(\chi\) that equals one near \(t=0\)
	and zero near \(t=1\).  We then set
	\begin{equation*}
		u_0(x,t)
		:=
		\chi(t)(u^{(0)}*\rho_\ell)(x)
		+(1-\chi(t))(u^{(1)}*\rho_\ell)(x).
	\end{equation*}
	Its defect is placed in a symmetric Reynolds stress using \(\mathcal R\).

	The blocks and corrector vanish at cell interfaces, including \(t=0,1\).
	Thus
	\[
		\|u_{q+1}(\cdot,j)-u_q(\cdot,j)\|_{L^2}
		\le C\lambda_{q+1}^{-\gamma}
		=C\lambda_q^{-1/5},
		\qquad j\in\{0,1\}.
	\]
	These errors are summable.  Suitable choices of \(\ell\) and
	\(\lambda_0\) then give the endpoint approximation.

	\subsubsection*{Stress decomposition and moving Hill blocks}

	The geometric lemma decomposes the mollified stress, up to pressure, into
	nine positive rank-one pieces:
	\[
		-\diver\E_\ell
		=\diver\left(\sum_{i=1}^9a_i(x,t)\xi_i\otimes\xi_i\right)
		+\nabla P^d.
	\]
	The positivity of the coefficients permits them to be encoded in the radius
	and residence time of a Hill vortex.  Each time cell is divided into nine
	disjoint subintervals, one for each direction, and the corresponding
	coefficient is frozen in time as \(a_i^k(x)\).  This scheduling also removes
	the leading quadratic interactions between different directions.

	Let
	\[
		r_{q+1}:=\lambda_{q+1}^{-\mu},
		\qquad \mu>0,
	\]
	be the base concentration parameter at stage \(q+1\); the complete
	parameter hierarchy is stated in \eqref{eq:intro-parameters} below.  For one
	coefficient \(a=a_i^k\) and direction \(\xi=\xi_i\), set
	\[
		R(x)=r_{q+1}a(x),
		\qquad
		A(t)=A_0\zeta(t),
		\qquad
		A_0^2=\frac{9}{2\pi}\int_{\T^3}a(x)\,\dd x.
	\]
	We orient the profile by \(e=-\xi\) and define
	\[
		V^p(x,t)
		=
		A(t)\widetilde H_{R(X(t)),-\xi}(x-X(t)).
	\]
	The center follows
	\begin{equation*}
		\dot X(t)=\frac{A(t)}{R(X(t))}\,\xi.
	\end{equation*}
	Thus \(R=r_{q+1}a\) couples the stress coefficient to the Hill scale, whose
	physical core radius is \(R^{2/3}\), and the trajectory preserves the exact
	Hill cancellation.  The normalization of \(A_0\) compensates for the fact
	that the direction is active only during one ninth of the cell.  Moreover,
	the speed is inversely proportional to \(a\): the vortex spends more time
	in regions where the coefficient is larger, which is the basis of the
	stress reconstruction below.

	\subsubsection*{Reconstruction of the stress by time averages}

	The source in \eqref{eq:intro-block-identity} is first replaced by a smooth
	auxiliary density of mass \(2\pi\).  The difference has zero spatial mean
	and can therefore be written as a divergence with an inverse-frequency
	gain.  Integration by parts along the trajectory converts the auxiliary
	source into a directional tensor flux.  Since the speed is proportional to
	\((r_{q+1}a)^{-1}\), the residence time recovers the coefficient \(a\),
	giving
	\begin{equation}\label{eq:intro-orbit-average}
		\frac1\tau\int_{I_i} U(x,t)\,\dd t
		=\diver\bigl(a(x)\xi\otimes\xi+G(x)\bigr),
		\qquad
		\|G\|_{L^1}
		\lesssim\frac{\|a\|_{C^1}}{\lambda}.
	\end{equation}
	
	A rational trajectory closes on the torus and can be repeated many times
	during one active interval.  The freezing error is
	\(O(\tau\|\partial_ta_i\|_{L^\infty})\), while the \(\lambda^{-1}\)
	factor controls incomplete periods and spatial variation; see
	Figure~\ref{fig:intro-orbit-averaging}.
	
	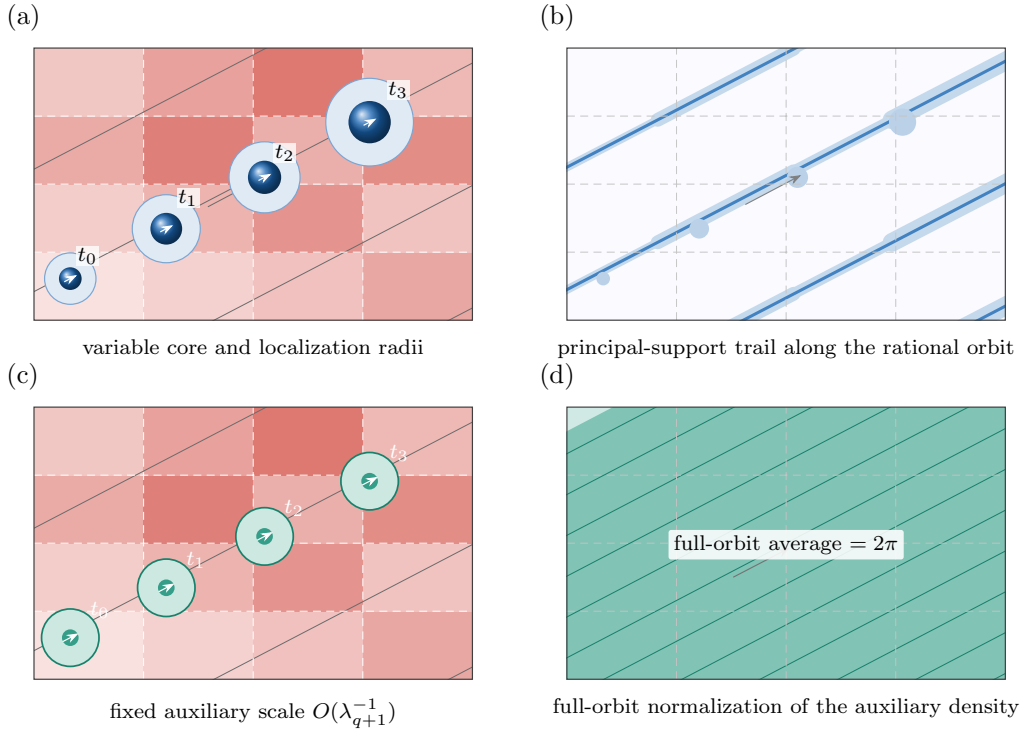
\begin{figure}[htbp]
		\centering
		\begin{tikzpicture}[
			x=1cm,y=1cm,
			orbit/.style={draw=black!58,line width=.34pt},
			flowarrow/.style={-{Stealth[length=1.6mm,width=1.15mm]}},
			panelbox/.style={draw=black!75,line width=.45pt},
			every node/.style={font=\footnotesize}
			]
			\definecolor{hillblue}{RGB}{25,102,180}
			\definecolor{auxgreen}{RGB}{26,145,120}
			\definecolor{warmred}{RGB}{205,55,45}
			
			\begin{scope}[shift={(0,4.75)}]
				\node[anchor=south west,font=\small] at (-.50,3.70) {\textup{(a)}};
				\fill[warmred!8] (0,0) rectangle (5.80,3.60);
				\foreach \x/\y/\c in {
					0/0/15,1.45/0/23,2.90/0/31,4.35/0/44,
					0/.90/25,1.45/.90/38,2.90/.90/54,4.35/.90/30,
					0/1.80/42,1.45/1.80/64,2.90/1.80/40,4.35/1.80/57,
					0/2.70/28,1.45/2.70/47,2.90/2.70/67,4.35/2.70/35}
				\fill[warmred!\c] (\x,\y) rectangle ++(1.45,.90);
				\begin{scope}
					\clip (0,0) rectangle (5.80,3.60);
					\foreach \s in {-2.94,-1.32,.30,1.92}
					\draw[orbit] (-.15,\s) -- (5.95,{\s+3.18});
					\draw[orbit,flowarrow] (2.30,1.50) -- (3.00,1.87);
					\draw[white!72,line width=.32pt,densely dashed]
					(1.45,0)--(1.45,3.60) (2.90,0)--(2.90,3.60)
					(4.35,0)--(4.35,3.60) (0,.90)--(5.80,.90)
					(0,1.80)--(5.80,1.80) (0,2.70)--(5.80,2.70);
					\foreach \x/\y/\ro/\ri in {
						.48/.55/.34/.15,
						1.75/1.21/.45/.21,
						3.05/1.89/.47/.22,
						4.44/2.62/.58/.28}{
						\fill[hillblue!14] (\x,\y) circle (\ro);
						\draw[hillblue!58,line width=.45pt] (\x,\y) circle (\ro);
						\shade[ball color=hillblue] (\x,\y) circle (\ri);
						\draw[white,flowarrow,line width=.45pt]
						(\x-.08,\y-.042) -- (\x+.09,\y+.047);
					}
					\node[anchor=south west,font=\scriptsize,
					fill=white,fill opacity=.78,text opacity=1,inner sep=.7pt]
					at (.55,.71) {$t_0$};
					\node[anchor=south west,font=\scriptsize,
					fill=white,fill opacity=.78,text opacity=1,inner sep=.7pt]
					at (1.88,1.48) {$t_1$};
					\node[anchor=south west,font=\scriptsize,
					fill=white,fill opacity=.78,text opacity=1,inner sep=.7pt]
					at (3.16,2.07) {$t_2$};
					\node[anchor=south west,font=\scriptsize,
					fill=white,fill opacity=.78,text opacity=1,inner sep=.7pt]
					at (4.66,2.91) {$t_3$};
				\end{scope}
				\draw[panelbox] (0,0) rectangle (5.80,3.60);
				\node[anchor=north,font=\scriptsize] at (2.90,-.12)
				{variable core and localization radii};
			\end{scope}
			
			\begin{scope}[shift={(7.05,4.75)}]
				\node[anchor=south west,font=\small] at (-.50,3.70) {\textup{(b)}};
				\fill[blue!2] (0,0) rectangle (5.80,3.60);
				\begin{scope}
					\clip (0,0) rectangle (5.80,3.60);
					\foreach \s in {-2.94,-1.32,.30,1.92}{
						\draw[hillblue!25,line width=3.2pt,line cap=round]
						(-.20,\s) -- (1.30,{\s+.78});
						\draw[hillblue!25,line width=5.2pt,line cap=round]
						(1.20,{\s+.73}) -- (3.00,{\s+1.67});
						\draw[hillblue!25,line width=4.2pt,line cap=round]
						(2.90,{\s+1.62}) -- (4.40,{\s+2.40});
						\draw[hillblue!25,line width=7.0pt,line cap=round]
						(4.30,{\s+2.35}) -- (5.98,{\s+3.22});
						\draw[hillblue!82,line width=1.15pt]
						(-.20,\s) -- (5.98,{\s+3.22});
					}
					\foreach \x/\y/\r in {
						.48/.55/.09,1.75/1.21/.13,3.05/1.89/.14,4.44/2.62/.18}
					\fill[hillblue!30] (\x,\y) circle (\r);
					\draw[black!50,flowarrow,line width=.42pt]
					(2.36,1.53) -- (3.10,1.92);
					\draw[gray!45,densely dashed,line width=.28pt]
					(1.45,0)--(1.45,3.60) (2.90,0)--(2.90,3.60)
					(4.35,0)--(4.35,3.60) (0,.90)--(5.80,.90)
					(0,1.80)--(5.80,1.80) (0,2.70)--(5.80,2.70);
				\end{scope}
				\draw[panelbox] (0,0) rectangle (5.80,3.60);
				\node[anchor=north,font=\scriptsize] at (2.90,-.12)
				{principal-support trail along the rational orbit};
			\end{scope}
			
			\begin{scope}[shift={(0,0)}]
				\node[anchor=south west,font=\small] at (-.50,3.70) {\textup{(c)}};
				\fill[warmred!8] (0,0) rectangle (5.80,3.60);
				\foreach \x/\y/\c in {
					0/0/15,1.45/0/23,2.90/0/31,4.35/0/44,
					0/.90/25,1.45/.90/38,2.90/.90/54,4.35/.90/30,
					0/1.80/42,1.45/1.80/64,2.90/1.80/40,4.35/1.80/57,
					0/2.70/28,1.45/2.70/47,2.90/2.70/67,4.35/2.70/35}
				\fill[warmred!\c] (\x,\y) rectangle ++(1.45,.90);
				\begin{scope}
					\clip (0,0) rectangle (5.80,3.60);
					\foreach \s in {-2.94,-1.32,.30,1.92}
					\draw[orbit] (-.15,\s) -- (5.95,{\s+3.18});
					\draw[white!72,line width=.32pt,densely dashed]
					(1.45,0)--(1.45,3.60) (2.90,0)--(2.90,3.60)
					(4.35,0)--(4.35,3.60) (0,.90)--(5.80,.90)
					(0,1.80)--(5.80,1.80) (0,2.70)--(5.80,2.70);
					\foreach \x/\y/\tt in {
						.48/.55/0,1.75/1.21/1,3.05/1.89/2,4.44/2.62/3}{
						\fill[auxgreen!22] (\x,\y) circle (.38);
						\draw[auxgreen!90!black,line width=.65pt] (\x,\y) circle (.38);
						\fill[auxgreen!88] (\x,\y) circle (.11);
						\draw[white,flowarrow,line width=.45pt]
						(\x-.10,\y-.052) -- (\x+.11,\y+.057);
						\node[white,anchor=south west,font=\scriptsize]
						at (\x+.13,\y+.13) {$t_{\tt}$};
					}
				\end{scope}
				\draw[panelbox] (0,0) rectangle (5.80,3.60);
				\node[anchor=north,font=\scriptsize] at (2.90,-.12)
				{fixed auxiliary scale \(O(\lambda_{q+1}^{-1})\)};
			\end{scope}
			
			\begin{scope}[shift={(7.05,0)}]
				\node[anchor=south west,font=\small] at (-.50,3.70) {\textup{(d)}};
				\fill[auxgreen!20] (0,0) rectangle (5.80,3.60);
				\begin{scope}
					\clip (0,0) rectangle (5.80,3.60);
					\foreach \s in {-3.05,-2.51,...,3.43}{
						\draw[auxgreen!55,line width=15pt,line cap=round]
						(-.18,\s) -- (5.98,{\s+3.22});
						\draw[auxgreen!88!black,line width=.35pt]
						(-.18,\s) -- (5.98,{\s+3.22});
					}
					\draw[gray!48,densely dashed,line width=.28pt]
					(1.45,0)--(1.45,3.60) (2.90,0)--(2.90,3.60)
					(4.35,0)--(4.35,3.60) (0,.90)--(5.80,.90)
					(0,1.80)--(5.80,1.80) (0,2.70)--(5.80,2.70);
					\draw[black!55,flowarrow,line width=.42pt]
					(2.20,1.35) -- (2.95,1.74);
					\node[fill=white,fill opacity=.88,text opacity=1,
					rounded corners=1pt,inner sep=2pt,font=\scriptsize]
					at (2.90,1.78) {full-orbit average \(=2\pi\)};
				\end{scope}
				\draw[panelbox] (0,0) rectangle (5.80,3.60);
				\node[anchor=north,font=\scriptsize] at (2.90,-.12)
				{full-orbit normalization of the auxiliary density};
			\end{scope}
		\end{tikzpicture}
		\caption{Schematic two-dimensional projection of the three-dimensional
			moving-core construction on \(\T^3\) (not to scale).  In
			panel~\textup{(a)}, the red background represents the frozen
			coefficient \(a_i^k\).  The dark inner balls show the vorticity core
			at scale \((R_i^k)^{2/3}\), while the faint outer balls show the
			support of the principal localized profile at scale
			\((R_i^k)^\alpha\), where \(R_i^k=r_{q+1}a_i^k\).
			Panel~\textup{(b)} shows the resulting variable-width support trail;
			the disjointness condition is
			\(2(R_i^k)^\alpha\leq\lambda_{q+1}^{-1}\).
			Panel~\textup{(c)} shows the auxiliary density
			\(\widetilde U_i^k(\,\cdot-x_i^k(t))\), whose support has the fixed
			scale \(O(\lambda_{q+1}^{-1})\), rather than another velocity vortex.
			Panel~\textup{(d)} depicts its normalized full-orbit average, which is
			used in the approximate stress cancellation
			\eqref{eq:intro-orbit-average}.}
		\label{fig:intro-orbit-averaging}
	\end{figure}
	
	\subsubsection*{The temporal corrector and the new stress}

	The orbit identity cancels the stress only after averaging over a time
	cell.  To restore the equation pointwise in time, sum the nine sources into
	\(U_{\rm tot}\) and define
	\begin{equation*}
		Q(x,t):=-\mathbb P\int_{k\tau}^{t}
		\left(U_{\rm tot}(x,s)
		-\frac1\tau\int_JU_{\rm tot}(x,z)\,\dd z\right)\dd s,
		\qquad t\in J,
	\end{equation*}
	where \(\mathbb P\) is the Leray projection.  Since the integrand has zero
	cell average, \(Q\) vanishes at both endpoints and gains the short factor
	\(\tau\).  Its time derivative cancels the zero-average source, while the
	cell average cancels the frozen stress through
	\eqref{eq:intro-orbit-average}.  The factor \(\tau\) makes \(Q\) smaller
	than the principal blocks in both the energy and gradient estimates.
	
	The new stress contains the coarse-flow, corrector, freezing,
	source-replacement, localization, viscous, and intrinsic block errors.
	They are estimated separately because they use different features of the
	construction: concentration, inverse frequency, or the short time-cell
	factor.  Their bounds make the stress decay and the increments summable,
	while endpoint vanishing preserves the prescribed traces.
	
	\subsection{Crossing the Hill threshold and analytical difficulties}

	\subsubsection*{Competing effects of concentration}

	The main difficulty is that concentration has opposite effects on the
	estimates.  For an energy-sized Hill block,
	\[
		\|\nabla H_{R,e}\|_{L^p}\lesssim R^{\frac2p-\frac53}.
	\]
	Thus small \(R\) improves the intrinsic and viscous errors below \(6/5\),
	but increases the gradient norm above \(6/5\).  We use \(p_0<6/5\) for
	the block errors, \(p_s>1\) close to one for source replacement, and
	\(\bar p>6/5\) for the limiting gradient.  Short time cells reduce the
	freezing and corrector errors, but must still contain enough complete
	rational orbits to average the stress.

	The main parameters are
	\begin{equation}\label{eq:intro-parameters}
		\lambda_{q+1}=\lambda_q^\sigma,
		\qquad
		\delta_q=\lambda_1^{2\gamma}\lambda_q^{-\gamma},
		\qquad
		r_{q+1}=\lambda_{q+1}^{-\mu},
		\qquad
		\tau_{q+1}=\lambda_{q+1}^{-\eta}.
	\end{equation}
	Here \(\delta_q\) is the stress level, \(r_q\) the concentration
	parameter, and \(\tau_q\) the cell length.  The rapid growth of
	\(\lambda_q\) separates consecutive stages, whereas the small exponent
	\(\gamma\) makes the stress decrease slowly enough to absorb the loss above
	the Hill threshold.

	We take \(p_0=28/27\), for which
	\(2/p_0-5/3=11/42>0\).  Together with the intrinsic block choices
	\(\alpha=2/11\) and \(\beta=502/231\), this gives a positive power of the
	Hill scale in the intrinsic and viscous estimates.  The source-replacement
	error is treated at \(p_s>1\) close to one.  The difference between the
	Hill source and the auxiliary source has zero mean, so a localized
	antidivergence produces an inverse-frequency gain; choosing \(p_s\) near
	one keeps the associated concentration loss small.

	\subsubsection*{The balance above \(6/5\)}

	For the final exponent, set
	\[
		\kappa_\ast:=\frac53-\frac2{\bar p}=\frac5{72003}>0.
	\]
	A principal block has amplitude \(\delta_{q+1}^{1/2}\), while its smallest
	Hill parameter is \(r_{q+1}\delta_{q+1}\) and its smallest physical core
	radius is \((r_{q+1}\delta_{q+1})^{2/3}\).  Therefore
	\[
		\|\nabla V_{q+1}\|_{L_t^\infty L_x^{\bar p}}
		\lesssim
		\delta_{q+1}^{1/2}
		(r_{q+1}\delta_{q+1})^{-\kappa_\ast}.
	\]
	Its summability requires
	\[
		-\frac\gamma2+\kappa_\ast(\mu+\gamma)<0.
	\]
	The factor \(\delta_{q+1}^{1/2}\) must compensate for the concentration
	loss.  This explains why increasing either \(\bar p\) or \(\mu\) makes the
	argument harder.

	There is also a restriction from time averaging.  Short cells improve the
	freezing error and the temporal primitive, but a cell must be long enough
	for the vortex to complete the required orbit periods.  The orbit and
	temporal-corrector estimates impose the following sufficient conditions,
	respectively,
	\[
		3-\mu+\frac\gamma2<-\eta,
		\qquad
		-\eta+4-\frac3{\bar p}+\frac{4n}{\sigma}+4\gamma<0.
	\]
	These restrictions are compatible for
	\[
	\sigma=200,
	\qquad n=20,
	\qquad \gamma=10^{-3},
	\qquad \mu=6+2\gamma=6.002,
	\qquad \eta=3,
	\]
	and \(\bar p=6/5+5\times10^{-5}\).  This small gain above \(6/5\)
	reflects the limited margin between the concentration loss and the orbit
	and corrector constraints.  Temporal concentration cannot replace this
	balance: it improves time-integrated norms but does not reduce the required
	\(L_t^\infty L_x^{\bar p}\) norm.

	\subsubsection*{Localization, background interaction, and time continuity}

	Several further errors must be controlled.  Localization and mollification
	create commutators, while the Helmholtz correction restores
	incompressibility after cutting off the exterior potential.  The intrinsic
	Hill motion also produces an interaction with the coarse velocity.  For a
	principal increment \(v_{q+1}\), concentration gives schematically
	\[
		\|v_{q+1}\otimes u_\ell+u_\ell\otimes v_{q+1}\|_{L_t^\infty L_x^1}
		\lesssim
		\delta_{q+1}^{1/2}r_{q+1}^{1/3}
		\lambda_{q+1}^{\frac{2n}{\sigma}+\gamma}.
	\]
	This must be smaller than the next stress level.  The disjoint time
	scheduling removes the leading interactions between different principal
	directions, while orbit averaging, freezing, and the short temporal
	primitive control the remaining terms.

	Finally, for some \(\alpha_0,c_0>0\), the iteration provides estimates of
	the form
	\[
		\|u_{q+1}-u_q\|_{C_tL_x^2}
		\lesssim\delta_{q+1}^{1/2},
		\qquad
		\|\nabla u_{q+1}\|_{C_t^{\alpha_0}L_x^{\bar p}}
		\leq
		\|\nabla u_q\|_{C_t^{\alpha_0}L_x^{\bar p}}
		+C\delta_{q+1}^{c_0}.
	\]
	The first estimate gives strong convergence in \(C_tL_x^2\); the second
	provides a common time modulus for the gradients.  Together they yield
	\(u\in C_tL_x^2\) and \(\nabla u\in C_tL_x^{\bar p}\).  A uniform
	fixed-time gradient bound alone would not give this time continuity.

	\subsection{Organization of the paper}
	
	The remainder of the paper is organized as follows.  Section~2 recalls the
	normalized Hill vortex and establishes estimates for its rescaled,
	localized, and mollified forms.  Section~3 introduces the moving Hill block
	and derives the associated Euler and Navier--Stokes defect identities.
	Section~4 states the iteration step and constructs the next Reynolds flow,
	including the decomposition of the Reynolds stress, the time partition,
	the space-dependent scales, the vortex trajectories, the auxiliary source,
	and the temporal corrector.  Section~5 proves the estimates required for
	the iteration and verifies the parameter restrictions.  Section~6
	constructs the initial Reynolds flow, passes to the limit, and proves
	Theorems~\ref{thm:main} and~\ref{thm:nonuniqueness}.  Section~7 discusses
	the limitations of the method and possible modifications.  Section~8
	presents further directions and open problems.

	\section{Hill's spherical vortex}
	
	Hill's spherical vortex is an axisymmetric, no-swirl solution of the
	three-dimensional Euler equations whose vorticity is supported in a
	ball and which translates without changing shape.
	It was introduced by Hill~\cite{Hill1894}; see also
	Lamb~\cite[\S165]{Lamb1932} for the classical formulas and
	Choi~\cite{Choi2024} for a modern treatment of its orbital stability in
	the axisymmetric no-swirl class.
	Only the explicit traveling-wave structure is used below.

	We record the profile carefully because three features enter the
	convex-integration scheme with their precise normalization: the
	translation speed, the directional integral of the localized velocity,
	and the potential character of the exterior field.  We first write the classical
	vortex in a frame in which the spherical core is stationary, then remove
	the constant velocity at infinity to obtain a decaying traveling profile.
	The following subsection rescales and rotates this unit profile and
	truncates its exterior potential without changing the vorticity core.
	Here \(\xi_i\) denotes a rank-one stress direction.  Also,
	\(Q_e\in SO(3)\) denotes a rotation, whereas \(Q_{q+1}\) is the temporal
	corrector introduced in Section~4.
	
	Let $(r,\theta,\varphi)$ be spherical coordinates with the axis of
	propagation along the \(z\)-axis.  In this preliminary description,
	\(r=|x|\) denotes the spherical radius; it should not be confused with
	the concentration parameters introduced later.  Thus
	\[
	z=r\cos\theta,
	\qquad
	s=r\sin\theta,
	\]
	where $s$ is the cylindrical radius.  The flow is axisymmetric and has
	no swirl; hence
	\[
	u_\varphi=0.
	\]
	
	An axisymmetric no-swirl incompressible velocity is represented by its
	Stokes stream function \(\psi(r,\theta)\), with
	\[
	u_r=\frac{1}{r^2\sin\theta}\,\partial_\theta \psi,
	\qquad
	u_\theta=-\frac{1}{r\sin\theta}\,\partial_r \psi.
	\]
	
	For a Hill vortex of radius \(a\) and translation speed \(U\), we use
	the sign convention
	\[
	\psi(r,\theta)=
	\begin{cases}
		-\dfrac{3U}{4}\left(1-\dfrac{r^2}{a^2}\right)r^2\sin^2\theta,
		& r\le a,\\[1.2ex]
		\dfrac{U}{2}\left(1-\dfrac{a^3}{r^3}\right)r^2\sin^2\theta,
		& r\ge a.
	\end{cases}
	\]
	These formulas agree, up to the choice of propagation direction, with the
	classical formulas in~\cite{Hill1894}; see also
	\cite[\S165]{Lamb1932} for the normalization used here.
	
	The corresponding velocity components are
	\[
	u_r=
	\begin{cases}
		-\dfrac{3U}{2}\left(1-\dfrac{r^2}{a^2}\right)\cos\theta,
		& r\le a,\\[1.2ex]
		U\left(1-\dfrac{a^3}{r^3}\right)\cos\theta,
		& r\ge a,
	\end{cases}
	\]
	\[
	u_\theta=
	\begin{cases}
		\dfrac{3U}{2}\left(1-\dfrac{2r^2}{a^2}\right)\sin\theta,
		& r\le a,\\[1.2ex]
		-\,U\left(1+\dfrac{a^3}{2r^3}\right)\sin\theta,
		& r\ge a,
	\end{cases}
	\]
	\[
	u_\varphi=0.
	\]
	
	The vorticity is purely azimuthal:
	\[
	\omega=\omega_\varphi\,e_\varphi,
	\]
	with
	\[
	\omega_\varphi=
	\begin{cases}
		-\dfrac{15U}{2a^2}\,r\sin\theta,
		& r\le a,\\[1.2ex]
		0,
		& r\ge a.
	\end{cases}
	\]
	In particular, in the normalization used here, the relative vorticity is
	\[
		\frac{\omega_\varphi}{r\sin\theta}
		=-\frac{15U}{2a^2}\mathbf 1_{B_a}.
	\]
	Thus the vorticity is compactly supported even though the
	induced velocity has a noncompact potential tail.
	
	These formulas describe a stationary field in the co-moving frame.  In
	that frame the velocity is an exact steady solution of the
	incompressible Euler equations,
	\[
	(u\cdot\nabla)u+\nabla p=0,
	\qquad
	\nabla\cdot u=0.
	\]
	
	For \(U=a=1\), the co-moving velocity \(u\) tends to \(e_3\) at
	infinity.  Subtracting this constant far-field velocity gives the
	decaying laboratory-frame profile
	\[
		H:=u-e_3.
	\]
	Since \(u=H+e_3\), the steady Euler equation becomes, in the sense of
	distributions,
	\begin{equation*}
		\,\partial_3 H+(H\cdot\nabla) H+\nabla P_H=0,
		\qquad
		\operatorname{div}H=0.
	\end{equation*}
	Thus $H(x+te_3)$ travels in the direction $-e_3$.

	In spherical coordinates, let \(\hat r\), \(\hat\theta\),
		and \(\hat\varphi\) denote the radial, polar, and azimuthal unit vectors,
		respectively.  Since \(H\) has no azimuthal component, it is determined
		by
	\begin{equation*}
		H\cdot \hat r = 
		\begin{cases}
			-\dfrac{3}{2}\left(\dfrac{5}{3}-r^2\right)\cos\theta,
			& r\le 1,\\[1.2ex]
			-\dfrac{1}{r^3}\cos\theta,
			& r\ge 1,
		\end{cases}
	\end{equation*}
	\[
	H\cdot \hat \theta = \begin{cases}
		\dfrac{3}{2}\left(\dfrac{5}{3}-2r^2\right)\sin\theta,
		& r\le 1,\\[1.2ex]
		-\dfrac{1}{2r^3}\sin\theta,
		& r\ge 1.
	\end{cases}
	\]
	
	For $r\ge1$, the profile $H$ coincides with the
		three-dimensional potential doublet
		$H=\nabla\Phi$, where $\Phi(x)=z/(2r^3)$.
	\begin{align*}
		H &= -\frac{1}{r^3}\cos\theta \begin{pmatrix}
			\sin \theta\cos \varphi\\ \sin \theta\sin \varphi\\ \cos \theta
		\end{pmatrix} -\frac{1}{2r^3}\sin\theta\begin{pmatrix}
			\cos \theta\cos \varphi\\ \cos \theta\sin \varphi\\ -\sin \theta
		\end{pmatrix}\nonumber\\
		&= -\frac{1}{r^5}\begin{pmatrix}
			\frac{3}{2}xz\\ \frac{3}{2}yz\\ z^2-\frac{1}{2}(x^2+y^2)
		\end{pmatrix}\nonumber\\
		&=\nabla\!\left(\frac{z}{2r^3}\right)=\nabla\Phi,
		\qquad r\ge1. 
	\end{align*}
	\begin{remark}
		In Cartesian coordinates, \(H\) is given by
		\begin{equation*}
			H = \begin{cases}
				-\begin{pmatrix}
					\frac{3}{2}xz\\\frac{3}{2}yz\\-3(x^2+y^2)-\frac{3}{2}z^2+\frac{5}{2}
				\end{pmatrix}& r\le 1,\\
				-\frac{1}{r^5}\begin{pmatrix}
					\frac{3}{2}xz\\ \frac{3}{2}yz\\ z^2-\frac{1}{2}(x^2+y^2)
				\end{pmatrix}& r\ge 1.
			\end{cases}
		\end{equation*}
		Thus \(H\in C^{0,1}(\R^3)\) and is smooth on
			\(\R^3\setminus\partial B_1(0)\).
	\end{remark}

	We shall use four consequences of these formulas.  The field \(H\) is a
	finite-energy divergence-free traveling Euler profile; its vorticity is
	confined to \(B_1\); its exterior velocity is the gradient
	\(\nabla\Phi\); and its only loss of smoothness occurs across
	\(\partial B_1\).  The potential representation permits us to remove the
	long-range tail by a cutoff.  Since the profile is only Lipschitz across
	the spherical interface, we subsequently mollify it to obtain the required
	higher-derivative estimates.
	
	\subsection{Scaled Hill vortex profile}
	
	We first introduce the spatial and amplitude scalings used in the
	construction.  The spatial scale of the vorticity core is \(R^{2/3}\),
	whereas the amplitude is \(R^{-1}\).  This preserves the \(L^2\)-size of
	the profile and changes its traveling speed from one to \(R^{-1}\).
	We then cut off the irrotational exterior field at the larger scale
	\(R^\alpha\).  Since \(\alpha<2/3\), this localization does not alter the
	vorticity core.
	
	For \(R>0\), define the rescaled Hill profile by
	\[
	H_R(x):=\frac1R\,H\!\left(\frac{x}{R^{\frac{2}{3}}}\right),
	\qquad
	P_R(x):=\frac1{R^2}\,P_H\!\left(\frac{x}{R^\frac{2}{3}}\right).
	\]
	Then
	\begin{equation}\label{scaled eulerprofile}
		\frac{1}{R}\,\partial_3 H_R+\operatorname{div}(H_R\otimes H_R)+\nabla P_R=0,
		\qquad
		\operatorname{div}H_R=0.
	\end{equation}
	
	More generally, for \(e\in\mathbb S^2\), choose
		\(Q_e\in SO(3)\) with \(Q_ee_3=e\), and define
	\[
	H_{R,e}(x):=Q_e H_R(Q_e^T x),
	\qquad
	P_{R,e}(x):=P_R(Q_e^T x).
	\]
	The definition is independent of the particular choice of \(Q_e\),
	because the Hill profile is axisymmetric about \(e_3\).  The field
	\(H_{R,e}\) is oriented along \(e\), travels intrinsically in the
	direction \(-e\), and satisfies
	\begin{equation*}
		\frac{1}{R}(e\cdot\nabla)H_{R,e}
		+\operatorname{div}(H_{R,e}\otimes H_{R,e})
		+\nabla P_{R,e}=0,
		\qquad
		\operatorname{div}H_{R,e}=0.
	\end{equation*}
	Equivalently,
	\[
		(x,t)\longmapsto
		H_{R,e}\!\left(x+\frac{t}{R}e\right)
	\]
	is a traveling Euler solution whose center moves with velocity
	\(-R^{-1}e\).
	
	Fix $0<R\le R_0<1$ and $0<\alpha<2/3$.  Let
		$\chi\in C^\infty(\R^3)$ be radial, with $\chi=0$ on $B_1$,
		$\chi=1$ on $\R^3\setminus B_2$, and $0\le\chi\le1$.  Set
	\[
	\chi_\alpha(x):=\chi\!\left(\frac{x}{R^\alpha}\right).
	\]
	Since \(R^{2/3}<R^\alpha\), the transition annulus
	\[
	A_R:=B_{2R^\alpha}(0)\setminus \overline{B}_{R^\alpha}(0)
	\]
	lies strictly outside the spherical vorticity core.  For
	\(|x|\ge R^{2/3}\), the rescaled profile is a gradient.  Indeed, the
	\((-3)\)-homogeneity of \(\nabla\Phi\) gives
	\[
	H_R = \frac1R\,H\!\left(\frac{x}{R^{\frac{2}{3}}}\right)=\frac1R\,\nabla\Phi\!\left(\frac{x}{R^{\frac{2}{3}}}\right)
	=R\nabla\Phi(x).
	\]
	
	The scalar potential and the localized Hill profile are therefore
	defined by
	\begin{equation}
		\label{truncation}
		\Pi_R:=R\chi_\alpha\Phi,
		\qquad
		\bar H_R:=H_R-\nabla\Pi_R
		=H_R-R\nabla(\chi_\alpha\Phi).
	\end{equation}
	Although \(\Phi\) is singular at the origin, \(\Pi_R\) is smooth because
	\(\chi_\alpha\) vanishes in a neighborhood of the origin.
	Thus \(\bar H_R=H_R\) on \(B_{R^\alpha}\), while
	\(\bar H_R=0\) on \(\R^3\setminus B_{2R^\alpha}\).  In particular, the
	gradient correction preserves the vorticity exactly:
	\[
		\nabla\times\bar H_R=\nabla\times H_R.
	\]
	It introduces only a divergence error in the annulus \(A_R\).  Since
	\(\Delta\Phi=0\) on \(\R^3\setminus\{0\}\),
	\begin{equation*}
		\diver \bar H_R
		=-R\Delta(\chi_\alpha\Phi)
		=-R(\Delta\chi_\alpha)\Phi
		-2R\nabla\chi_\alpha\cdot\nabla\Phi,
		\qquad \supp\diver\bar H_R\subseteq \overline A_R.
	\end{equation*}
	\begin{remark}
		As shown in Lemma~\ref{lemm 1}, the scaling and radial localization
		give
		\[
			\frac1R\int_{\R^3}\bar H_R(x)\,\dd x=-2\pi e_3,
		\]
		independently of \(R\).  The same scaling produces the factor
		\(R^{-1}\) in the traveling equation and in the trajectory equation
		of Proposition~\ref{main prop}.
	\end{remark}
	
	\subsection{Estimates for the building block}
	
	We collect the estimates for the localized profile.  The first lemma
	records its support, Lebesgue norms, derivative bounds, and directional
	integral.  We then derive the approximate traveling equation and the
	identity obtained by differentiating in \(R\), and finally mollify the
	spherical interface.
	
	\begin{lemma}\label{lemm 1}
		The fields $\bar H_R$ and $\Pi_R$, defined in \eqref{truncation},
		satisfy the following properties, with constants independent of
		\(R\in(0,R_0]\).  The constants may depend on \(\alpha\), \(R_0\),
		and the cutoff \(\chi\), and also on the displayed Lebesgue exponent.
		\begin{enumerate}
			\item The localization identities are
				\[
				\bar H_R=H_R,\quad \Pi_R=0
				\quad\text{on }B_{R^\alpha}(0),
				\qquad
				\bar H_R=0,\quad \nabla\Pi_R=H_R
				\quad\text{on }\R^3\setminus B_{2R^\alpha}(0).
				\]
				Consequently,
				\(\supp\bar H_R\subseteq\overline B_{2R^\alpha}(0)\) and
				\(\supp\diver\bar H_R\subseteq\overline A_R\).
			\item $\|\bar H_R\|_{L^1}\leq CR(1+|\log R|)$ and
				$\|\bar H_R\|_{L^p}\leq CR^{\frac{2}{p}-1}$ for every
				$1<p\le \infty$.
			\item For every \(1\le p\le\infty\),
				$\|\nabla\bar H_R\|_{L^p}\leq CR^{\frac{2}{p}-\frac{5}{3}}$,
				\[
				\|\nabla^2\bar H_R\|_{L^p(\R^3\setminus
				\partial B_{R^{2/3}})}
				\leq CR^{\frac{2}{p}-\frac{7}{3}},
				\qquad
				\|\diver\bar H_R\|_{L^p}
				\leq CR^{1-4\alpha+\frac{3\alpha}{p}}.
				\]
			\item \begin{equation}
				\label{eq:hill-directional-integral}
				\frac{1}{R}\int_{\mathbb{R}^3}\bar H_R(x)\,\dd x = \begin{pmatrix}
					0\\0\\ -2\pi
				\end{pmatrix}.
			\end{equation}
		\end{enumerate}
	\end{lemma}
	\begin{proof}
		The identities in Part~1 follow from the definition of the cutoff and
			the exterior formula \(H_R=R\nabla\Phi\).

		For the norm estimates, the explicit formulas for \(H\) and \(\Phi\)
			give, for \(n=0,1,2\),
		\begin{equation}\label{eq:hill-pointwise-derivatives}
			|\nabla^n H_R(x)|\le
			\begin{cases}
				CR^{-1-\frac{2n}{3}},& |x|<R^{2/3},\\
				CR|x|^{-n-3},& |x|>R^{2/3},
			\end{cases}
		\end{equation}
		where the estimate for \(n=2\) is understood away from the interface.
		Moreover, on the support of a derivative of \(\chi_\alpha\) one has
			\(|x|\simeq R^\alpha\).  Hence, for \(n=0,1,2\),
		\begin{equation}\label{eq:Pi-pointwise}
			|\nabla^n\Pi_R(x)|\le CR|x|^{-n-2},
			\qquad x\ne0.
		\end{equation}
		Combining \eqref{eq:hill-pointwise-derivatives} and
			\eqref{eq:Pi-pointwise}, we obtain
		\begin{equation}\label{est barH}
			|\nabla^n\bar H_R(x)|\le
			\begin{cases}
				CR^{-1-\frac{2n}{3}},& |x|<R^{2/3},\\
				CR|x|^{-n-3},& R^{2/3}<|x|<2R^\alpha,
			\end{cases}
		\end{equation}
		again with \(n=2\) understood away from
			\(\partial B_{R^{2/3}}\).  Integrating \eqref{est barH} in polar
			coordinates proves the \(L^p\)-bounds in Parts~2 and~3.  At \(p=1\),
			the exterior contribution to \(\|\bar H_R\|_{L^1}\) is
			\(CR\int_{R^{2/3}}^{2R^\alpha}r^{-1}\,\dd r\), which accounts for
			the factor \(1+|\log R|\).

		It remains to estimate the divergence.  On \(A_R\),
			\(|\nabla^k\chi_\alpha|\le CR^{-k\alpha}\), while
			\(|\Phi|\le CR^{-2\alpha}\) and
			\(|\nabla\Phi|\le CR^{-3\alpha}\).  Thus
		\[
			|\diver\bar H_R|
			\le CR^{1-4\alpha}\mathbf 1_{A_R},
		\]
		which gives
			\(\|\diver\bar H_R\|_{L^p}\le
			CR^{1-4\alpha+3\alpha/p}\).
		The warning concerning the second derivative is essential: \(H\) is
			Lipschitz across the spherical interface, but it need not have a
			classical second derivative there.  The mollification below removes
			this singularity.

		Finally, compact support of \(\bar H_R\) and integration by parts give,
			componentwise,
		\begin{align*}
			\frac1R\int_{\R^3}\bar H_R(x)\,\dd x
			&=-\frac1R\int_{\R^3}x\,\diver\bar H_R(x)\,\dd x\\
			&=\int_{\R^3}x\bigl((\Delta\chi_\alpha)\Phi
			+2\nabla\chi_\alpha\cdot\nabla\Phi\bigr)\,\dd x\\
			&=-\int_{\R^3}\Phi\nabla\chi_\alpha\,\dd x
			+\int_{\R^3}x(\nabla\chi_\alpha\cdot\nabla\Phi)\,\dd x.
		\end{align*}
		Write \(\chi_\alpha'(r)=\frac{\dd}{\dd r}\chi(r/R^\alpha)\).
		Since \(\Phi=\cos\theta/(2r^2)\) and
			\(\partial_r\Phi=-\cos\theta/r^3\), the last line equals
		\[
			-\frac32\left(\int_0^\infty\chi_\alpha'(r)\,\dd r\right)
			\int_0^{2\pi}\!\int_0^\pi
			\cos\theta\,\widehat r\,\sin\theta\,\dd\theta\,\dd\varphi.
		\]
		The radial integral is one.  The first two angular components vanish,
			and the third is
		\[
			2\pi\int_0^\pi\cos^2\theta\sin\theta\,\dd\theta
			=\frac{4\pi}{3}.
		\]
		This proves \eqref{eq:hill-directional-integral}.
	\end{proof}

	The next proposition shows that the truncated block
		\(\bar H_R\) satisfies the traveling Euler-profile equation up to error
		terms.  We also compute \(\partial_R\bar H_R\), because the scale varies
		along the moving trajectory.
		We use the following Bogovski\u{\i} notation; see
		\cite{Bogovskii1979,Galdi2011} and, for the scaling-invariant \(L^p\)
		bound used below, \cite{AcostaDuranMuschietti2006}.
		If \(f=(f_1,f_2,f_3)\) is a vector field supported in a ball \(B_r\)
		and \(\int_{B_r}f_i\,\dd x=0\) for each \(i\), we define the
		componentwise Bogovski\u{\i} tensor by
		\[
			(\mathcal B_r f)_{ij}:=(\mathcal B_r f_i)_j.
		\]
		It is extended by zero outside \(B_r\) and satisfies
		\[
		\diver\mathcal B_r f=f,\qquad
		\supp\mathcal B_r f\subseteq\overline B_r,\qquad
		\|\mathcal B_r f\|_{L^p}\le Cr\|f\|_{L^p},
		\quad 1<p<\infty.
		\]
		This componentwise tensor is not required to be symmetric; symmetry is
		restored later by the symmetric inverse-divergence operator on the torus.
	
	\begin{proposition}
		Let \(\alpha\in(0,2/3)\).  Then the truncated block
			\(\bar H_R\) satisfies the following identities in the sense of
			distributions on \(\R^3\).
		\begin{equation}\label{eqn of prinblock}
			\frac{1}{R}\,\partial_3 \bar H_R+\operatorname{div}(\bar H_R\otimes \bar H_R)+\nabla P_1= \diver F_1,
		\end{equation}
		\begin{equation}\label{derivative with R}
			\partial_R\bar H_R = \frac{1}{R} \bar H_R+\nabla P_2+ \diver F_2,
		\end{equation}
		where the error and pressure terms have the following properties:
		\begin{enumerate}
			\item After fixing the additive constant in \(P_1\),
				\[
				\supp P_1,\ \supp P_2,\ \supp F_1,\ \supp F_2
				\subseteq\overline B_{2R^\alpha}(0).
				\]
			\item For every \(p\in(1,\infty)\),
				\(\|F_1\|_{L^p}\le CR^{2-6\alpha+3\alpha/p}\) and
				\(\|F_2\|_{L^p}\le CR^{2/p-4/3}\).
			\item
				\[
				\|\nabla P_2\|_{L^\infty}\le CR^{-3\alpha},
				\qquad
				\|\nabla^2P_2\|_{L^\infty}\le CR^{-4\alpha},
				\qquad
				\|\diver F_2\|_{L^\infty}\le CR^{-2},
				\]
				and, away from the spherical interface,
				\[
				\|\nabla\diver F_2\|_
				{L^\infty(\R^3\setminus\partial B_{R^{2/3}})}
				\le CR^{-8/3}.
				\]
		\end{enumerate}
	\end{proposition}

	\begin{proof}
		The calculation follows the localization mechanism in
			\cite[Proposition~2.1]{BrueColomboKumar2024}.  Since
			\(\bar H_R=H_R-\nabla\Pi_R\) and \(H_R\) satisfies
			\eqref{scaled eulerprofile}, we have
		
		\begin{multline*}
			\frac{1}{R}\,\partial_3 \bar H_R
			+\frac{1}{R}\,\partial_3\nabla\Pi_R
			+\operatorname{div}(\bar H_R\otimes\bar H_R)
			+\operatorname{div}(\nabla\Pi_R\otimes\bar H_R)\\
			+\operatorname{div}(\bar H_R\otimes\nabla\Pi_R)
			+\operatorname{div}(\nabla\Pi_R\otimes\nabla\Pi_R)
			+\nabla P_R=0.
		\end{multline*}
		We use the identity
		\begin{equation*}
			\operatorname{div}(\nabla\Pi_R\otimes\nabla\Pi_R)
			=(\Delta\Pi_R)\nabla\Pi_R
			+\frac12\nabla|\nabla\Pi_R|^2.
		\end{equation*}
		The vector field
			\(g_R:=(\Delta\Pi_R)\nabla\Pi_R\) is supported in
			\(\overline A_R\).  To apply \(\mathcal B_{2R^\alpha}\), we verify
			that it has zero mean.  The identity
		\[
			g_R=\diver\!\left(\nabla\Pi_R\otimes\nabla\Pi_R
			-\frac12|\nabla\Pi_R|^2I\right)
		\]
		and the exterior decay
		\(\nabla\Pi_R=R\nabla\Phi=O(R|x|^{-3})\) show, by integration
		over \(B_L\) and passage to the limit \(L\to\infty\), that
		\[
			\int_{\R^3}g_R\,\dd x=0.
		\]
		Indeed, the boundary flux is \(O(R^2L^{-4})\).  Thus every component
		of \(g_R\) satisfies the compatibility condition for the
		componentwise Bogovski\u{\i} operator.
		
		Equation~\eqref{eqn of prinblock} for the truncated block
			follows upon setting
		\begin{equation*}
			P_1=P_R+\frac1R\partial_3\Pi_R
			+\frac12|\nabla\Pi_R|^2,
		\end{equation*}
		\begin{equation}\label{eqn of F1}
			F_1=-\nabla\Pi_R\otimes\bar H_R
			-\bar H_R\otimes\nabla\Pi_R
			-\mathcal B_{2R^\alpha}g_R.
		\end{equation}
		
		On the exterior of \(B_{2R^\alpha}\), we have
			\(\nabla\Pi_R=H_R\) and \(\Delta\Pi_R=0\).  The potential-flow
			form of \eqref{scaled eulerprofile} therefore implies
			\(\nabla P_1=0\) there.  We fix the additive constant so that
			\(P_1=0\) in the exterior.  The stated support properties of
			\(P_1\) and \(F_1\) now follow from the support-preserving property
			of \(\mathcal B_{2R^\alpha}\).
		
		The support of \(\nabla\Pi_R\otimes\bar H_R\) lies in
			\(\overline A_R\) by Lemma~\ref{lemm 1}.  On this annulus,
			\[
			|\bar H_R|+|\nabla\Pi_R|\le CR^{1-3\alpha},
			\qquad
			|\nabla^2\Pi_R|\le CR^{1-4\alpha}.
			\]
		Since \(|A_R|\le CR^{3\alpha}\), it follows that
		\begin{align*}
			\|\nabla\Pi_R\otimes\bar H_R
			+\bar H_R\otimes\nabla\Pi_R\|_{L^p}
			&\le CR^{2-6\alpha+\frac{3\alpha}{p}},\\
			\|g_R\|_{L^p}
			&\le CR^{2-7\alpha+\frac{3\alpha}{p}}.
		\end{align*}
		The scale-\(R^\alpha\) Bogovski\u{\i} estimate then yields
		\[
			\|\mathcal B_{2R^\alpha}g_R\|_{L^p}
			\le CR^\alpha\|g_R\|_{L^p}
			\le CR^{2-6\alpha+\frac{3\alpha}{p}}.
		\]
		Combining the preceding estimates with
			\eqref{eqn of F1} proves the bound for \(F_1\).
		
		Next, we compute \(\partial_R H_R\).
		\begin{equation*}
			\partial_RH_R
			=-\frac1RH_R-\frac23\frac{x}{R}\cdot\nabla H_R
			=\frac1RH_R-\frac23\diver\!\left(H_R\otimes\frac{x}{R}\right).
		\end{equation*}
		Differentiating the cutoff term as well gives
		\begin{equation}
			\partial_R\bar H_R
			=\frac1R\bar H_R-\nabla(R\Phi\,\partial_R\chi_\alpha)
			-\frac23D_R,
			\qquad
			D_R:=\diver\!\left(H_R\otimes\frac{x}{R}\right).
			\label{eq:R-derivative-bar-H}
		\end{equation}
		In the exterior region, \(H_R=R\nabla\Phi\) is homogeneous of degree
			\(-3\), and therefore \(D_R=0\) for \(|x|>R^{2/3}\).  Thus
			\(\supp D_R\subseteq\overline B_{R^{2/3}}\).  Integrating
			\eqref{eq:R-derivative-bar-H} and using
			\eqref{eq:hill-directional-integral}, we also obtain
		\begin{equation*}
			\int_{\R^3}D_R\,\dd x
				=-\frac32R\,\partial_R\!\left(
				\frac1R\int_{\R^3}\bar H_R\,\dd x\right)=0.
		\end{equation*}
		
		We may therefore set
		\[
			P_2:=-R\Phi\,\partial_R\chi_\alpha,
			\qquad
			F_2:=-\frac23\mathcal B_{R^{2/3}}D_R.
		\]
		This proves \eqref{derivative with R}.  From
			\eqref{eq:hill-pointwise-derivatives},
		\[
			\|D_R\|_{L^p}\le CR^{-2+\frac2p},
			\qquad
			\|D_R\|_{L^\infty}\le CR^{-2},
			\qquad
			\|\nabla D_R\|_
			{L^\infty(\R^3\setminus\partial B_{R^{2/3}})}
			\le CR^{-8/3}.
		\]
		The scale-\(R^{2/3}\) Bogovski\u{\i} estimate gives
		\[
			\|F_2\|_{L^p}
			\le CR^{2/3}\|D_R\|_{L^p}
			\le CR^{\frac2p-\frac43}.
		\]
		Finally,
		\[
			\partial_R\chi_\alpha(x)
			=-\frac{\alpha}{R}
			\left(\frac{x}{R^\alpha}\cdot
			\nabla\chi\!\left(\frac{x}{R^\alpha}\right)\right)
		\]
		is supported in \(\overline A_R\).  The bounds for \(P_2\),
			\(\diver F_2=-\frac23D_R\), and its piecewise gradient now follow
			directly.
	\end{proof}
	\begin{corollary}[Mollified Hill block]\label{cor:mollified-hill}
		Fix \(p\in(1,\infty)\) and a radial function
			\(\rho\in C_c^\infty(B_1)\) such that \(\rho\ge0\) and
			\(\int_{\R^3}\rho\,\dd x=1\).  Set
		\[
			\beta:=\frac{14}{3}+\frac{3\alpha-2}{p}-6\alpha,
			\qquad
			\rho_{R^\beta}(x):=R^{-3\beta}\rho(x/R^\beta),
			\qquad
			\widetilde H_R:=\rho_{R^\beta}\ast\bar H_R.
		\]
		Then \(\beta>2/3\), and there exist
			\(\widetilde P_1,\widetilde P_2,\widetilde F_1,\widetilde F_2\)
			such that
		\begin{align}
			\frac1R\partial_3\widetilde H_R
			+\diver(\widetilde H_R\otimes\widetilde H_R)
			+\nabla\widetilde P_1
			&=\diver\widetilde F_1,
			\label{eq:mollified-traveling}\\
			\partial_R\widetilde H_R
			&=\frac1R\widetilde H_R+\nabla\widetilde P_2
			+\diver\widetilde F_2.
			\label{eq:mollified-R-derivative}
		\end{align}
		Moreover,
		\[
			\|\widetilde F_1\|_{L^p}
			\le CR^{2-6\alpha+\frac{3\alpha}{p}},
			\qquad
			\|\widetilde F_2\|_{L^p}
			\le CR^{\frac2p-\frac43},
		\]
		and
		\[
			\|\nabla\widetilde P_2\|_{L^\infty}
			\le CR^{-3\alpha},
			\qquad
			\|\nabla^2\widetilde P_2\|_{L^\infty}
			\le CR^{-4\alpha},
			\qquad
			\|\diver\widetilde F_2\|_{L^\infty}\le CR^{-2},
		\]
		\begin{equation}
			\|\nabla\diver\widetilde F_2\|_{L^\infty}
			\le CR^{-2-\beta}.
			\label{eq:mollified-F2-derivative}
		\end{equation}
		For every \(1<q\le\infty\), the mollified profile also satisfies
		\[
			\|\widetilde H_R\|_{L^q}\le CR^{\frac2q-1},
			\qquad
			\|\nabla\widetilde H_R\|_{L^q}
			\le CR^{\frac2q-\frac53},
			\qquad
			\|\nabla^2\widetilde H_R\|_{L^q}
			\le CR^{\frac2q-\frac53-\beta},
		\]
		and
		\[
			\|\diver\widetilde H_R\|_{L^q}
			\le CR^{1-4\alpha+\frac{3\alpha}{q}}.
		\]
		At \(q=1\), the corresponding endpoint estimates are
		\[
			\|\widetilde H_R\|_{L^1}
			\le CR(1+|\log R|),
			\qquad
			\|\nabla\widetilde H_R\|_{L^1}\le CR^{1/3},
		\]
		\[
			\|\nabla^2\widetilde H_R\|_{L^1}
			\le CR^{1/3-\beta},
			\qquad
			\|\diver\widetilde H_R\|_{L^1}\le CR^{1-\alpha}.
		\]
		The fields \(\widetilde H_R\), \(\widetilde P_1\),
			\(\widetilde P_2\), \(\widetilde F_1\), and \(\widetilde F_2\)
			are supported in \(\overline B_{3R^\alpha}(0)\), and
		\[
			\frac1R\int_{\R^3}\widetilde H_R\,\dd x=-2\pi e_3.
		\]
	\end{corollary}

	\begin{proof}
		Mollifying \eqref{eqn of prinblock} in space gives
			\eqref{eq:mollified-traveling} with
		\[
			\widetilde P_1:=P_1\ast\rho_{R^\beta}
		\]
		and
		\[
			\widetilde F_1
			:=F_1\ast\rho_{R^\beta}
			+\widetilde H_R\otimes\widetilde H_R
			-(\bar H_R\otimes\bar H_R)\ast\rho_{R^\beta}.
		\]
		Young's convolution inequality gives
			\(\|\rho_{R^\beta}\ast F_1\|_{L^p}
			\le CR^{2-6\alpha+3\alpha/p}\).
		Set
		\[
		\mathcal C_R:=(\bar H_R\otimes\bar H_R)\ast\rho_{R^\beta}
		-(\bar H_R\ast\rho_{R^\beta})\otimes
		(\bar H_R\ast\rho_{R^\beta}).
		\]
		Then
		\begin{align*}
			\|\mathcal C_R\|_{L^p}
			&\le CR^\beta\|\bar H_R\|_{L^{2p}}
			\|\nabla\bar H_R\|_{L^{2p}}\\
			&\le CR^{\beta+\frac{2}{p}-\frac{8}{3}}
			=CR^{2-6\alpha+\frac{3\alpha}{p}}.
		\end{align*}
		Here we used precisely the stated choice of \(\beta\).  Since
			\(3\alpha-2<0\) and \(p>1\),
		\[
			\beta>\frac{14}{3}+(3\alpha-2)-6\alpha
			=\frac83-3\alpha>\frac23.
		\]
		This proves the estimate for \(\widetilde F_1\).

		We next compute the derivative with respect to \(R\).  Since the
			mollification scale depends on \(R\),
		\begin{equation*}
			\partial_R\widetilde H_R
			=(\partial_R\bar H_R)\ast\rho_{R^\beta}
			+\bar H_R\ast\partial_R\rho_{R^\beta}.
		\end{equation*}
		Mollifying \eqref{derivative with R} gives
		\begin{equation*}
			(\partial_R\bar H_R)\ast\rho_{R^\beta}
			=\frac1R\widetilde H_R
			+\nabla(P_2\ast\rho_{R^\beta})
			+\diver(F_2\ast\rho_{R^\beta}).
		\end{equation*}
		For vector fields \(F,G\), define the second-order tensor
			\(F\circledast G\) by
			\((F\circledast G)_{ij}:=F_i\ast G_j\), and set
		\[
			K_R(x):=\frac{x}{R}\rho_{R^\beta}(x).
		\]
		A direct differentiation of the rescaled kernel gives
		\[
			\partial_R\rho_{R^\beta}
			=-\frac{\beta}{R}
			\left(3\rho_{R^\beta}
			+x\cdot\nabla\rho_{R^\beta}\right)
			=-\beta\,\diver K_R.
		\]
		Consequently,
		\[
			\bar H_R\ast\partial_R\rho_{R^\beta}
			=-\beta\,\diver(\bar H_R\circledast K_R).
		\]
		Thus \eqref{eq:mollified-R-derivative} holds with
		\[
			\widetilde P_2:=P_2\ast\rho_{R^\beta},
			\qquad
			\widetilde F_2
			:=F_2\ast\rho_{R^\beta}
			-\beta\,\bar H_R\circledast K_R.
		\]
		The kernel satisfies
			\(\|K_R\|_{L^1}\le CR^{\beta-1}\) and
			\(\|\nabla K_R\|_{L^1}\le CR^{-1}\).  Hence
		\begin{align*}
			\|\bar H_R\circledast K_R\|_{L^p}\le C\|K_R\|_{L^1}\|\bar H_R\|_{L^p}\le CR^{\beta+\frac2p-2} \le CR^{\frac{2}{p}-\frac{4}{3}}.
		\end{align*}
		We used \(\beta>2/3\) in the last inequality.  Similarly,
		\begin{align*}
			\|\diver(\bar H_R\circledast K_R)\|_{L^\infty}
			&\le C\|K_R\|_{L^1}\|\nabla\bar H_R\|_{L^\infty}
			\le CR^{\beta-\frac83}
			\le CR^{-2},\\
			\|\nabla\diver(\bar H_R\circledast K_R)\|_{L^\infty}
			&\le C\|\nabla K_R\|_{L^1}
			\|\nabla\bar H_R\|_{L^\infty}
			\le CR^{-8/3}.
		\end{align*}
		Since \(0<R<1\) and \(\beta>2/3\),
		\[
			R^{-8/3}\le R^{-2-\beta}.
		\]
		Thus this term satisfies the bound asserted in
		\eqref{eq:mollified-F2-derivative}.
		Because \(H\) is only Lipschitz, two derivatives cannot be
			taken uniformly across the spherical interface before
			mollification.  Instead, we place one derivative on the mollifier:
		\[
			\|\nabla\diver(F_2\ast\rho_{R^\beta})\|_{L^\infty}
			\le C\|\diver F_2\|_{L^\infty}
			\|\nabla\rho_{R^\beta}\|_{L^1}
			\le CR^{-2-\beta}.
		\]
		Together with the preceding estimates and Young's inequality, this
			proves all asserted bounds for \(\widetilde P_2\) and
			\(\widetilde F_2\).

		The estimates for \(\widetilde H_R\) follow from
			Lemma~\ref{lemm 1}; for the second derivative, use
		\[
			\|\nabla^2\widetilde H_R\|_{L^q}
			\le \|\nabla\bar H_R\|_{L^q}
			\|\nabla\rho_{R^\beta}\|_{L^1}
			\le CR^{\frac2q-\frac53-\beta}.
		\]
		Convolution preserves the integral.  Finally,
			\(\beta>2/3>\alpha\) implies \(R^\beta<R^\alpha\), so convolution
			enlarges each support by at most \(R^\beta\); this yields the stated
			support in \(\overline B_{3R^\alpha}\).
	\end{proof}
	
	For \(e\in\mathbb S^2\), we use the rotated notation
	\[
		\widetilde H_{R,e}(x)
		:=Q_e\widetilde H_R(Q_e^Tx),
		\qquad
		\widetilde P_{j,R,e}(x)
		:=\widetilde P_j(Q_e^Tx),
	\]
	and
	\[
		\widetilde F_{j,R,e}(x)
		:=Q_e\widetilde F_j(Q_e^Tx)Q_e^T,
		\qquad j\in\{1,2\}.
	\]
	Equations
	\eqref{eq:mollified-traveling}--\eqref{eq:mollified-R-derivative} and all
	the preceding estimates are rotation invariant, while
	\[
		\frac1R\int_{\R^3}\widetilde H_{R,e}\,\dd x=-2\pi e.
	\]

	\section{The moving Hill block}
	
	We now modulate the amplitude and scale of the localized Hill profile
	and translate its center along a time-dependent trajectory.  The velocity
	of the center is chosen to match the intrinsic traveling speed of the
	profile; this cancels the leading transport term against the quadratic
	Euler term.  Scale and amplitude modulation leave a distinguished vector
	source whose spatial mean is prescribed by
	\eqref{eq:hill-directional-integral}.  All remaining terms are placed in a
	symmetric stress.  A Helmholtz correction restores incompressibility, and
	the final corollary incorporates viscosity.
	
	\begin{proposition}[Moving Hill block]\label{main prop}
		Let \(I\subset\R\) be an interval, fix
			\(e\in\mathbb S^2\) and \(p\in(1,\infty)\), and choose
			\[
			\beta=\frac{14}{3}+\frac{3\alpha-2}{p}-6\alpha
			\]
			as in Corollary~\ref{cor:mollified-hill}.  Shrinking \(R_0\) if
			necessary, assume
			\[
				3R_0^\alpha<\frac12.
			\]
			Let
		\[
		\eta\in C_c^\infty(I),
		\qquad
		\mathfrak R\in C^1(\T^3;(0,R_0]).
		\]
		Choose \(t_0\in I\) and \(x_0\in\T^3\), and let \(X\) solve
		\begin{equation}\label{eq:moving-hill-trajectory}
			\begin{cases}
				X'(t)= -\dfrac{\eta(t)}{\mathfrak R(X(t))}e,\\
				X(t_0)= x_0.
			\end{cases}
		\end{equation}
		Set \(R(t):=\mathfrak R(X(t))\).
		
		Define the principal block
		\begin{equation}
			V^p(x,t):=\eta(t)\,\widetilde H_{R(t),e}(x-X(t)).
			\label{eq:V-def}
		\end{equation}
		We suppress the parameters \(R\) and \(e\) on the pressure and error
			terms when no confusion can arise.  The profile identities are
		\begin{equation*}
			\frac{1}{R}\,e\cdot\nabla\widetilde H_{R,e}
			+\diver(\widetilde H_{R,e}\otimes\widetilde H_{R,e})
			+\nabla\widetilde P_1
			=\diver\widetilde F_1,
		\end{equation*}
		\begin{equation*}
			\partial_R\widetilde H_{R,e}
			=\frac{1}{R}\widetilde H_{R,e}
			+\nabla\widetilde P_2+\diver\widetilde F_2.
		\end{equation*}
		The support condition on \(R_0\) allows the compactly supported
			Euclidean profiles to be periodized on \(\T^3=(\R/\Z)^3\) without
			overlap between distinct copies.
		
		Since \(V^p\) is periodic,
		\[
			\int_{\T^3}\diver V^p(x,t)\,\dd x=0,
		\]
		so the periodic inverse Laplacian below is well defined on
		\(\diver V^p\).  To restore incompressibility on \(\T^3\), define
		\begin{equation*}
			V^c:=-\nabla\Delta^{-1}\diver V^p,
			\qquad
			V:=V^p+V^c.
		\end{equation*}
		Then
		\[
			\int_{\T^3}V^c(x,t)\,\dd x=0,
			\qquad
			\diver V(\cdot,t)=0
		\]
		for every $t\in I$, and
		\begin{align}
			\partial_t V+\diver(V\otimes V)+\nabla P
			&= \frac{\dd}{\dd t}(\eta(t)R(t))\frac{1}{R(t)}
				\widetilde H_{R(t),e}(x-X(t))+\diver F,
			\label{eq:euler-defect-form}
		\end{align}
		where \(F\) is a symmetric matrix field.  The following
			estimates hold:
		\begin{enumerate}
			\item The stress satisfies
			\[
			\|F\|_{L_t^\infty L_x^1}
			\le C\|\eta\|_{L^\infty}^2\|R^\kappa\|_{L^\infty(I)}
			(1+\|\nabla\log\mathfrak R\|_{L^\infty}),
			\]
			where
			\[
			\kappa = \min\left\{
			\frac{2+3\alpha}{2p}-4\alpha,
			2-6\alpha+\frac{3\alpha}{p},
			\frac{2}{p}-\frac{4}{3}\right\}.
			\]
			\item For every \(q\in(1,\infty)\), the displayed spatial estimates
			below are understood uniformly in $t\in I$; all powers of the
			variable \(R(t)\) are taken in \(L^\infty(I)\), and all unlabelled
			norms of \(\mathfrak R\) are over \(\T^3\):
			\begin{align*}
				\|V\|_{L_t^\infty L_x^q}
				&\le C\|\eta\|_\infty
				\|R^{\frac2q-1}\|_{L^\infty(I)},\\
				\|\nabla V\|_{L_t^\infty L_x^q}
				&\le C\|\eta\|_\infty
				\|R^{\frac2q-\frac53}\|_{L^\infty(I)},\\
				\|\partial_tV\|_{L_t^\infty L_x^q}
				&\le C\|\eta\|_\infty^2\|R^{-3}\|_{L^\infty(I)}
				(1+\|\nabla\mathfrak R\|_\infty)
				+C\|\eta'\|_\infty\|R^{-1}\|_{L^\infty(I)},\\
				\|\partial_t\nabla V\|_{L_t^\infty L_x^q}
				&\le C\|\eta\|_\infty^2
				\|R^{-3-\beta}\|_{L^\infty(I)}
				(1+\|\nabla\mathfrak R\|_\infty)
				+C\|\eta'\|_\infty
				\|R^{-\frac53}\|_{L^\infty(I)}.
			\end{align*}
			\item For every \(q\in(1,\infty)\),
				\(\|R^{-1}\widetilde H_{R,e}\|_{L^q}
				\le CR^{\frac{2}{q}-2}\), and
			\begin{equation*}
				\int_{\T^3}\frac{1}{R}\widetilde H_{R,e}\,\dd x
					=-2\pi e.
			\end{equation*}
		\end{enumerate}
		In particular,
		\[
			\int_{\T^3}V(x,t)\,\dd x
			=-2\pi\eta(t)R(t)e.
		\]
		Thus the block is not itself mean zero; the distinguished vector
			source in \eqref{eq:euler-defect-form} is exactly the time
			derivative of its spatial mean.
		
	\end{proposition}

	\begin{proof}
		Write \(Y:=x-X(t)\).  Differentiating \eqref{eq:V-def} gives
		\begin{align}
			\partial_t V^p(x,t)
			&=\eta'(t)\widetilde H_{R(t),e}(Y)
			-\eta(t)X'(t)\cdot\nabla\widetilde H_{R(t),e}(Y)
			\notag\\
			&\quad +\eta(t)R'(t)
				\left[\partial_R\widetilde H_{R,e}\right]_{R=R(t)}(Y).
			\label{eq:time-derivative}
		\end{align}
		Moreover,
		\begin{equation*}
			\diver(V^p\otimes V^p)(x,t)
			=\eta(t)^2\,
			\diver\bigl(\widetilde H_{R(t),e}
			\otimes\widetilde H_{R(t),e}\bigr)(Y).
		\end{equation*}
		By Corollary~\ref{cor:mollified-hill},
		\begin{equation*}
			\diver(\widetilde H_{R,e}\otimes\widetilde H_{R,e})
			+\nabla\widetilde P_1
			=-\frac1R(e\cdot\nabla)\widetilde H_{R,e}
			+\diver\widetilde F_1.
		\end{equation*}
		The scale derivative satisfies
		\begin{equation*}
			\partial_R\widetilde H_{R,e}
			=\frac1R\widetilde H_{R,e}
			+\nabla\widetilde P_2+\diver\widetilde F_2.
		\end{equation*}
		Combining these identities with the trajectory equation
			\eqref{eq:moving-hill-trajectory}, the transport term cancels and
			we obtain
		\begin{equation}\label{eq:principal-moving-block}
			\partial_tV^p+\diver(V^p\otimes V^p)+\nabla P^p
			=\frac{\dd}{\dd t}(\eta R)\frac1R
			\widetilde H_{R,e}(Y)+\diver A,
		\end{equation}
		where all profiles on the right are evaluated at \(R=R(t)\), and
		\[
			P^p(x,t):=\eta(t)^2\widetilde P_{1,R(t),e}(Y)
			-\eta(t)R'(t)\widetilde P_{2,R(t),e}(Y),
		\]
		\[
			A(x,t):=\eta(t)^2\widetilde F_{1,R(t),e}(Y)
			+\eta(t)R'(t)\widetilde F_{2,R(t),e}(Y).
		\]
		
		Since
			\(\partial_tV^c=-\nabla\partial_t\Delta^{-1}\diver V^p\),
			this term can be absorbed into the pressure.
			Let \(\mathcal R\) denote the symmetric inverse-divergence operator on
			\(\T^3\) introduced in~\cite{DeLellisSzekelyhidi2013}; we use its
			\(L^p\) and order-\((-1)\) bounds in the form established in
			\cite{BuckmasterDeLellisSzekelyhidiVicol2019,BuckmasterVicol2019}.
			The vector field \(\diver A\) has zero spatial mean, and
			hence
			\[
				\diver\mathcal R(\diver A)=\diver A.
			\]
		Define
		\begin{align}
			P&:=P^p+\partial_t\Delta^{-1}\diver V^p,
			\notag\\
			F&:=\mathcal R(\diver A)
			+V^p\otimes V^c+V^c\otimes V^p+V^c\otimes V^c.
			\label{eq:moving-hill-stress}
		\end{align}
		The tensor \(F\) is symmetric.  Adding the equation for \(V^c\) to
			\eqref{eq:principal-moving-block} now proves
			\eqref{eq:euler-defect-form}.

		It remains to establish the estimates.  The chain rule and
			\eqref{eq:moving-hill-trajectory} give
		\begin{equation}\label{eq:moving-radius-derivative}
			R'(t)
			=\nabla\mathfrak R(X(t))\cdot X'(t)
			=-\eta(t)e\cdot\nabla\log\mathfrak R(X(t)).
		\end{equation}
		
		Set \(F_3:=\mathcal R(\diver A)\).  Since
			\(\mathcal R\diver\) is a Calder\'on--Zygmund operator on
			\(\T^3\)~\cite{LemarieRieusset2002},
			Corollary~\ref{cor:mollified-hill} and
			\eqref{eq:moving-radius-derivative} give
		\begin{align}
			\|F_3\|_{L_t^\infty L_x^1}
			&\le C\|F_3\|_{L_t^\infty L_x^p}\notag\\
			&\le C\|\eta\|_{L^\infty}^2
			\left(
			\|R^{2-6\alpha+\frac{3\alpha}{p}}\|_{L^\infty(I)}
			+\|\nabla\log\mathfrak R\|_{L^\infty}
			\|R^{\frac2p-\frac43}\|_{L^\infty(I)}
			\right).
			\label{eq:moving-F3-estimate}
		\end{align}

		We use two complementary bounds for the Helmholtz correction.  For
			every \(s\in(1,\infty)\), Calder\'on--Zygmund boundedness gives
		\[
			\|V^c\|_{L^s}\le C\|V^p\|_{L^s},
		\]
		while Poincar\'e's inequality and the definition of \(V^c\) give
		\[
			\|V^c\|_{L^s}
			\le C\|\nabla V^c\|_{L^s}
			\le C\|\diver V^p\|_{L^s}.
		\]
		In particular, at \(s=2p\),
		\[
			\|V^p\|_{L^{2p}}
			\le C\|\eta\|_{L^\infty}R^{\frac1p-1},
			\qquad
			\|V^c\|_{L^{2p}}
			\le C\|\eta\|_{L^\infty}
			R^{1-4\alpha+\frac{3\alpha}{2p}}.
		\]
		The first of the two correction bounds also implies
			\(\|V^c\otimes V^c\|_{L^p}
			\le C\|V^p\|_{L^{2p}}\|V^c\|_{L^{2p}}\).  Therefore,
		\begin{equation}
			\|V^p\otimes V^c+V^c\otimes V^p+V^c\otimes V^c\|_{L^p}
			\le C\|\eta\|_{L^\infty}^2
			R^{\frac{2+3\alpha}{2p}-4\alpha}.
			\label{eq:moving-corrector-interaction}
		\end{equation}

		Combining \eqref{eq:moving-F3-estimate},
			\eqref{eq:moving-corrector-interaction}, and
			\eqref{eq:moving-hill-stress}, and using
			\(|\T^3|=1\), yields the asserted stress estimate with
		\[
			\kappa=\min\left\{
			\frac{2+3\alpha}{2p}-4\alpha,\,
			2-6\alpha+\frac{3\alpha}{p},\,
			\frac2p-\frac43
			\right\}.
		\]
		
		For the remaining estimates, fix \(q\in(1,\infty)\).  The
			profile estimates in Corollary~\ref{cor:mollified-hill} give
		\[
			\|V^p\|_{L^q}
			\le C|\eta|R^{\frac2q-1},
			\qquad
			\|\nabla V^p\|_{L^q}
			\le C|\eta|R^{\frac2q-\frac53}.
		\]
		The Calder\'on--Zygmund bounds for the Helmholtz projection give
		\begin{align*}
			\|V^c\|_{L^q}&\le C\|V^p\|_{L^q}, &
			\|\nabla V^c\|_{L^q}&\le C\|\nabla V^p\|_{L^q},\\
			\|\partial_tV^c\|_{L^q}&\le C\|\partial_tV^p\|_{L^q},&
			\|\partial_t\nabla V^c\|_{L^q}
			&\le C\|\partial_t\nabla V^p\|_{L^q}.
		\end{align*}
		It remains to estimate the two time derivatives of \(V^p\).
			Equation~\eqref{eq:time-derivative} and
			Corollary~\ref{cor:mollified-hill} yield, at each fixed \(t\),
		\begin{align*}
			\|\partial_tV^p\|_{L^\infty}
			&\le \frac{|\eta|^2}{R}
			\|\nabla\widetilde H_{R,e}\|_{L^\infty}
			+|\eta'|\|\widetilde H_{R,e}\|_{L^\infty}\\
			&\quad+|\eta R'|\left(
			\frac1R\|\widetilde H_{R,e}\|_{L^\infty}
			+\|\nabla\widetilde P_2\|_{L^\infty}
			+\|\diver\widetilde F_2\|_{L^\infty}\right)\\
			&\le C\left(
			|\eta|^2R^{-\frac83}
			+|\eta'|R^{-1}
			+|\eta R'|R^{-2}\right),
		\end{align*}
	and
		\begin{align*}
			\|\partial_t\nabla V^p\|_{L^\infty}
			&\le \frac{|\eta|^2}{R}
			\|\nabla^2\widetilde H_{R,e}\|_{L^\infty}
			+|\eta'|\|\nabla\widetilde H_{R,e}\|_{L^\infty}\\
			&\quad+|\eta R'|\left(
			\frac1R\|\nabla\widetilde H_{R,e}\|_{L^\infty}
			+\|\nabla^2\widetilde P_2\|_{L^\infty}
			+\|\nabla\diver\widetilde F_2\|_{L^\infty}\right)\\
			&\le C\left(
			|\eta|^2R^{-\frac83-\beta}
			+|\eta'|R^{-\frac53}
			+|\eta R'|R^{-2-\beta}\right).
		\end{align*}
		In the last two estimates we used \(\alpha<2/3\) and
			\(\beta>2/3\) to dominate the pressure terms by the displayed
			powers of \(R\).
		Using \eqref{eq:moving-radius-derivative} in the form
			\[
				|R'(t)|
				\le |\eta(t)|R(t)^{-1}
				\|\nabla\mathfrak R\|_{L^\infty},
			\]
		together with \(R\le R_0<1\), gives the asserted time-derivative
			bounds.  The \(L^q\)-estimates follow from the \(L^\infty\)-bounds
			because \(|\T^3|=1\).
		
		The final assertion follows from Corollary~\ref{cor:mollified-hill}
			and is unchanged by periodization.
	\end{proof}
	
	\begin{corollary}[Navier--Stokes moving Hill block]
	\label{cor:navier-stokes-moving-hill}
		Under the assumptions of Proposition~\ref{main prop}, the
			divergence-free block \(V\) satisfies
		\begin{equation}
			\partial_t V+\diver(V\otimes V)-\nu \Delta V+\nabla P
			= \frac{\dd}{\dd t}(\eta(t)R(t))\frac{1}{R(t)}
				\widetilde H_{R(t),e}(x-X(t))+\diver F^{NS},
			\label{eq:NS-block}
		\end{equation}
		where \(F^{NS}\) is symmetric and satisfies
		\begin{equation*}
			\|F^{NS}\|_{L_t^\infty L_x^1}
				\le C\|\eta\|_{L^\infty}^2
				\|R^\kappa\|_{L^\infty(I)}
				(1+\|\nabla\log\mathfrak R\|_{L^\infty})
				+C\nu\|\eta\|_{L^\infty}
				\|R^{\frac2p-\frac53}\|_{L^\infty(I)}.
		\end{equation*}
	\end{corollary}
	\begin{proof}
		Since \(\diver V=0\),
		\[
			-\nu\Delta V
			=\diver\!\left[-\nu\bigl(\nabla V+(\nabla V)^T\bigr)\right].
		\]
		Thus, treating viscosity as part of the stress gives
		\begin{align*}
			\partial_t V+\diver(V\otimes V)-\nu \Delta V+\nabla P
			&= \frac{\dd}{\dd t}(\eta(t)R(t))\frac{1}{R(t)}
			\widetilde H_{R(t),e}(x-X(t))\\
			&\quad+\diver\!\left(F-\nu(\nabla V+(\nabla V)^T)\right).
		\end{align*}
		
		Thus \eqref{eq:NS-block} follows after setting
			\[
				F^{NS}:=F-\nu(\nabla V+(\nabla V)^T).
			\]
		This tensor is symmetric.  By
			Proposition~\ref{main prop},
		\[
			\|\nu(\nabla V+(\nabla V)^T)\|_{L_t^\infty L_x^1}
			\le C\nu\|\nabla V\|_{L_t^\infty L_x^p}
			\le C\nu\|\eta\|_{L^\infty}
			\|R^{\frac2p-\frac53}\|_{L^\infty(I)}.
		\]
	\end{proof}

	\section{The iteration step}

	We carry out one stage of the Navier--Stokes--Reynolds iteration.
	Starting from an admissible triple
	$(u_q,p_q,\E_q)$, we regularize the background flow, decompose the
	mollified stress into rank-one components, and assign each component to
	a moving Hill block on a short time interval.  An auxiliary source and
	a temporal corrector then convert the orbit-averaged cancellation into
	a pointwise-in-time equation.  Here \,admissible\, means that the velocity
	is continuous in time and the triple is smooth in the interior of each
	time cell, with bounded one-sided derivatives at the cell interfaces;
	the Navier--Stokes--Reynolds system is understood distributionally.
	This class is preserved by the construction below and becomes smooth
	after the mollification at the beginning of the next stage.
	
	To avoid boundary artifacts in the space--time convolution, all triples
	are constructed on a fixed open interval \(I_\ast\supset[0,1]\), and the
	iteration is restricted to \([0,1]\) only after mollification.  All
	time-dependent estimates below are understood on \(I_\ast\); this harmless
	convention will not be repeated.

	We begin by fixing the parameter hierarchy.  At the $q$th stage,
	$\lambda_q$ controls the density of the rational orbits, $\delta_q$
	measures the Reynolds stress, $r_q$ is the base concentration scale,
	and $\tau_q$ is the length of a time cell.  For
	$q\in\mathbb Z_{\geq0}$, set
	\begin{equation*}
		\lambda_{q+1} = \lambda_q^{\sigma}, \quad \delta_q = \lambda_1^{2\gamma}\lambda_q^{-\gamma}, \quad r_{q+1}= \lambda_{q+1}^{-\mu}, \quad \tau_{q+1}= \lambda_{q+1}^{-\eta}.
	\end{equation*}
	The constants $\mu,\eta,\gamma,\sigma$ are positive, while $\lambda_0$
	depends on the initial data and will be chosen sufficiently large.  We
	also take $\lambda_0$, $\sigma$, and $\eta$ to be integers.
	For the intrinsic moving-block estimates we fix, throughout this
	section,
	\begin{equation}\label{eq:intrinsic-exponents}
		p_0=\frac{28}{27},\qquad
		\alpha=\frac{2}{11},\qquad
		\beta=\frac{502}{231},\qquad
		\kappa=\frac12.
	\end{equation}
	The numerical choices of the remaining parameters, and the verification
	of all exponent inequalities, are given in the parameter-selection
	subsection \ref{sec:choose para} below.
	
	Suppose that the Navier--Stokes--Reynolds system at the $q$th stage is
	\begin{align}\label{eq:NSR}
		\partial_t u_q+\diver(u_q\otimes u_q)-\nu \Delta u_q+\nabla p_q
		&= \diver(\mathcal{E}_q),\\
		\diver u_q&=0\nonumber.
	\end{align}
	The triple is admissible on \(\T^3\times I_\ast\), and the induction
	assumptions are
	\begin{equation}\label{induction hypo}
		\|\mathcal{E}_q\|_{L_t^\infty L_x^1}\le \delta_{q+1},\quad
		\|u_q\|_{L_t^\infty L_x^2}
		\le 2\delta_0^{\frac12}-\delta_q^{\frac12},\quad
		\|(\nabla_xu_q,\partial_tu_q)\|_{L_t^\infty L_x^4}
		\le \lambda_q^n.
	\end{equation}

	Here \(n>0\) is fixed in the parameter selection below.
	\begin{proposition}\label{iteration}
		Let \(\bar p=\frac65+5\times10^{-5}\).  Choose the
			constants \(\mu,\eta,\gamma,\sigma\) as in the parameter selection
			below, and let \(\lambda_0\) be sufficiently large.  If
			\((u_q,p_q,\E_q)\) is admissible and satisfies
			\eqref{eq:NSR} and \eqref{induction hypo}, then there exists an
			admissible triple
			\((u_{q+1},p_{q+1},\E_{q+1})\) satisfying \eqref{eq:NSR} and the
			following bounds:
		\begin{enumerate}
			\item \(\|\mathcal{E}_{q+1}\|_{L_t^\infty L_x^1}\le
			\delta_{q+2}\),
			\(\|u_{q+1}\|_{L_t^\infty L_x^2}\le
			2\delta_0^{1/2}-\delta_{q+1}^{1/2}\), and
			\(\|(\nabla_xu_{q+1},\partial_tu_{q+1})\|_{L_t^\infty L_x^4}
			\le\lambda_{q+1}^n\),
			\item \(\|u_{q+1}-u_q\|_{L_t^\infty L_x^2}
			\le C\delta_{q+1}^{1/2}\),
			\item There exist \(\alpha_0,\vartheta>0\), independent of \(q\),
			such that
			\[
			\|\nabla u_{q+1}\|_{C_t^{\alpha_0}L_x^{\bar p}}
			\le
			\|\nabla u_q\|_{C_t^{\alpha_0}L_x^{\bar p}}
			+C\delta_{q+1}^{\vartheta},
			\]
			\item \(\|u_{q+1}(\cdot,0)-u_q(\cdot,0)\|_{L^2}
			\le C\lambda_q^{-1/5}\) and
			\(\|u_{q+1}(\cdot,1)-u_q(\cdot,1)\|_{L^2}
			\le C\lambda_q^{-1/5}\).
		\end{enumerate}
	\end{proposition}

	\subsection{Mollification step}
	We regularize the background flow and Reynolds stress at a
	space--time scale $l$.  This provides the derivative bounds needed to
	define the space-dependent radii and trajectories of the vortex cores.
	The commutator generated by mollifying the nonlinear term is included
	in the new Reynolds stress.  Fix a sufficiently small absolute constant
	$c_0>0$ and set
	\begin{equation*}
		l:=c_0\lambda_{q+1}^{-n/\sigma-\gamma}.
	\end{equation*}
	Let
	\(\rho_l(t,x)=l^{-4}\rho(t/l,x/l)\) be a smooth space--time mollifier
	on \(\mathbb R\times\mathbb T^3\), where \(\rho\) is supported in the
	unit space--time ball, and
	set
	\begin{equation}\label{mollified eqn}
		u_l = \rho_l\ast u_q,\quad
			p_l = \rho_l\ast p_q,\quad
			\E_l = \rho_l\ast \E_q
			+u_l\otimes u_l-(u_q\otimes u_q)\ast\rho_l.
	\end{equation}

	Then \((u_l,p_l,\E_l)\) solves \eqref{eq:NSR} and satisfies
	\begin{align*}
		\|\E_l\|_{L_t^\infty L_x^1}
		&\le \|\rho_{l}\ast \E_q\|_{L_t^\infty L_x^1}
		+\|u_l\otimes u_l-(u_q\otimes u_q)\ast \rho_l\|_{L_t^\infty L_x^1}\nonumber\\
		&\le \delta_{q+1}+ C l\|u_q\|_{L_t^\infty L_x^2}
		\| (\nabla_xu_q,\partial_tu_q)\|_{L_t^\infty L_x^2}\nonumber\\
		&\le  \delta_{q+1}
		+C c_0\delta_0^\frac{1}{2}\lambda_{q+1}^{-\gamma}\nonumber\\
		&\le 2\delta_{q+1}.
	\end{align*}
	Here the last inequality follows from
	$\delta_0^{1/2}\le\lambda_1^{2\gamma}$ and the definition of
	$\delta_{q+1}$.  This verifies that the smaller scale still controls the
	mollification commutator.

	We also check the inverse powers of $l$.  At $q=0$ the initial triple is
	fixed and smooth before $\lambda_0$ is chosen, so the following bounds hold
	after increasing $\lambda_0$.  For $q\ge1$, we use
	$\lambda_{q+1}\ge\lambda_2=\lambda_1^\sigma$ and
	$W^{1,4}(\T^3)\hookrightarrow L^\infty(\T^3)$ to obtain
	\begin{align*}
		\|\E_l\|_{C^0_{x,t}}
		&\le C l^{-3}\|\E_q\|_{L_t^\infty L_x^1}
		+C\lambda_q^{2n}\nonumber\\
		&\le C\lambda_1^{2\gamma}
		\lambda_{q+1}^{3n/\sigma+2\gamma}
		+C\lambda_{q+1}^{2n/\sigma}
		\le C\lambda_{q+1}^{3n/\sigma+3\gamma},\\
		\|\E_l\|_{\dot C^1_{x,t}}
		&\le C l^{-4}\|\E_q\|_{L_t^\infty L_x^1}
		+C l^{-1}\lambda_q^{2n}\nonumber\\
		&\le C\lambda_1^{2\gamma}
		\lambda_{q+1}^{4n/\sigma+3\gamma}
		+C\lambda_{q+1}^{3n/\sigma+\gamma}
		\le C\lambda_{q+1}^{4n/\sigma+4\gamma}.
	\end{align*}
	In the last inequalities we used
	$\lambda_1^{2\gamma}\le\lambda_{q+1}^{2\gamma/\sigma}$.
	Thus all subsequent bounds involving powers of $l$ retain the exponents
	used below.

	\subsection{Error decomposition}
	
	The decomposition of a symmetric matrix into nine positive rank-one
	components is the geometric lemma of
	\cite{Nash1954,DeLellisSzekelyhidi2013,
	BuckmasterDeLellisSzekelyhidiVicol2019}.  The refinement producing
	directions with rational component ratios and orbit periods comparable to
	\(\lambda_{q+1}^2\) follows~\cite{BrueColomboKumar2024}; the perturbation
	matrix \(K\) and the density argument in item~\textup{(i)} below are
	specific to the present three-dimensional setting.
	
	We resolve the mollified Reynolds stress, up to a scalar pressure,
	into finitely many positive rank-one components.  The directions are
	chosen to have rational components, periods comparable to
	$\lambda_{q+1}^2$, and \(O(\lambda_{q+1}^{-1})\)-dense trajectories;
	they will later determine the trajectories of the moving vortex cores.
	The positive coefficients determine both the amplitudes and the spatial
	scales of the corresponding perturbations.
	
	\begin{lemma}[{Error Decomposition}]\label{lemma:error dec}
		Given $\lambda_{q+1}\geq192$ and $\delta_{q+1}>0$, there exist
		unit vectors $\xi_i\in\R^3$, $i=1,\dots,9$, such that,
		for every
		\(\E_l\in
		C^\infty(\mathbb T^3\times I_\ast;\operatorname{Sym}_3)\),
		the
		following statements hold.
		
		\begin{enumerate}[label = (\roman*)]
			\item The curve $s\mapsto s\xi_i$ on $\mathbb{T}^3$ has minimal period
			$c_i\lambda_{q+1}^2$ for some $c_i\in[1/4,2]$, and its image is
			\(C_0\lambda_{q+1}^{-1}\)-dense in \(\mathbb T^3\), where
			\(C_0\) is independent of \(q\) and \(i\).
			\item  The following decomposition holds:
			\begin{align}\label{eq:error dec}
				-\diver(\E_l)
					=\diver\left(\sum_{i=1}^9a_i(x,t)\xi_i\otimes\xi_i\right)
					-\nabla\zeta.
			\end{align}
			
			\item The functions $a_i(x, t)$ are smooth and satisfy
			\begin{align}\label{eq:est_a}
				a_i(x,t)&\ge \delta_{q+1},\\
				\|a_i\|_{L_t^\infty L_x^1}&\le192\delta_{q+1},\notag\\
				\|a_i\|_{C^0_{x,t}}
				&\le C\bigl(\delta_{q+1}+\|\E_l\|_{C^0_{x,t}}\bigr),\notag\\
				\|a_i\|_{C^1_{x,t}}
				&\le C\bigl(\delta_{q+1}
					+\|\E_l\|_{C^1_{x,t}}\bigr).\notag
			\end{align}
			
		\end{enumerate}
	\end{lemma}
	
	\begin{proof}
		Define the following nine unit vectors:
		\begin{align*}
			e_1&=(1,0,0)^T,
			&e_2&=(0,1,0)^T,
			&e_3&=(0,0,1)^T,\\
			e_4&=2^{-1/2}(1,1,0)^T,
			&e_5&=2^{-1/2}(0,1,1)^T,
			&e_6&=2^{-1/2}(1,0,1)^T,\\
			e_7&=2^{-1/2}(1,-1,0)^T,
			&e_8&=2^{-1/2}(0,1,-1)^T,
			&e_9&=2^{-1/2}(1,0,-1)^T.
		\end{align*}
		For a symmetric \(3\times3\) matrix \(\bar\E\), use the
			decomposition
		\begin{align}\label{dec-r-bar}
			\bar{\E} = \sum_{i=1}^9 \bar{\Gamma}_i(\bar{\E}) \, e_i \otimes e_i,
		\end{align}
		where $\bar{\Gamma}_i$ are smooth functions given by
		\begin{align*}
			\bar\Gamma_1(\bar\E)
			&:=\bar\E_{11}-\bar\E_{12}-\bar\E_{13}-\frac12,\\
			\bar\Gamma_2(\bar\E)
			&:=\bar\E_{22}-\bar\E_{12}-\bar\E_{23}-\frac12,\\
			\bar\Gamma_3(\bar\E)
			&:=\bar\E_{33}-\bar\E_{23}-\bar\E_{13}-\frac12,\\
			\bar\Gamma_4(\bar\E)&:=2\bar\E_{12}+\frac14,
			&\bar\Gamma_5(\bar\E)&:=2\bar\E_{23}+\frac14,\\
			\bar\Gamma_6(\bar\E)&:=2\bar\E_{13}+\frac14,
			&\bar\Gamma_7(\bar\E)&:=\frac14,\\
			\bar\Gamma_8(\bar\E)&:=\frac14,
			&\bar\Gamma_9(\bar\E)&:=\frac14.
		\end{align*}
		
		If \(\|\bar\E-I_{3\times3}\|_\infty<1/16\), then
			\(1/8\le\bar\Gamma_i(\bar\E)\le2\), and
		%\begin{align}
		$\max_{i, j, k} \left|\frac{\partial \bar{\Gamma}_i}{\partial \bar{\E}_{j, k}}\right| \leq 2$.
		%\end{align}
		
		Put \(N=\lambda_{q+1}\), and let \(K\) denote the matrix
		\begin{equation*}
			K=\frac1{N^2}
			\begin{pmatrix}
				N^2&1&N\\
				N&N^2+1&1\\
				1&N&N^2+2
			\end{pmatrix}.
		\end{equation*}
		Define vectors \(\xi_1,\dots,\xi_9\in\R^3\), whose
			component ratios are rational, by
		\begin{align*}
			\xi_i\coloneqq-\frac{Ke_i}{|Ke_i|},\qquad i=1,\dots,9.
		\end{align*}
		To verify item~\textup{(i)}, write
		\[
		v_i=e_i\quad (i=1,2,3),\qquad
		v_i=\sqrt2\,e_i\quad (i=4,\ldots,9),
		\]
		so that \(v_i\in\mathbb Z^3\), and set
		\[
		M=N^2K=N^2I+NA+B,\qquad m_i=Mv_i,
		\]
		where
		\[
		A=\begin{pmatrix}0&0&1\\1&0&0\\0&1&0\end{pmatrix},
		\qquad
		B=\begin{pmatrix}0&1&0\\0&1&1\\1&0&2\end{pmatrix}.
		\]
		Then \(Ke_i\) is parallel to \(m_i\), and hence
		\(\xi_i=-m_i/|m_i|\).  For example,
		\[
		m_1=(N^2,N,1),\quad
		m_4=(N^2+1,N^2+N+1,N+1),\quad
		m_6=(N^2+N,N+1,N^2+3).
		\]
		The vectors \(m_i\) need not themselves be primitive.  If
		\(d_i\) denotes the greatest common divisor of their components,
		a direct calculation gives
		\[
		d_i=1\quad (i\ne6,9),\qquad
		d_6=\gcd(N+1,4),\qquad
		d_9=\gcd(N-1,2).
		\]
		Thus \(1\le d_i\le4\), and
		\(\widehat m_i:=m_i/d_i\) has relatively prime components.  In
		particular,
		\[
		\xi_i=-\frac{\widehat m_i}{|\widehat m_i|}
		=-\frac{m_i}{|m_i|}.
		\]
		Since \(A\) is orthogonal, \(\|B\|\le\sqrt8\), and
		\(|v_i|\le\sqrt2\), the columns of \(M\) and the reverse triangle
		inequality give, for \(N\ge192\),
		\[
		N^2\le |m_i|\le2N^2.
		\]
		Since \(\widehat m_i\) is primitive, the minimal period is
		\[
		|\widehat m_i|=c_iN^2.
		\]
		As \(1\le d_i\le4\) and \(N^2\le |m_i|\le2N^2\), we have
		\[
		\frac14\le c_i=\frac{|m_i|}{d_iN^2}\le2.
		\]
		This proves the period assertion in item~\textup{(i)} without any
		parity assumption on \(N\).
		
		It remains to verify the density assertion.  Notice first that
		\(m_i\) and \(\widehat m_i\) have the same annihilator lattice.
		A direct computation gives
		\[
		\bigl(\det(v_i,Av_i,Bv_i)\bigr)_{i=1}^9
		=(1,1,1,1,1,4,-1,-1,-2),
		\]
		so \(v_i,Av_i,Bv_i\) span \(\mathbb R^3\) for every \(i\).
		If \(n\in\mathbb Z^3\) satisfies \(n\cdot m_i=0\), then
		\[
		N^2(n\cdot v_i)+N(n\cdot Av_i)+n\cdot Bv_i=0.
		\]
		The three scalar products are integers.  If \(0<|n|<N/4\) and
		\(n\cdot v_i\ne0\), then
		\[
		N^2\le N\sqrt2\,|n|+4|n|
		<\frac{\sqrt2}{4}N^2+N<N^2,
		\]
		a contradiction.  Hence \(n\cdot v_i=0\).  If
		\(n\cdot Av_i\ne0\), then
		\(N\le|n\cdot Bv_i|\le4|n|<N\), again a contradiction.
		Thus \(n\) is orthogonal to \(v_i,Av_i,Bv_i\), and hence \(n=0\).
		Consequently,
		\[
		\min\{|n|:n\in\mathbb Z^3\setminus\{0\},\ n\cdot m_i=0\}
		\ge\frac N4.
		\]
		Let \(\pi_i\) be the orthogonal projection onto \(m_i^\perp\).
		The dual of the two-dimensional lattice
		\(\pi_i(\mathbb Z^3)\) is
		\(\{n\in\mathbb Z^3:n\cdot m_i=0\}\).  The standard
		two-dimensional transference estimate for the covering radius
		therefore gives
		\[
		\sup_{x\in\mathbb T^3}
		\operatorname{dist}\bigl(x,\{s\xi_i:s\in\mathbb R\}\bigr)
		\le\frac{C_0}{N}.
		\]
		
		For a symmetric \(3\times3\) matrix \(\E\), apply
			\eqref{dec-r-bar} to \(\bar\E=K^{-1}\E K^{-T}\) and multiply on the
			left by \(K\) and on the right by \(K^T\).  This gives
		\begin{align*}
			\E = \sum_{i=1}^9 \Gamma_i(\E) \, \xi_i \otimes \xi_i,
			%\end{align}
			\qquad \mbox{where }
			%\begin{align}
			\Gamma_i(\E)\coloneqq
				|Ke_i|^2\,\bar{\Gamma}_i(K^{-1}\E K^{-T}),
		\end{align*}

		\begin{align}
			\max_{i, j, k} \left|\frac{\partial \Gamma_i}{\partial \E_{j, k}}\right| \leq 4.
			\label{eq:Gamma-derivative-bound}
		\end{align}
		Since \(K=I+O(N^{-1})\), the condition \(N\ge192\) and the choice
		\[
		\zeta=32\bigl(|\E_l|^2+\delta_{q+1}^2\bigr)^{1/2}
		\]
		ensure that
		\[
		\left\|K^{-1}\left(I-\frac{\E_l}{\zeta}\right)K^{-T}-I
		\right\|_\infty<\frac1{16}.
		\]
		Thus the argument of every \(\Gamma_i\) used below lies in the
		positivity neighborhood of the geometric decomposition.
		
		Next, we define
		\begin{align*}
			a_i(x, t) \coloneqq  \zeta(x, t) \, \Gamma_i \left(I_{3 \times 3} - \frac{1}{\zeta(x, t)} \E_{l}(x, t)\right), \qquad \text{where} \quad \zeta(x, t) \coloneqq 32 \left(|\E_l(x, t)|^2 + \delta_{q+1}^2\right)^{\frac{1}{2}}.
		\end{align*}
		It follows that
		\begin{align*}
			\sum a_i(x, t) \xi_i \otimes \xi_i = - \E_{l}(x, t) + \zeta(x, t) I_{3 \times 3}.
		\end{align*}
		Thus item~(ii) holds.  The bounds
		\(1/8\le\Gamma_i\le2\), together with
		\(\zeta\ge32\delta_{q+1}\), give \(a_i\ge4\delta_{q+1}\), and hence
		the stated lower bound.  Moreover,
		\(\zeta\le32(|\E_l|+\delta_{q+1})\); using
		\(\|\E_l\|_{L_t^\infty L_x^1}\le2\delta_{q+1}\) gives the claimed
		\(L_t^\infty L_x^1\)-estimate.  Finally, differentiating the formula
		for \(a_i\), using
		\(|\nabla_{x,t}\zeta|\le32|\nabla_{x,t}\E_l|\), and applying
		\eqref{eq:Gamma-derivative-bound} gives
		\begin{align*}
			\|a_i\|_{C^{0}_{x, t}} \leq C(\delta_{q+1}+ \|\E_l\|_{C^{0}_{x, t}}), \qquad \|a_i\|_{C^{1}_{x, t}} \leq C\bigl(\delta_{q+1}+\|\E_l\|_{C^{1}_{x, t}}\bigr).
		\end{align*}
		
	\end{proof}
	
	\subsection{Time partition and cutoff functions}

	The nine directions from Lemma~\ref{lemma:error dec} are activated
	successively, thereby avoiding interactions
	between distinct directional families.  We partition time into cells
	and divide each cell into nine subintervals, one for each direction.
	On every cell the coefficients are frozen at their temporal averages;
	the resulting discrepancy will later be included in the temporal
	error.  Smooth cutoffs make the perturbations vanish near the endpoints
	of their active intervals.
	
	Because \(\tau_{q+1}^{-1}\in\N\), set
	\(\mathcal K_{q+1}:=\mathbb Z\) and partition the time axis into cells
	\begin{align*}
		\mathcal{T}^k \coloneqq [k\tau_{q+1}, (k+1)\tau_{q+1}),
		\qquad k\in\mathcal K_{q+1}.
	\end{align*}

	Only the finitely many cells meeting \(I_\ast\) enter the construction
	on that interval.  The integrality assumption makes both \(0\) and
	\(1\) cell interfaces.

	Each interval $\mathcal{T}^k$ is further divided into nine
		subintervals of equal length:
	\begin{equation*}
		\mathcal{T}^k_i \coloneqq \left[\tau_{q+1} \left(k +\frac{i-1}{9}\right),\; \tau_{q+1} \left(k+ \frac{i}{9}\right)\right) \, ,
		\quad i=1, \dots 9\, , \quad  k\in \mathcal K_{q+1} \, .
	\end{equation*}
	Thus $\mathcal{T}^k=\bigcup_{i=1}^9\mathcal{T}^k_i$.
	For the cutoff construction, we also use the slightly shorter interval
	\begin{align}
		\bar{\mathcal{T}}^k_i \coloneqq \left[\tau_{q+1} \left(k +\frac{i-1}{9} + \frac{1}{\lambda_{q+1}}\right),\; \tau_{q+1} \left(k+ \frac{i}{9} - \frac{1}{\lambda_{q+1}}\right)\right) \, ,
		\quad i=1, 2, \dots ,9\, , \quad  k\in \mathcal K_{q+1} \, 
		\label{iteration: short tauik}.
	\end{align}

	We use the coefficients $a_i(x,t)$ from
	Lemma~\ref{lemma:error dec} and denote their time averages by
	\begin{equation*}
		a_i^k(x):= \frac{1}{\tau_{q+1}}
		\int_{\mathcal{T}^k} a_i(x,t)\, \dd t.
	\end{equation*}
	Equation~\eqref{eq:est_a} gives
	\begin{align}
		a_i^k(x) \ge \delta_{q+1} \, ,
		\quad 
		\| a_i^k \|_{L^1_x} \le 192 \delta_{q+1} \, ,
		\quad 
		\| a_i^k \|_{C^0_{x}} \leq C  \lambda^{3\frac{n}{\sigma}+3 \gamma}_{q+1}
		\, ,
		\quad 
		\| a_i^k \|_{C^1_{x}}
		\le C  \lambda^{4\frac{n}{\sigma}+4 \gamma}_{q+1}
		\, .
		\label{eq:est_aik}
	\end{align}
	Moreover,
	\begin{equation*}
		\|a_i^k-a_i\|_{L_t^\infty(\mathcal T^k;L_x^1)}
			\le \tau_{q+1}
			\|\partial_ta_i\|_{L_t^\infty(\mathcal T^k;L_x^1)}
		\le \tau_{q+1}   \lambda^{4\frac{n}{\sigma}+4 \gamma}_{q+1}.
	\end{equation*}
	
	Thus the freezing error is small when \(\tau_{q+1}\) is
		sufficiently small.
	
	\bigskip
	
	For $k\in\mathcal K_{q+1}$ and $i\in\{1,\dots,9\}$, choose a
		smooth cutoff $\zeta_i^k:\mathbb R\to[0,1]$ such that
		$\supp\zeta_i^k\subset\operatorname{int}(\mathcal T_i^k)$,
		$\zeta_i^k\equiv1$ on $\bar{\mathcal T}_i^k$, and
	\begin{equation}\label{eqn:zeta-est}
		\|\frac{\dd}{\dd t}\zeta_i^k\|_{L^\infty_t}
		\le 10 \frac{\lambda_{q+1}}{\tau_{q+1}}.
	\end{equation}
	
	\subsection{Space-dependent vortex scale}

	We couple the radius of each core to the local size of the corresponding
	frozen stress coefficient.  This choice produces the weighted orbit
	average needed to recover the tensor
	$a_i^k\xi_i\otimes\xi_i$, while $r_{q+1}$ determines the overall
	concentration scale.  The separation condition below ensures that the
	localized cores remain narrower than the spacing of the rational
	orbits.  We set
	\begin{align}
		\mathfrak R^k_i(x) \coloneqq r_{q+1} a^k_i(x) \, .
		\label{def: rki}
	\end{align}
	
	Because \(\operatorname{supp}\widetilde H_{R,e}\subset
	B_{3R^\alpha}\), we choose \(r_{q+1}\) small enough that
	\begin{equation}\label{control of r}
		6\sup_{x\in\T^3}(\mathfrak R_i^k(x))^\alpha
			\le\lambda_{q+1}^{-1},
			\qquad i=1,\dots,9,\quad k\in\mathcal K_{q+1}.
	\end{equation}

	Thus the diameter of every principal core is at most the geometric
	spacing scale \(\lambda_{q+1}^{-1}\).  In particular, the Hill source
	and the auxiliary source introduced below are contained, after a common
	translation, in a ball of radius \(\lambda_{q+1}^{-1}\).

	\subsection{Trajectory of the vortex core}\label{sub:traj cent}
	
	We next choose the amplitude and trajectory of each core.  The amplitude is
	normalized by the spatial mean of $a_i^k$, while the center moves at the
	intrinsic Hill speed in the direction $\xi_i$.  The initial point is
	selected so that the average of $a_i^k$ along one rational orbit agrees
	with its torus average.  This normalization produces the desired
	directional stress after time averaging.
	
	The cutoff used in Proposition~\ref{main prop} is the scalar
		multiple of \(\zeta_i^k\) given by
	\begin{align*}
		\eta(t) = \eta^k_i \zeta^k_i(t),
	\end{align*}
	where \(\eta_i^k\) is chosen to produce exact cancellation:
	\begin{align}\label{eqn:choiceetamagic}
		(\eta_i^k )^2 =  \frac{9}{2\pi} \int_{\mathbb{T}^3}a_i^k(x)\,\dd x
		\in [ \delta_{q+1}, 400 \delta_{q+1}].
	\end{align}
	The bounds for \(\eta_i^k\) follow from
		\eqref{eq:est_aik}.  This value is chosen now so that the averaged
		rank-one tensor has the required normalization in the cancellation
		below.

	Next, we define the trajectory of the center of the core as
	\begin{equation}\label{eq:traj1}
		\frac{\dd}{\dd t}x^k_i(t)
			=\frac{\eta_i^k\zeta_i^k(t)}{\mathfrak R_i^k(x^k_i(t))}\xi_i,
			\qquad t\in\mathcal T_i^k.
	\end{equation}
	For later use, set
	\[
	R_i^k(t):= \mathfrak R_i^k(x_i^k(t)).
	\]
	Set
		\(t_0:=\tau_{q+1}(k+(i-1)/9+\lambda_{q+1}^{-1})\) and choose
		\(x_i^k(t_0)\) so that
	\begin{equation*}
		\frac{1}{c_i\lambda^2_{q+1}}
			\int_0^{c_i\lambda_{q+1}^2}
			a_i^k(x_i^k(t_0)+s\xi_i)\,\dd s
			=\int_{\mathbb T^3}a_i^k(x)\,\dd x
			=\frac{2\pi(\eta_i^k)^2}{9}.
	\end{equation*}

	Such a point exists: the orbit-average on the left is a continuous
	function on the connected quotient of \(\mathbb T^3\) by the closed
	\(\xi_i\)-orbit, and its average over that quotient is precisely the
	spatial mean of \(a_i^k\).  It must therefore attain that mean.

	Solve the ODE~\eqref{eq:traj1} forward and backward from
		\(x_i^k(t_0)\) to obtain a trajectory on the whole interval
		\(\mathcal T_i^k\).
	We impose
	\begin{align}
		\lambda_{q+1}^3 r_{q+1}\delta_{q+1}^{1/2}
			\le \frac{\tau_{q+1}}{C_\ast},
		\label{traj: ineq for cutoff}
	\end{align}
	where $C_\ast$ is a sufficiently large absolute constant.  This condition
	is verified uniformly, including at $q=0$, in Section~\ref{sec:choose para}.

	On the plateau $\bar{\mathcal T}_i^k$ where
		$\zeta_i^k=1$, the trajectory is periodic with period
	\begin{equation}\label{eq:imp}
		T_i^k 
		= \int_0^{c_i\lambda_{q+1}^2} \frac{r_{q+1}}{\eta_i^k} a_i^k(x_i^k(t_0) + s\xi_i)\,\dd s
		= \frac{2\pi c_i\lambda_{q+1}^2 r_{q+1}\eta_i^k}{9} 
		\le Cc_i\lambda_{q+1}^2r_{q+1}\delta_{q+1}^{1/2}
			\le\frac{\tau_{q+1}}{10\lambda_{q+1}}.
	\end{equation}

	Condition~\eqref{traj: ineq for cutoff} is intended to
		guarantee that many complete periods lie inside
		\(\bar{\mathcal T}_i^k\).  Let \(M_i^k\) be the largest nonnegative
		integer such that
		\[
		t_0+M_i^kT_i^k
		\le\tau_{q+1}\left(k+\frac{i}{9}-\lambda_{q+1}^{-1}\right).
		\]
	
	By the maximality of \(M_i^k\), the completed periods differ
		from \(\tau_{q+1}/9\) by at most the two cutoff layers and one additional
		period:
	\begin{equation}\label{eq:time control}
		0\le \frac{1}{9}\tau_{q+1}-M_i^kT_i^k\le \frac{2\tau_{q+1}}{\lambda_{q+1}}+T_i^k\le C\frac{\tau_{q+1}}{\lambda_{q+1}}.
	\end{equation}
	\subsection{Estimates of the building block}
	
	We apply Corollary~\ref{cor:navier-stokes-moving-hill} with its intrinsic
	exponent set equal to \(p_0\) from \eqref{eq:intrinsic-exponents}.  We
	take profile direction
	$e=-\xi_i$, amplitude $\eta_i^k\zeta_i^k$, radius
	$\mathfrak R_i^k(x_i^k(t))$, and center $x_i^k(t)$.  The sign is determined by the
	normalization of the Hill impulse: with this choice the resulting block
	$V_i^k$ moves and has impulse in the $+\xi_i$ direction.  It satisfies

	\begin{equation}\label{block in interval}
		\partial_t V^k_i+\diver(V^k_i\otimes V^k_i)-\nu \Delta V^k_i+\nabla P^k_i
		= \frac{\dd}{\dd t}(\eta^k_i\zeta^k_i(t) R^k_i(t))
			\frac{1}{R^k_i(t)}\tilde H_{R^k_i(t),-\xi_i}
			(x-x_i^k(t))+ \diver(F^{k}_i).
	\end{equation}
	With these choices of radius and amplitude,
	Proposition~\ref{main prop} and
	Corollary~\ref{cor:navier-stokes-moving-hill} give the following estimates.
	Since \(\mathfrak R_i^k=r_{q+1}a_i^k\),
	\begin{align*}
		\|\nabla\log\mathfrak R_i^k\|_{L_x^\infty}
		&=
		\left\|\frac{\nabla a_i^k}{a_i^k}\right\|_{L_x^\infty}
		\nonumber\\
		&\le
		C\delta_{q+1}^{-1}\lambda_q^{4n+4\gamma\sigma}
		\le
		C\lambda_{q+1}^{\frac{4n}{\sigma}+5\gamma}.
	\end{align*}
	Using the bounds for \(\eta_i^k\), \(R_i^k\), and
	\(\nabla\log\mathfrak R_i^k\), we obtain
	\begin{align}
		\|F_i^k\|_{L_t^\infty L_x^1}
		&\le
		C(\eta_i^k)^2
		\|R_i^k\|_{L_t^\infty}^{\kappa}
		\left(1+\|\nabla\log\mathfrak R_i^k\|_{L_x^\infty}\right)
		\nonumber\\
		&\quad
		+C\nu\eta_i^k
		\|R_i^k\|_{L_t^\infty}^{\frac{2}{p_0}-\frac53}
		\nonumber\\
		&\le
		C\delta_{q+1}r_{q+1}^{\kappa}
		\lambda_{q+1}^{
		\kappa\left(\frac{3n}{\sigma}+3\gamma\right)
		+\frac{4n}{\sigma}+5\gamma}
		\nonumber\\
		&\quad
		+C\nu\delta_{q+1}^{1/2}
		r_{q+1}^{\frac{2}{p_0}-\frac53}
		\lambda_{q+1}^{
		\left(\frac{3n}{\sigma}+3\gamma\right)
		\left(\frac{2}{p_0}-\frac53\right)}.
		\label{error: Fki}
	\end{align}
	The velocity satisfies
	\begin{align}
		\|V_i^k\|_{L_t^\infty L_x^{3/2}}
		&\le C\eta_i^k\|R_i^k\|_{L_t^\infty}^{1/3}\le
		C\delta_{q+1}^{1/2}r_{q+1}^{1/3}
		\lambda_{q+1}^{\frac{n}{\sigma}+\gamma},
		\label{iteration: L32 Vki}
		\\
		\|V_i^k\|_{L_t^\infty L_x^2}
		&\le C\eta_i^k
		\le C\delta_{q+1}^{1/2}.
		\label{iteration: L2 Vki}
	\end{align}
	For \(1<s\le6/5\), the exponent \(2/s-5/3\) is nonnegative.
	Using the upper bound for \(R_i^k\), we obtain
	\begin{align}
		\|\nabla V_i^k\|_{L_t^\infty L_x^s}
		&\le
		C_s\eta_i^k
		\left\|(R_i^k)^{\frac{2}{s}-\frac53}\right\|_{L_t^\infty}
		\nonumber\\
		&\le
		C_s\delta_{q+1}^{1/2}
		r_{q+1}^{\frac{2}{s}-\frac53}
		\lambda_{q+1}^{
		\left(\frac{3n}{\sigma}+3\gamma\right)
		\left(\frac{2}{s}-\frac53\right)},
		\qquad 1<s\le\frac65.
		\label{iteration: Lp DVki}
	\end{align}
	For \(6/5<s<\infty\), the exponent is negative.  Using
	\(R_i^k\ge r_{q+1}\delta_{q+1}\), we instead obtain
	\begin{align}
		\|\nabla V_i^k\|_{L_t^\infty L_x^s}
		&\le
		C_s\eta_i^k
		\left\|(R_i^k)^{\frac{2}{s}-\frac53}\right\|_{L_t^\infty}
		\nonumber\\
		&\le
		C_s\delta_{q+1}^{1/2}
		r_{q+1}^{\frac{2}{s}-\frac53}
		\delta_{q+1}^{\frac{2}{s}-\frac53},
		\qquad \frac65<s<\infty.
		\label{iteration: Lp DVki special}
	\end{align}
	The following time-derivative bounds hold for every
	\(s\in(1,\infty)\).

	\begin{align}
		\|\partial_t V^k_i\|_{L^\infty_t L^s_x}  & \leq C_s (\eta^k_i)^2 \|(R^k_i)^{-3}\|_{L^\infty_t} \left( 1 + \|\nabla \mathfrak R^k_i\|_{L^\infty_x} \right) + C_s  \eta^k_i \|(\zeta^k_i)^\prime\|_{L^\infty_t} \|(R^k_i)^{-1}\|_{L^\infty_t} \nonumber \\
		& \leq C \delta_{q+1}^{-2} r_{q+1}^{-3} (1 + r_{q+1} \lambda_{q+1}^{4 \frac{n}{\sigma}+4 \gamma}  % \delta_{q+1}^{-11}\lambda_q^{24 n}
		) + C \delta_{q+1}^{-\frac{1}{2}} (\lambda_{q+1} \tau_{q+1}^{-1}) r_{q+1}^{-1} \nonumber \\
		& \leq C \delta_{q+1}^{-2} r_{q+1}^{-3} (1 + r_{q+1} \lambda_{q+1}^{4 \frac{n}{\sigma}+4 \gamma}). 
		\label{iteration: Lp part Vki}
	\end{align}

	The last inequality follows from
		\eqref{traj: ineq for cutoff}.  Similarly,
	\begin{align}
		\|\partial_t \nabla V^k_i\|_{L^\infty_t L^s_x}  & \leq C_s (\eta^k_i)^2 \|(R^k_i)^{-3-\beta}\|_{L^\infty_t} \left( 1 + \|\nabla \mathfrak R^k_i\|_{L^\infty_x} \right) + C_s  \eta^k_i \|(\zeta^k_i)^\prime\|_{L^\infty_t} \|(R^k_i)^{-\frac{5}{3}}\|_{L^\infty_t} \nonumber \\
		& \leq C \delta_{q+1}^{-2-\beta} r_{q+1}^{-3-\beta} (1 + r_{q+1} \lambda_{q+1}^{4 \frac{n}{\sigma}+4 \gamma} 
		) + C \delta_{q+1}^{-\frac{7}{6}} (\lambda_{q+1} \tau_{q+1}^{-1}) r_{q+1}^{-\frac{5}{3}} \nonumber \\
		& \le C \delta_{q+1}^{-2-\beta} r_{q+1}^{-3-\beta} (1 + r_{q+1} \lambda_{q+1}^{4 \frac{n}{\sigma}+4 \gamma} 
		) .
		\label{est:dtdx of V}
	\end{align}
	The last estimate again uses
		\eqref{traj: ineq for cutoff}. 
	
	\subsection{Auxiliary building block}
	
	The source generated by the moving Hill block has the correct spatial
	mean but remains concentrated at the core scale.  To compare its time
	average with the prescribed rank-one stress, we replace it by an
	orbit-adapted auxiliary source.  Its profile is normalized both in
	space and along the rational trajectory, so that its time average
	recovers $\diver(a_i^k\xi_i\otimes\xi_i)$ up to a small error tensor.
	The difference between the original and auxiliary sources will be
	handled by a symmetric antidivergence.  For each $i,k$, define
	\begin{equation}\label{eq:U}
		U_i^k(x,t) := \frac{\dd}{\dd t} \left( \eta_i^k \zeta_i^k(t) R_i^k(t) \right) \tilde{U}_i^k(x - x_i^k(t)) \xi_i,
		\quad t \in \mathcal{T}_i^k.
	\end{equation}
	
	Extend \(U_i^k\) by zero and set the locally finite sum
	
	\[
	U_{q+1}:=\sum_{k\in\mathcal K_{q+1}}
		\sum_{i=1}^9U_i^k.
	\]
	
	Each component \(U_i^k\) replaces the corresponding
		concentrated Hill source
	\[
	\frac{\dd}{\dd t}\left(\eta_i^k\zeta_i^k(t)R_i^k(t)\right)
	\frac{1}{R_i^k(t)}\tilde H_{R_i^k(t),-\xi_i}.
	\]
	The auxiliary profile is chosen so that the antidivergence of the
	difference between these two sources is small.  In addition,
	$\widetilde U_i^k$ has the following properties:
	
	\begin{itemize}
		\item[(i)] The profile has mass \(2\pi\), matching
			the impulse of \(R^{-1}\widetilde H_{R,-\xi_i}\):
		\begin{equation}\label{eqn:tildeu1}
			\int_{\mathbb{T}^3} \tilde{U}_i^k(x)\, \dd x = 2\pi \, ,
		\end{equation}

		\item[(ii)] For every \(x\in\mathbb T^3\), its
			full-orbit average satisfies
		\begin{equation}\label{eqn:tildeu2}
			\frac{1}{c_i\lambda^2_{q+1}}\int_0^{c_i \lambda_{q+1}^2} \tilde{U}_i^k(x - (x_i^k(t_0) + s \xi_i))\, \dd s
			= 2\pi \, ,
		\end{equation}

		\item[(iii)] $\operatorname{supp}\widetilde U_i^k
		\subset B_{C\lambda_{q+1}^{-1}}(0)$, and for every
		$p\in[1,\infty]$,
		\begin{equation}\label{eq:est tildeU}
			\| \tilde{U}_i^k \|_{L^p} \le C(p) \lambda_{q+1}^{3 - \frac{3}{p}}\,, \qquad \text{and} \qquad \| \nabla \tilde{U}_i^k \|_{L^p} \le C(p) \lambda_{q+1}^{4 - \frac{3}{p}}\, .
		\end{equation}
	\end{itemize}
	
	Only these normalization, orbit-average, and localization properties are
	used below.  For completeness, we recall the standard normalized
	periodization construction from~\cite{BrueColomboKumar2024}.  The special rational
	geodesics chosen in Lemma~\ref{lemma:error dec} are
	\(C_0\lambda_{q+1}^{-1}\)-dense, with a uniform constant \(C_0\).
	Choose a nonnegative \(\varphi\in C_c^\infty(B_{2C_0})\) that is
	strictly positive on \(\overline B_{C_0}\).  We regard the rescaled bump
	below as a function on \(\mathbb T^3\) by \(\mathbb Z^3\)-periodization and set
	\[
	\varphi_{\lambda_{q+1}}(x)
	:=\lambda_{q+1}^3\varphi(\lambda_{q+1}x).
	\]
	Define the orbit average without inserting the initial translation,
	\begin{equation}\label{eq:Phi-orbit}
		\Phi_{\lambda_{q+1}}(x)
		:=\frac{1}{c_i\lambda_{q+1}^2}
		\int_0^{c_i\lambda_{q+1}^2}
		\varphi_{\lambda_{q+1}}(x-s\xi_i)\,\dd s,
		\qquad
		\widetilde U_i^k(x)
		:=\frac{2\pi\varphi_{\lambda_{q+1}}(x)}
		{\Phi_{\lambda_{q+1}}(x)}.
	\end{equation}
	The covering property of the orbit gives
	\(C^{-1}\le\Phi_{\lambda_{q+1}}\le C\) and
	\(\|\nabla\Phi_{\lambda_{q+1}}\|_{L^\infty}\le C\lambda_{q+1}\).
	Moreover, \(\Phi_{\lambda_{q+1}}\) is invariant under translation along
	the \(\xi_i\)-orbit.  Consequently, averaging
	\(\widetilde U_i^k(x-(x_i^k(t_0)+s\xi_i))\) cancels the denominator in
	\eqref{eq:Phi-orbit} and proves \eqref{eqn:tildeu2}.  Fubini's theorem
	then yields \eqref{eqn:tildeu1}, while the quotient rule and the support
	of \(\varphi_{\lambda_{q+1}}\) give \eqref{eq:est tildeU}.

	\begin{proposition}\label{prop:auxiliary}
		Let \(U_{q+1}\) be the locally finite
		sum defined above.  Then
		
		\begin{equation}
			\label{eqn:estU}
			\|U_{q+1}\|_{L^\infty_t L^p_x} + \lambda_{q+1}^{-1} \|\nabla_x U_{q+1}\|_{L^\infty_t L^p_x}
			\leq C(p)  \lambda_{q+1}^{3 - \frac{3}{p}}
			\sup_{i,k}\|a^k_i\|_{ C^1}.
		\end{equation}
		Moreover, for every \(i\) and \(k\), there exists a smooth
			symmetric tensor \(G_i^k=G_i^k(x)\) such that
		
		\begin{equation}\label{eqn:ptauu}
			\frac{1}{\tau_{q+1}}
				\int_{\mathcal{T}^k}U_{q+1}(x,t)\,\dd t
				=\sum_{i=1}^9\diver\left(
				a_i^k(x)\xi_i\otimes\xi_i+G_i^k(x)\right),
		\end{equation}
		
		% where $G^k_i$ is a symmetric $2$-by-$2$ tensor and satisfy the following estimate:
		\begin{align}\label{eq:Gik}
			\|G_i^k\|_{L_x^1}
			\le
			C \frac{\| a_i^k \|_{C^1}}{\lambda_{q+1}} \, .
		\end{align}
	\end{proposition}
	\begin{proof}[Proof of Proposition~\ref{prop:auxiliary}]
		
		The definition \eqref{eq:U} of \(U_i^k\) and
			\eqref{eq:est tildeU} imply, for every
			\(t\in\mathcal T_i^k\),
		\begin{align}\label{est-u-proof}
			\|U_i^k(t)\|_{L^p_x} + \lambda_{q+1}^{-1} \|\nabla_x U_i^k(t)\|_{ L^p_x}
			& \le \sup_{s\in \mathcal T^k_i} \Big| \frac{\dd}{\dd s}(\eta_i^k \zeta_i^k(s) R_i^k(s))\Big|
			\Big(\|\tilde{U}_i^k\|_{L^p_x} + \lambda_{q+1}^{-1} \|\nabla \tilde{U}_i^k\|_{ L^p_x} \Big) \nonumber
			\\&
			\le C(p) \lambda_{q+1}^{3- \frac{3}{p}} \sup_{s\in \mathcal T^k_i} \Big| \frac{\dd}{\dd s}(\eta_i^k \zeta_i^k(s) R_i^k(s))\Big|
			.
		\end{align}
		Using \eqref{def: rki}, \eqref{eqn:choiceetamagic}, and
			\eqref{eqn:zeta-est}, we estimate the supremum in
			\eqref{est-u-proof}, for every \(s\in\mathcal T_i^k\), by
		\begin{align}
			\Big| \frac{\dd}{\dd s}(\eta_i^k \zeta_i^k(s) R_i^k(s))\Big| 
			&
			\leq \Big| \eta_i^k R_i^k(s) \frac{\dd\zeta_i^k}{\dd s}\Big| + \Big| (\eta^k_i \zeta^k_i)^2 \frac{\nabla R^k_i}{R^k_i} \Big| 
			\notag\\&\leq 
			C \delta_{q+1}^{\frac{1}{2}} r_{q+1} \|a^k_i\|_{L^\infty_x} \, \frac{\lambda_{q+1}}{\tau_{q+1}}
			+
			C \delta_{q+1} \|\frac{\nabla a_i^k}{a_i^k}\|_{L^\infty_x}
			\notag\\& \leq 
			C \|a^k_i\|_{ C^1_x} \, .
			\label{corrector: an int est}
		\end{align}
		For the second term we used
			\(a_i^k\ge\delta_{q+1}\), while
			\eqref{traj: ineq for cutoff} controls the first term.  Combining this
			estimate with \eqref{est-u-proof} proves \eqref{eqn:estU}.
		
		To prove \eqref{eqn:ptauu}, it suffices to establish
		\begin{align*}
			\frac{1}{\tau_{q+1}}\int_{\mathcal{T}_i^k}U_i^k(x,t)\, \dd t
			= \diver\left( a_i^k(x) \xi_i \otimes \xi_i
			+ G_i^k(x)\right) \, ,
			\quad x\in \mathbb{T}^3 \, ,
		\end{align*}
		and then sum over \(i=1,\dots,9\).  Integration by
			parts gives
		\begin{align}
			\int_{\mathcal{T}_i^k}U_i^k(x,t)\, \dd t & = \int_{\mathcal{T}_i^k}\frac{\dd}{\dd t} \left( \eta_i^k \zeta_i^k(t) R_i^k(t) \right) \tilde{U}_i^k(x - x_i^k(t)) \xi_i\, \dd t  \nonumber \\
			& = - \int_{\mathcal{T}_i^k}  \eta_i^k \zeta_i^k(t) R_i^k(t)  \frac{\dd}{\dd t} \left( \tilde{U}_i^k(x - x_i^k(t)) \right) \, \xi_i\, \dd t \nonumber \\
			& = \int_{\mathcal{T}_i^k}  \left( \eta_i^k \zeta_i^k(t) \right)^2  \diver \left( \tilde{U}_i^k(x - x_i^k(t)) \xi_i \otimes \xi_i \right)\, \dd t \nonumber \\
			& = \diver \left( \xi_i \otimes \xi_i\int_{\mathcal{T}_i^k}  \left( \eta_i^k \zeta_i^k(t) \right)^2   \tilde{U}_i^k(x - x_i^k(t)) \, \dd t \right)\, .
			\label{aux bb: int U}
		\end{align}
		
		Recall \(t_0\), \(T_i^k\), and \(M_i^k\) from
			Section~\ref{sub:traj cent}, and set
			\(\widetilde{\mathcal T}_i^k=[t_0,t_0+M_i^kT_i^k]\).
			By \eqref{iteration: short tauik},
			\(\widetilde{\mathcal T}_i^k\subseteq\bar{\mathcal T}_i^k\), and hence
		$\zeta_i^k(t)=1$ for
		$t\in\widetilde{\mathcal T}_i^k$.
		
		Next, we write
		\begin{align}
			\int_{\mathcal{T}_i^k}&  \left( \eta_i^k \zeta_i^k(t) \right)^2   \tilde{U}_i^k(x - x_i^k(t)) \, \dd t \nonumber \\
			& =\big(\eta_i^k \big)^2\int_{\tilde{\mathcal{T}}_i^k} \tilde{U}_i^k(x - x_i^k(t))  \, \dd t
			+ \big(\eta_i^k \big)^2
			\int_{\mathcal{T}_i^k\setminus \tilde{\mathcal{T}}_i^k}    \zeta_i^k(t) ^2  
			\tilde{U}_i^k(x - x_i^k(t)) \, \dd t \eqqcolon I + II \, .
			\label{aux bb: I + II}
		\end{align}
		The tensor \(II\,\xi_i\otimes\xi_i\) contributes to
			\(G_i^k\).  By \eqref{eq:est tildeU},
		\begin{equation}
			\| II \|_{L^1} 
			\le 
			(\eta_i^k)^2 \| \tilde{U}_i^k \|_{L^1} \int_{\mathcal{T}_i^k\setminus \tilde{\mathcal{T}}_i^k}
			(\zeta_i^k(t))^2\,\dd t
			\le 
			C \delta_{q+1} \tau_{q+1} \lambda_{q+1}^{-1} \, .
			\label{aux bb: L1 est II}
		\end{equation}
		Next, we investigate the main term $I$.
		Make the change of variables
			\(t=t(s)\in\widetilde{\mathcal T}_i^k\), determined by
			\(x_i^k(t(s))=x_i^k(t_0)+s\xi_i\) and \(t(0)=t_0\).  Then
		\begin{equation*}
			t'(s) = \frac{r_{q+1}}{\eta_i^k \zeta_i^k(t(s))} a_i^k(x_i^k(t_0) + s \xi_i)\, .
		\end{equation*}
		This change of variables gives
		\begin{align}
			I & = \eta_i^k r_{q+1} \int_0^{c_i M_i^k \lambda_{q+1}^2} a_i^k (x_i^k(t_0) + s\xi_i)\tilde{U}_i^k(x - (x_i^k(t_0) + s\xi_i)) \,  \dd s \nonumber \\
			& =\eta_i^k r_{q+1} a_i^k(x) \int_0^{c_i M_i^k \lambda_{q+1}^2} \tilde{U}_i^k(x - (x_i^k(t_0) + s\xi_i)) \,  \dd s \nonumber \\
			& + \eta_i^k r_{q+1} \int_0^{c_i M_i^k \lambda_{q+1}^2} (a_i^k(x_i^k(t_0) + s\xi_i) - a_i^k(x)) \tilde{U}_i^k(x - (x_i^k(t_0) + s\xi_i)) \,  \dd s \nonumber \\
			& \eqqcolon I' + II',
			\label{aux bb: I + II prime}
		\end{align}
		where $I'$ is the main term and $II'$ will be part of the error $G^k_i$.
		By \eqref{eqn:tildeu2}, we rewrite the main term as
		\begin{align}
			I'
			&=\eta_i^k r_{q+1}a_i^k(x)
			\int_0^{c_iM_i^k\lambda_{q+1}^2}
			\widetilde U_i^k(x-(x_i^k(t_0)+s\xi_i))\,\dd s\notag\\
			&=\eta_i^k r_{q+1}2\pi c_iM_i^k\lambda_{q+1}^2a_i^k(x)
			=(\tau_{q+1}+E_i^k)a_i^k(x),
			\label{aux bb: def I prime}
		\end{align}
		where $E_i^k:=9T_i^kM_i^k-\tau_{q+1}$.  By
		\eqref{eq:imp} and \eqref{eq:time control},
		\begin{equation}\label{aux bb: est E}
			|E_i^k|\le
				C\frac{\tau_{q+1}}{\lambda_{q+1}}.
		\end{equation}

		The term \(II'\) in
			\eqref{aux bb: I + II prime} satisfies
		\begin{align}
			\| II' \|_{L^1} \le C \eta_i^k r_{q+1} 2\pi c_i M_i^k \lambda^2_{q+1}  \frac{\| a_i^k \|_{C^1}}{\lambda_{q+1}} \leq C \left(\tau_{q+1} + |E_i^k|\right)\frac{\| a_i^k \|_{C^1}}{\lambda_{q+1}} \leq C \tau_{q+1}\frac{\| a_i^k \|_{C^1}}{\lambda_{q+1}}.
			\label{aux bb: L1 est II prime}
		\end{align}
		Combining \eqref{aux bb: int U},
			\eqref{aux bb: I + II}, \eqref{aux bb: I + II prime}, and
			\eqref{aux bb: def I prime}, we obtain
		\begin{align*}
			\frac{1}{\tau_{q+1}}\int_{\mathcal{T}_i^k}U_i^k(x,t)\, \dd t
			&=\diver\left(a_i^k(x)\xi_i\otimes\xi_i+G_i^k(x)\right),\\
			G_i^k(x)
			&:=\frac{1}{\tau_{q+1}}
			\left(II+II^\prime+E_i^k a_i^k(x)\right)\xi_i\otimes\xi_i.
		\end{align*}
		Finally, \eqref{aux bb: L1 est II},
			\eqref{aux bb: est E}, and \eqref{aux bb: L1 est II prime} give
		\begin{align*}
			\|G^k_i\|_{L^1} \leq C \frac{\delta_{q+1}}{\lambda_{q+1}} + C \frac{\| a_i^k \|_{C^1}}{\lambda_{q+1}} + C \frac{\| a_i^k \|_{L^1}}{\lambda_{q+1}} \leq C \frac{\| a_i^k \|_{C^1}}{\lambda_{q+1}}. 
		\end{align*}
	\end{proof}

	\subsection{The temporal corrector}
	
	The auxiliary source cancels the Reynolds stress only after averaging
	over a complete time cell.  We therefore introduce a temporal
	corrector as the primitive of its zero-mean oscillatory part.  The
	cellwise cancellation confines this primitive to an interval of length
	at most $\tau_{q+1}$, producing the small factor needed in the
	estimates, while the Leray projection preserves incompressibility.  Here
	\(\mathbb P=\operatorname{Id}-\nabla\Delta^{-1}\diver\), where
	\(\Delta^{-1}\) is the inverse Laplacian on the periodic domain
	\(\mathbb T^3\), acting on mean-zero scalar functions.
	Define $Q_{q+1}:\mathbb T^3\times I_\ast\to\mathbb R^3$ by
	\begin{align}\label{eq:Q}
		-Q_{q+1}(x,t)\coloneqq
			\mathbb P\int_{k\tau_{q+1}}^t
			\left(U_{q+1}(x,s)-\frac1{\tau_{q+1}}
			\int_{\mathcal T^k}U_{q+1}(x,z)\,\dd z\right)\dd s,
		\qquad t\in\mathcal T^k.
	\end{align}
	
	The integrand in \eqref{eq:Q} has zero integral over every complete
	cell $[k\tau_{q+1},(k+1)\tau_{q+1}]$.  Consequently, only the current
	cell contributes to the primitive, and its effective time length is at
	most $\tau_{q+1}$.
	\begin{equation}\label{eq:Qid2}
		-Q_{q+1}(x,t) = \mathbb P \int_{k \tau_{q+1}}^t \left(U_{q+1}(x,s)  - \frac{1}{\tau_{q+1}}\int_{\mathcal{T}^k} U_{q+1}(x,\tau)\dd\tau\right)\, \dd s.
	\end{equation}

	In particular, \(Q_{q+1}\) and \(\nabla Q_{q+1}\) vanish at every
	cell interface.  Their time derivatives may have bounded jumps there,
	so \(Q_{q+1}\) is continuous and piecewise smooth, in agreement with the
	admissibility convention at the beginning of this section.  Since the
	velocity itself is continuous, these jumps create no temporal Dirac
	masses in the distributional equation.

	The norms of \(Q_{q+1}\) are estimated using
		\eqref{eqn:estU}.  By \eqref{eq:Qid2} and the Calder\'on--Zygmund bounds
		for \(\mathbb P\),
	\begin{align}\label{eqn:Qest}
		\|Q_{q+1}\|_{L_t^\infty L_x^p}
		&\le C\int_{k\tau_{q+1}}^t
		\|U_{q+1}(\cdot,s)\|_{L_x^p}\,\dd s
		\le C\tau_{q+1}\|U_{q+1}\|_{L_t^\infty L_x^p},\nonumber\\
		\|\nabla Q_{q+1}\|_{L_t^\infty L_x^p}
		&\le C\int_{k\tau_{q+1}}^t
		\|\nabla U_{q+1}(\cdot,s)\|_{L_x^p}\,\dd s
		\le C\tau_{q+1}\|\nabla_xU_{q+1}\|_{L_t^\infty L_x^p}.
	\end{align}
	
	These estimates hold for every \(p\in(1,\infty)\) and
		\(t\in\mathcal T^k\).  Equation~\eqref{eq:Q} also gives
	\begin{align}
		\partial_tQ_{q+1}(x,t)
		&=-\mathbb P U_{q+1}(x,t)
		+\frac{1}{\tau_{q+1}}\int_{\mathcal T^k}
		\mathbb P U_{q+1}(x,s)\,\dd s,\nonumber\\
		\|\partial_tQ_{q+1}\|_{L_t^\infty L_x^p}
		&\le2\|\mathbb P U_{q+1}\|_{L_t^\infty L_x^p}
		\le C\|U_{q+1}\|_{L_t^\infty L_x^p}.
		\label{corrector: part Q Lp no space deriv}
	\end{align}
	
	Similarly,
	\begin{align}
		\|\partial_t \nabla Q_{q+1}\|_{L^\infty_t L^p_x}   \leq 2 \|\nabla\mathbb{P}U_{q+1}\|_{L^\infty_t L^p_x}  \leq C \|\nabla U_{q+1}\|_{L^\infty_t L^p_x}.
		\label{corrector: part Q Lp}
	\end{align}
	\subsection{Construction of the next Reynolds flow}
	
	\leavevmode\par\noindent
	
	All components of the perturbation are now available.  We add the
	moving Hill blocks and the temporal corrector to the mollified
	background flow.  Substitution into the equation produces the next
	Navier--Stokes--Reynolds system, with the remaining terms grouped into
	the linear interaction, corrector, temporal-freezing,
	source-replacement, and building-block errors.
	
	Define the velocity field \(u_{q+1}\) by
	\begin{equation}
		u_{q+1}(x,t) \coloneqq u_l(x,t) + v_{q+1}(x,t) + Q_{q+1}(x,t) \, ,
		\quad\quad 
		x\in \mathbb{T}^3\, , \, \, t\in I_\ast \, .
		\label{def: uq+1}
	\end{equation}
	Here the main velocity perturbation is
	\begin{equation*}
		v_{q+1}(x,t)
			=\sum_{k\in\mathcal K_{q+1}}\sum_{i=1}^9V_i^k(x,t).
	\end{equation*}
	
	The fields \(V_i^k\) are the adapted building blocks, and
		\(Q_{q+1}\) is the temporal corrector constructed above.
	
	{
		\begin{remark}
			For every grid index $k\in\{0,\dots,\tau_{q+1}^{-1}\}$,
			$v_{q+1}(x,k\tau_{q+1})=Q_{q+1}(x,k\tau_{q+1})=0$
			by the cutoff and \eqref{eq:Qid2}.  Thus
			\begin{equation*}
				u_{q+1}(x,k\tau_{q+1})
					=u_l(x,k\tau_{q+1}),
				\quad \text{for every grid index $k$ and $x\in \mathbb{T}^3$}\, .
			\end{equation*}
			In particular, for $j\in\{0,1\}$,
			\begin{equation}\label{est:mol-nomol}
				\|u_{q+1}(\cdot, j) - u_q(\cdot, j)\|_{L^2} 
				\le \|u_q-u_l\|_{L_t^\infty L_x^2}
				\le Cl\|(\nabla_xu_q,\partial_tu_q)\|_{L_t^\infty L_x^2}
				\le Cc_0\lambda_{q+1}^{-\gamma}
				=Cc_0\lambda_q^{-1/5}.
			\end{equation}
		\end{remark}
	}

	We define the new pressure field as
	\begin{align*}
		\nabla p_{q+1}\coloneqq{}&\nabla p_l-\nabla\zeta
		+\sum_{k\in\mathcal K_{q+1}}\sum_{i=1}^9\nabla P_i^k\nonumber\\
		&-(\Id-\mathbb P)\left(U_{q+1}(x,t)
		-\frac1{\tau_{q+1}}\int_{\mathcal T^k}
		U_{q+1}(x,\tau)\,\dd\tau\right),
		\qquad t\in\mathcal T^k.
	\end{align*}
	
	Here \(p_l\), \(\zeta\), and \(P_i^k\) are defined in
		\eqref{mollified eqn}, \eqref{eq:error dec}, and
		\eqref{block in interval}, respectively.  The last term is the gradient
		part of the temporal-corrector source. The locally
		finite pressure sum is defined up to a function of time.  The fields
		\(u_{q+1}\) and \(p_{q+1}\) satisfy the
		Navier--Stokes--Reynolds system distributionally with stress
	\begin{equation}\label{eq:Rq+1 dec}
		\E_{q+1} \coloneqq 
		\E_{q+1}^{(l)} 
		+
		\E_{q+1}^{(c)}
		+
		\E_{q+1}^{(t)}
		+
		\E_{q+1}^{(s)}
		+
		\sum_{i=1}^{9}(F_i^k+G_i^k),
			\qquad t\in\mathcal T^k.
	\end{equation}

	Here \(F_i^k\) is defined in \eqref{block in interval} and
		\(G_i^k\) in Proposition~\ref{prop:auxiliary}.  The stress
		\(\E_{q+1}^{(l)}\) contains terms linear in the principal perturbation,
		whereas \(\E_{q+1}^{(c)}\) contains the temporal-corrector terms:
	
	\begin{align*}
		\E_{q+1}^{(l)} 
		& \coloneqq v_{q+1} \otimes u_l
		+ u_l \otimes v_{q+1} \, , \\
		\E_{q+1}^{(c)} & \coloneqq  Q_{q+1} \otimes (u_l + v_{q+1}) + (u_l + v_{q+1}) \otimes Q_{q+1}  
		+ Q_{q+1} \otimes Q_{q+1} -\nu (\nabla Q_{q+1}+(\nabla Q_{q+1})^T)\, .       
	\end{align*}
	Finally, for \(t\in\mathcal T^k\), define
		\(\E_{q+1}^{(t)}\) as the coefficient-freezing error and
		\(\E_{q+1}^{(s)}\) as the source-replacement error:
	\begin{align*}
		\E_{q+1}^{(t)}(x,t)
		& \coloneqq \sum_{i=1}^9(a_i^k(x)-a_i(x,t))\xi_i\otimes\xi_i,
		\qquad t\in\mathcal T^k,
	\end{align*}
	\begin{align*}
		S_i^k(x,t)
		&\coloneqq\left[
		\frac{1}{R_i^k(t)}
		\widetilde H_{R_i^k(t),-\xi_i}(\,x-x_i^k(t))
		-\widetilde U_i^k(\,x-x_i^k(t))\xi_i
		\right]\nonumber\\
		&\quad\times
		\frac{\dd}{\dd t}\bigl(\eta_i^k\zeta_i^k(t)R_i^k(t)\bigr),\\
		\E_{q+1}^{(s)}(x,t)
		&\coloneqq\sum_{i=1}^9\mathcal R(S_i^k(\cdot,t))(x),
		\qquad t\in\mathcal T^k.
	\end{align*}
	The impulse identity for the Hill source and \eqref{eqn:tildeu1} give
	\[
	\int_{\mathbb T^3}S_i^k(x,t)\,\dd x=0,
	\]
	so the standard symmetric inverse-divergence \(\mathcal R\) on the torus
	is applicable.  By \eqref{control of r} and property~\textup{(iii)} of
	\(\widetilde U_i^k\), both terms are supported in a
	common ball of radius \(C\lambda_{q+1}^{-1}\).  Poincar\'e's inequality on
	that ball gives
	\(\|S_i^k\|_{W^{-1,p}}\le
	C_p\lambda_{q+1}^{-1}\|S_i^k\|_{L^p}\); the order-\(-1\) estimate for
	\(\mathcal R\) therefore yields, for \(1<p<\infty\),
	\[
	\diver\mathcal R S_i^k=S_i^k,\qquad
	\|\mathcal R S_i^k\|_{L^p}
	\le C_p\lambda_{q+1}^{-1}\|S_i^k\|_{L^p}.
	\]
	Finally, substituting \eqref{block in interval}, \eqref{eq:Q},
	\eqref{eqn:ptauu}, and \eqref{eq:error dec} into
	\eqref{def: uq+1} gives \eqref{eq:NSR} with the stress
	\eqref{eq:Rq+1 dec}; the signs of
	\(\E_{q+1}^{(t)}\) and of the pressure correction
	\(-\nabla\zeta\) are exactly those dictated by
	\eqref{eq:error dec}.

	\section{Proof of Proposition~\ref{iteration}}
	
	We verify the conclusions of Proposition~\ref{iteration}.  We first control the
	velocity increment and its space--time regularity, then estimate each
	component of the new Reynolds stress, and finally choose the parameters
	so that all the resulting inequalities hold simultaneously.
	
	\subsection{Velocity estimates}
	
	We begin with the bounds for the new velocity.  The $L^2$ estimate
	controls the size of the increment, the
	$C_t^{\alpha_0}L_x^{\bar p}$ estimate propagates the target gradient
	regularity, and the derivative estimates preserve the inductive $L^4$
	bound.  Each estimate yields an explicit restriction on the iteration
	parameters.  By \eqref{est:mol-nomol},
	\eqref{iteration: L2 Vki}, and \eqref{eqn:Qest},

	At each time, at most one pair \((i,k)\) is active.  Hence the norm of
	the sum of moving blocks is controlled by the supremum over \((i,k)\),
	as used below.

	\begin{align*}
		\|u_{q+1}-u_{q}\|_{L_t^\infty L_x^2}
		&\le \|u_l-u_q\|_{L_t^\infty L_x^2}
		+\sup_{i,k}\|V_i^k\|_{L_t^\infty L_x^2}
		+\|Q_{q+1}\|_{L_t^\infty L_x^2}\\
		&\le C\lambda_{q+1}^{-\gamma}+C\delta_{q+1}^{\frac{1}{2}}
		+C\tau_{q+1}\|U_{q+1}\|_{L_t^\infty L_x^2}\\
		&\le C\lambda_{q+1}^{-\gamma}
		+C\lambda_1^\gamma\lambda_{q+1}^{-\frac{\gamma}{2}}
		+C\tau_{q+1}\lambda_{q+1}^{\frac{3}{2}+4\frac{n}{\sigma}+4\gamma}.
	\end{align*}
	To obtain
		\(\|u_{q+1}-u_q\|_{L_t^\infty L_x^2}
		\le C\delta_{q+1}^{1/2}
		=C\lambda_1^\gamma\lambda_{q+1}^{-\gamma/2}\), it is
		sufficient that
	\begin{equation}\label{res:para1}
		\lambda_{q+1}^{-\eta+\frac{3}{2}+4\gamma+4\frac{n}{\sigma}}
		\le \lambda_{q+1}^{-\frac{\gamma}{2}}
		\le\delta_{q+1}^{1/2}.
	\end{equation}
	Moreover,
	\begin{equation*}
		\|u_{q+1}\|_{L_t^\infty L_x^2}
		\le 2\delta_0^{\frac{1}{2}}-\delta_{q}^{\frac{1}{2}}
		+\|u_{q+1}-u_{q}\|_{L_t^\infty L_x^2}
		\le 2\delta_0^{\frac{1}{2}}-\delta_{q}^{\frac{1}{2}}
		+C\delta_{q+1}^{\frac{1}{2}}
		\le 2\delta_0^{\frac{1}{2}}-\delta_{q+1}^{\frac{1}{2}},
	\end{equation*}
	provided that \(\lambda_0\) is sufficiently large.

	\begin{remark}
		For \(2<s<4\), the same argument gives the following
			\(L_t^\infty L_x^s\)-bound:
		\begin{align*}
			\|u_{q+1}-u_{q}\|_{L_t^\infty L_x^s}
			&\le \|u_l-u_q\|_{L_t^\infty L_x^s}
			+\sup_{i,k}\|V_i^k\|_{L_t^\infty L_x^s}
			+\|Q_{q+1}\|_{L_t^\infty L_x^s}\\
			&\le Cl\|(\nabla_xu_q,\partial_tu_q)\|_{L_t^\infty L_x^4}
			+C\delta_{q+1}^{\frac{1}{2}}\|R^{\frac{2}{s}-1}\|_{L_t^\infty}
			+C\tau_{q+1}\|U_{q+1}\|_{L_t^\infty L_x^s}\\
			&\le C\lambda_{q+1}^{-\gamma}
			+C\lambda_1^{(\frac4s-1)\gamma}
			\lambda_{q+1}^{-\frac{\gamma}{2}
			+(1-\frac{2}{s})(\gamma+\mu)}
			+C\lambda_{q+1}^{-\eta+3-\frac{3}{s}
			+4\frac{n}{\sigma}+4\gamma}.
		\end{align*}
		Since \(I_\ast\) has finite length, this also implies, for every
		time exponent \(1\le r\le\infty\),
		\[
			\|u_{q+1}-u_q\|_{L_t^r(I_\ast;L_x^s)}
			\le |I_\ast|^{1/r}
			\|u_{q+1}-u_q\|_{L_t^\infty(I_\ast;L_x^s)}.
		\]
		This formulation avoids incorrectly replacing the \(L_t^r\)-norm of
		the sum over all time cells by the supremum of the individual
		cellwise \(L_t^r\)-norms.
	\end{remark}

	The perturbation also preserves the zero-mean condition.  Indeed,
	Proposition~\ref{main prop} and the choice \(e=-\xi_i\) give
	\[
		\int_{\T^3}V_i^k(x,t)\,\dd x
		=2\pi\eta_i^k\zeta_i^k(t)R_i^k(t)\xi_i.
	\]
	On the other hand, \eqref{eqn:tildeu1} and \eqref{eq:U} show that
	\(\int_{\T^3}U_i^k\,\dd x
	=2\pi\frac{\dd}{\dd t}(\eta_i^k\zeta_i^kR_i^k)\xi_i\).
	The cell average of this spatial mean vanishes because \(\zeta_i^k\)
	vanishes at the endpoints of its support.  Taking the spatial mean in
	\eqref{eq:Q} therefore yields
	\[
		\int_{\T^3}Q_{q+1}(x,t)\,\dd x
		=-2\pi\eta_i^k\zeta_i^k(t)R_i^k(t)\xi_i
	\]
	on the active subinterval, and zero elsewhere.  Thus
	\(v_{q+1}+Q_{q+1}\) has zero spatial mean; since mollification preserves
	the mean, so does \(u_{q+1}\).

	\bigskip
	We next estimate \(\nabla u_{q+1}\).  By
		\eqref{iteration: Lp DVki} and \eqref{eqn:Qest},
	\begin{align*}
		\|\nabla u_{q+1}\|_{L_t^\infty L_x^{6/5}}
		&\le \|\nabla u_l\|_{L_t^\infty L_x^{6/5}}
		+\sup_{i,k}\|\nabla V_i^k\|_{L_t^\infty L_x^{6/5}}
		+\|\nabla Q_{q+1}\|_{L_t^\infty L_x^{6/5}}\\
		&\le \|\nabla u_q\|_{L_t^\infty L_x^{6/5}}
		+C \delta_{q+1}^{\frac{1}{2}}
		+C\tau_{q+1}\lambda_{q+1}^{\frac{3}{2}+4\frac{n}{\sigma}+4\gamma} \, .
	\end{align*}
	The desired gain by \(C\delta_{q+1}^{\epsilon}\) follows
		provided that
	\begin{equation}\label{res:param 2}
		-\eta+\frac{3}{2}+4\frac{n}{\sigma}+4\gamma<0.
	\end{equation}
	\begin{remark}
		For an exponent slightly larger than \(6/5\), the
			power of \(R_i^k\) is negative; hence one must use
			\eqref{iteration: Lp DVki special}, rather than
			\eqref{iteration: Lp DVki}.  Because \(r_{q+1}\) is much smaller
			than \(\delta_{q+1}\), this loss restricts the attainable exponent.
			More precisely,
		\begin{align*}
			\|\nabla u_{q+1}\|_{L_t^\infty L_x^{\bar p}}
			&\le \|\nabla u_l\|_{L_t^\infty L_x^{\bar p}}
			+\sup_{i,k}\|\nabla V_i^k\|_{L_t^\infty L_x^{\bar p}}
			+\|\nabla Q_{q+1}\|_{L_t^\infty L_x^{\bar p}}\\
			&\le \|\nabla u_q\|_{L_t^\infty L_x^{\bar p}}
			+C \delta_{q+1}^{\frac{1}{2}}r_{q+1}^{\frac{2}{\bar p} - \frac{5}{3}}
			\delta_{q+1}^{\frac{2}{\bar p} - \frac{5}{3}}
			+C\tau_{q+1}\lambda_{q+1}^{4-\frac{3}{\bar p}+4\frac{n}{\sigma}+4\gamma} \, .
		\end{align*}

		A gain by \(\delta_{q+1}^{\epsilon}\) follows if
		\begin{equation}\label{res:extra 1}
			-\frac{\gamma}{2} +( \frac{5}{3}-\frac{2}{\bar p})(\mu+\gamma)<0,\quad -\eta +4-\frac{3}{\bar p}+4\frac{n}{\sigma}+4\gamma<0.
		\end{equation}
	\end{remark}

	Taking \(s=4\) instead gives
	\begin{equation*}
		\|\nabla u_{q+1}\|_{L_t^\infty L_x^{4}}\le \|\nabla u_q\|_{L_t^\infty L_x^{4}}
		+C \delta_{q+1}^{\frac{1}{2}}r_{q+1}^{-\frac{7}{6}}
		\delta_{q+1}^{-\frac{7}{6}}
		+C\tau_{q+1}\lambda_{q+1}^{\frac{13}{4}+4\frac{n}{\sigma}+4\gamma} \, .
	\end{equation*}
	To keep the induction assumption $\|\nabla u_{q+1}\|_{L_t^\infty L_x^{4}}\le \lambda_{q+1}^n$, we need
	\begin{equation}\label{res:forgot1}
		-\frac{\gamma}{2} + \frac{7}{6}(\mu+\gamma)<n,\quad
		-\eta +\frac{13}{4}+4\frac{n}{\sigma}+4\gamma<n.
	\end{equation}

	\bigskip

	We now justify the H\"older continuity in time without hiding any exponent
	loss.  Set
	\begin{align*}
		D_{q+1}
		&:=\delta_{q+1}^{-2-\beta}r_{q+1}^{-3-\beta}
		\left(1+r_{q+1}\lambda_{q+1}^{4n/\sigma+4\gamma}\right),\\
		B&:=\max\left\{
		(2+\beta)\gamma+(3+\beta)\mu,\,
		(2+\beta)\gamma+(2+\beta)\mu
		+\frac{4n}{\sigma}+4\gamma
		\right\}.
	\end{align*}
	Thus \(D_{q+1}\le C\lambda_{q+1}^{B}\).  For every Banach-valued
	function \(f\) that is continuous and piecewise \(C^1\),
	\[
		[f]_{C_t^{\alpha_0}}
		\le C\|f\|_{L_t^\infty}^{1-\alpha_0}
		\|\partial_tf\|_{L_t^\infty}^{\alpha_0}.
	\]
	This estimate applies globally here: the supports of the \(V_i^k\) are
	disjoint in time, each \(\nabla V_i^k\) vanishes at the endpoints of its
	support, and \(\nabla Q_{q+1}\) vanishes at the cell interfaces.  Using
	\eqref{iteration: Lp DVki}, \eqref{est:dtdx of V},
	\eqref{eqn:Qest}, and \eqref{corrector: part Q Lp}, we obtain
	\begin{align*}
		[\nabla(v_{q+1}+Q_{q+1})]_{C_t^{\alpha_0}L_x^{6/5}}
		&\le C\delta_{q+1}^{\frac12(1-\alpha_0)}D_{q+1}^{\alpha_0}
		+C\tau_{q+1}^{1-\alpha_0}
		\lambda_{q+1}^{\frac32+4n/\sigma+4\gamma}\\
		&\le C\lambda_1^{\gamma(1-\alpha_0)}
		\lambda_{q+1}^{-\frac\gamma2(1-\alpha_0)+B\alpha_0}
		+C\lambda_{q+1}^{-\eta+\frac32+4n/\sigma+4\gamma
		+\eta\alpha_0}.
	\end{align*}
	Both exponents are negative when \(\alpha_0>0\) is sufficiently small,
	by \eqref{res:param 2}.

	For the target exponent \(\bar p\), set
	$\kappa_{\bar p}:=\frac53-\frac2{\bar p}$ and introduce the two positive
	margins
	\[
		\theta_V:=\frac\gamma2-
		\kappa_{\bar p}(\mu+\gamma),
		\qquad
		\theta_Q:=\eta-4+\frac3{\bar p}
		-\frac{4n}{\sigma}-4\gamma.
	\]
	Their positivity is precisely \eqref{res:extra 1}.  The same interpolation
	then gives
	\begin{align*}
		[\nabla(v_{q+1}+Q_{q+1})]_{C_t^{\alpha_0}L_x^{\bar p}}
		&\le
		C\lambda_1^{(1-2\kappa_{\bar p})\gamma(1-\alpha_0)}
		\lambda_{q+1}^{-\theta_V(1-\alpha_0)+B\alpha_0}
		+C\lambda_{q+1}^{-\theta_Q+\eta\alpha_0}.
	\end{align*}
	The displayed powers of $\lambda_1$ are retained explicitly.  Once the
	base frequency is fixed they are independent of $q$; the first stage is a
	single finite term, and the strict negative powers of $\lambda_{q+1}$
	control the summability of all later stages.
	Choose \(\alpha_0>0\) so small that all four exponents above are
	strictly negative.  Since \(\delta_{q+1}\) is a fixed multiple of
	\(\lambda_{q+1}^{-\gamma}\), there is then a \(\vartheta>0\), independent
	of \(q\), such that the preceding bound and the already established
	\(L_t^\infty L_x^{\bar p}\) bound yield
	\[
		\|\nabla u_{q+1}\|_{C_t^{\alpha_0}L_x^{\bar p}}
		\le \|\nabla u_q\|_{C_t^{\alpha_0}L_x^{\bar p}}
		+C_{\lambda_1}\delta_{q+1}^{\vartheta}.
	\]
	Here \(C_{\lambda_1}<\infty\) may depend on the fixed base frequency
	\(\lambda_1\), but is independent of \(q\).  This dependence is harmless
	because \(\lambda_1\) is fixed throughout the iteration.

	Finally, \eqref{corrector: part Q Lp no space deriv} and
		\eqref{iteration: Lp part Vki} give the time-derivative estimate
	\begin{align*}
		\|\partial_t u_{q+1}\|_{L^\infty_tL^p_x} &\le \|\partial_t u_{l}\|_{L^\infty_tL^p_x}+ \sup_{i,k}\|\partial_t V_i^k\|_{L^\infty_tL^p_x}+\|\partial_t Q_{q+1}\|_{L^\infty_tL^p_x}\\
		&\le \lambda_q^n+ C\delta_{q+1}^{-2} r_{q+1}^{-3} (1 + r_{q+1} \lambda_{q+1}^{4 \frac{n}{\sigma}+4 \gamma})+C\lambda_{q+1}^{4 \frac{n}{\sigma}+4 \gamma+3-\frac{3}{p}}.
	\end{align*}
	
	We apply this general estimate with $p=4$.
	The bound
		\(\|\partial_tu_{q+1}\|_{L_t^\infty L_x^4}\le\lambda_{q+1}^n\)
		therefore requires
	\begin{equation}\label{res:para3}
		4 \frac{n}{\sigma}+4 \gamma+3-\frac{3}{4}<n ,\quad 3\mu+2\gamma<n,\quad 2\mu+2\gamma+4 \frac{n}{\sigma}+4 \gamma<n.
	\end{equation}

	\subsection{Reynolds-stress estimates}
	
	We estimate the components of the Reynolds stress defined in the
	preceding section.  Each component will be bounded by a fixed fraction
	of the target level $\delta_{q+2}$.  We treat successively the linear
	interaction, the temporal corrector, the freezing and
	source-replacement errors, and the intrinsic errors of the moving
	blocks.
	
	By \eqref{iteration: L32 Vki}, Poincar\'e's inequality, and
		H\"older's inequality,
	\begin{align*}
		\|\E^{(l)}_{q+1}\|_{L^\infty_t L^1_x}&\le C\|v_{q+1}\|_{L^\frac{3}{2}}\|u_l\|_{L^3} \le C\|v_{q+1}\|_{L^\frac{3}{2}}\|\nabla u_l\|_{L^4}\\
		&  \leq  C \delta_{q+1}^{\frac{1}{2}} r_{q+1}^{\frac{1}{3}}\left(\lambda_{q+1}^{\frac{n}{\sigma}+ \gamma} \right)\lambda_{q}^n.
	\end{align*}
	When this and the other terms containing $\delta_{q+1}^{1/2}$ are
	compared with $\delta_{q+2}$, we use the exact ratio
	\[
		\frac{\delta_{q+1}^{1/2}}{\delta_{q+2}}
		=\lambda_1^{-\gamma}
		\lambda_{q+1}^{\gamma\sigma-\gamma/2}.
	\]
	Thus the retained factor $\lambda_1^{-\gamma}$ is favorable, and the
	displayed exponent inequalities remain sufficient.
	After choosing \(\lambda_0\) sufficiently large to absorb
		the fixed constant, the bound
		\(\|\E^{(l)}_{q+1}\|_{L_t^\infty L_x^1}
		\le\delta_{q+2}/100\) follows from the exponent inequality
	
	\begin{equation}\label{res:para4}
		\lambda_{q+1}^{2\frac{n}{\sigma}+ \gamma-\frac{\mu}{3}-\frac{\gamma}{2}} <\lambda_{q+1}^{-\gamma\sigma}.
	\end{equation}
	For the corrector stress, we estimate
	\begin{align*}
		\|\E_{q+1}^{(c)}\|_{L_t^\infty L_x^1}
		& \le C\|Q_{q+1}\|_{L_t^\infty L_x^2}
		\bigl(\|u_l\|_{L_t^\infty L_x^2}
		+\|v_{q+1}\|_{L_t^\infty L_x^2}
		+\|Q_{q+1}\|_{L_t^\infty L_x^2}\bigr)
		+C\|\nabla Q_{q+1}\|_{L_t^\infty L_x^{6/5}}\\
		&\le C\tau_{q+1}\lambda_{q+1}^{\frac{3}{2}+4\frac{n}{\sigma}+4\gamma}\delta_{0}^{\frac{1}{2}}
		{\le \frac{1}{100}\delta_{q+2}}.
	\end{align*}
	It is therefore sufficient that
	\begin{equation}\label{res:para5}
		\lambda_{q+1}^{4\frac{n}{\sigma}+ 4\gamma+\frac{3}{2}-\eta} <\lambda_{q+1}^{-\gamma\sigma}.
	\end{equation}
	For the coefficient-freezing error,
	\begin{align*}
		\|\E_{q+1}^{(t)}\|_{L_t^\infty L_x^1}
		& \le C\tau_{q+1}\|a_i\|_{\dot C^1}
		\le C\tau_{q+1}\lambda_{q+1}^{4\frac{n}{\sigma}+4\gamma}
		{\le \frac{1}{100}\delta_{q+2}},
	\end{align*}
	where the last inequality follows from
		\eqref{res:para5}.
	
	Choose an exponent $p_s>1$ for the source-replacement estimate,
	independently of the exponent used in Proposition~\ref{main prop}.
	From Proposition~\ref{main prop}, \eqref{corrector: an int est}, and
	\eqref{eq:est tildeU},
	\begin{align*}
		\|\E_{q+1}^{(s)}\|_{L^\infty_t L^{p_s}_x} &\leq C(p_s) \lambda_{q+1}^{-1}
		\left( \| \frac{1}{R_i^k}\tilde H_{R_i^k}\|_{L^\infty_t L^{p_s}_x} + \|\tilde U_i^k\|_{L^\infty_tL^{p_s}_x}\right) 
		\sup_{i,k}\Big| \frac{\dd}{\dd t}(\eta_i^k \zeta_i^k R_i^k)\Big|
		\\
		&\leq C(p_s) \lambda_{q+1}^{-1}(\delta_{q+1}^{\frac{2}{p_s} - 2} r_{q+1}^{{\frac{2}{p_s} - 2}}+\lambda_{q+1}^{3-\frac{3}{p_s}} )\lambda^{4\frac{n}{\sigma}+4 \gamma}_{q+1}
		{\leq \frac{\delta_{q+2}}{100}}.
	\end{align*}
	Since $L^{p_s}(\T^3)\hookrightarrow L^1(\T^3)$, the preceding bound
	controls the required $L^1$ stress.  It closes provided
	\begin{align}\label{res:para6}
		-1+\frac{4n}{\sigma}+4\gamma
		+(\mu+\gamma)\left(2-\frac2{p_s}\right)
		&< -\gamma\sigma,
		&
		-1+\frac{4n}{\sigma}+4\gamma+3-\frac3{p_s}
		&< -\gamma\sigma.
	\end{align}
	Both inequalities hold by taking $p_s>1$ sufficiently close to one.
	It remains to use \eqref{error: Fki} and
		\eqref{eq:Gik} for the building-block and averaging errors:
	\begin{align}\label{res:para7}
		\|F_i^k\|_{L_t^\infty L_x^1}&\le C \delta_{q+1} \, r_{q+1}^{\kappa} \, (\lambda_{q+1}^{3 \frac{n}{\sigma}+3 \gamma})^{\kappa} \left(\lambda_{q+1}^{4 n/\sigma+5 \gamma} \right)
		+C\nu\delta^{\frac{1}{2}}_{q+1}r_{q+1}^{\frac{2}{p_0}-\frac{5}{3}}(\lambda_{q+1}^{3 \frac{n}{\sigma}+3 \gamma})^{\frac{2}{p_0}-\frac{5}{3}}
		\le \frac{\delta_{q+2}}{100}, \\
		\|G_i^k\|_{L_x^1}&\le 
		C \frac{\| a_i^k \|_{C^1}}{\lambda_{q+1}}
		\le C\lambda_{q+1}^{4 \frac{n}{\sigma}+4 \gamma-1}
		\le \frac{\delta_{q+2}}{100}  \, .\notag
	\end{align}
	At a fixed time there are four global error tensors, at most one active
	tensor $F_i^k$, and nine static tensors $G_i^k$ on the current time cell.
	Consequently, after
	increasing \(\lambda_0\) once more to absorb all fixed constants,
	\[
		\|\E_{q+1}\|_{L_t^\infty L_x^1}
		\le \frac{14}{100}\,\delta_{q+2}
		<\delta_{q+2}.
	\]
	\subsection{Choice of parameters}\label{sec:choose para}
	
	We collect the restrictions obtained above and exhibit one choice for
	which they are all strict.  We distinguish the intrinsic block exponent
	$p_0$, the source-replacement
	exponent $p_s$, and the final regularity exponent $\bar p$, since these
	exponents enter different estimates.  The geometric condition
	\eqref{control of r} gives
	\begin{equation*}
		1-\alpha\mu+3\alpha\left(\frac{n}{\sigma}+\gamma\right)<0.
	\end{equation*}
	For \eqref{traj: ineq for cutoff}, the fixed prefactor in
	$\delta_{q+1}^{1/2}=\lambda_1^\gamma\lambda_{q+1}^{-\gamma/2}$ must also
	be retained.  Indeed,
	\begin{equation*}
		\frac{\lambda_{q+1}^3r_{q+1}\delta_{q+1}^{1/2}}
		{\tau_{q+1}}
		=\lambda_1^\gamma
		\lambda_{q+1}^{3-\mu-\gamma/2+\eta}.
	\end{equation*}
	Thus the stronger first-stage condition
	\begin{equation*}
		3-\mu+\frac\gamma2<-\eta
	\end{equation*}
	is sufficient uniformly for every $q\ge0$ after increasing $\lambda_1$.
	We next collect \eqref{res:para1}, \eqref{res:param 2},\eqref{res:forgot1},
	\eqref{res:para3}--\eqref{res:para7}.  Since $\sigma$ is a positive
	integer, \eqref{res:para5} implies \eqref{res:para1} and
	\eqref{res:param 2}.  The remaining restrictions reduce to
	\begin{align*}
		4 \frac{n}{\sigma}+4 \gamma+3-\frac{3}{4}<n ,\quad 3\mu+2\gamma<n,\quad 2\mu+2\gamma+4 \frac{n}{\sigma}+4 \gamma<n,\\
		2\frac{n}{\sigma}-\frac{\mu}{3}+\frac{\gamma}{2} <-\gamma\sigma,\quad  4\frac{n}{\sigma}+ 4\gamma+\frac{3}{2}-\eta <-\gamma\sigma, \quad  4\frac{n}{\sigma}+ 4\gamma-1 <-\gamma\sigma,\\
		-\mu\kappa+(3\frac{n}{\sigma}+ 3\gamma)\kappa+4\frac{n}{\sigma}+ 4\gamma <-\gamma\sigma,\quad -\frac{\gamma}{2}+(3\frac{n}{\sigma}+ 3\gamma-\mu)(\frac{2}{p_0}-\frac{5}{3})<-\gamma\sigma,\\
		-\frac{\gamma}{2} + \frac{7}{6}(\mu+\gamma)<n,\quad -\eta +\frac{13}{4}+4\frac{n}{\sigma}+4\gamma<n.
	\end{align*}
	These conditions are supplemented by the two source-exponent
	inequalities in \eqref{res:para6}.  We first choose
	$\sigma = 200$, $n=20$, $\gamma= \frac{1}{1000}$,
	$\mu = 6+2\gamma=6.002$, and
	$\eta = 3$.  For the intrinsic block estimates, set
	$p_0=\frac{28}{27}$ and $\alpha = \frac{2}{11}$.  Then
	\[
	\kappa = \min\left\{\frac{2+3\alpha}{2p_0}-4\alpha,
	2-6\alpha+\frac{3\alpha}{p_0},
	\frac{2}{p_0}-\frac{4}{3}\right\}=\frac12,
	\]
	and we choose
	$\beta = \frac{14}{3}+\frac{3\alpha-2}{p_0}-6\alpha
	=\frac{502}{231}$.  Independently, take $p_s=\frac{101}{100}$ in
	\eqref{res:para6}.  The two geometric left-hand sides are
	approximately \(-0.03618\) and \(-3.0015<-\eta\).  Moreover,
	\[
		\lambda_1^\gamma
		\lambda_{q+1}^{3-\mu-\gamma/2+\eta}
		=\lambda_1^\gamma\lambda_{q+1}^{-5\gamma/2}
		\le \lambda_1^{-3\gamma/2},
	\]
	which verifies \eqref{traj: ineq for cutoff}, also for $q=0$, once
	$\lambda_1$ is sufficiently large.  The inequalities that become less
	favorable when $\mu$ is increased have the values
	\[
		3\mu+2\gamma=18.008<20,
		\qquad
		2\mu+2\gamma+4\frac n\sigma+4\gamma=12.410<20,
		\qquad
		-\frac\gamma2+\frac76(\mu+\gamma)=7.003<20.
	\]
	The two source-exponent left-hand sides in \eqref{res:para6} are
	approximately \(-0.47713\) and \(-0.56630\), both below
	$-\gamma\sigma=-0.2$.  The bounds in \eqref{res:para4},
	\eqref{res:para5}, and the \(G_i^k\)-part of \eqref{res:para7} have
	exponents \(-1.80017\), \(-1.096\), and \(-0.596\), respectively.
	The two $\mu$-dependent exponents in the $F_i^k$-part of
	\eqref{res:para7} are \(-2.44550\) and \(-1.49310\).  Hence all of
	these exponents are strictly below \(-\gamma\sigma=-0.2\), and direct
	substitution verifies the remaining strict inequalities as well.
	
	For the final exponent \(\bar p>6/5\), the gradient estimates impose
	the additional restrictions in \eqref{res:extra 1}, namely
	\begin{equation*}
		-\frac{\gamma}{2}
		+\left(\frac{5}{3}-\frac{2}{\bar p}\right)(\mu+\gamma)<0,
		\qquad
		-\eta+4-\frac{3}{\bar p}
		+4\frac{n}{\sigma}+4\gamma<0.
	\end{equation*}
	For
	\(\sigma=200\), \(n=20\), \(\gamma=10^{-3}\),
	\(\mu=6+2\gamma=6.002\), and
	\(\eta=3\), both inequalities are strict when
	\(\bar p=6/5+5\times10^{-5}\): their left-hand sides are approximately
	\(-8.31424\times10^{-5}\) and \(-1.09590\), respectively.  In
	particular, the positive interpolation margins remain available, and the
	corresponding time H\"older exponent may still be chosen sufficiently small.

	\section{Proof of the main theorem}
	We first construct a smooth Reynolds solution on the fixed interval
	\(I_\ast\supset[0,1]\) that approximates the two prescribed endpoint
	fields.  Consider
	\begin{equation}\label{eq:NSR_0}
		\partial_tu_0+\diver(u_0\otimes u_0)-\nu\Delta u_0+\nabla p_0
		=\diver\E_0,
		\qquad \diver u_0=0,
	\end{equation}
	where \(\E_0\) is symmetric.
	
	Fix \(u^{(0)},u^{(1)}\in L^2_\sigma(\T^3)\), and let
	$\rho_\ell$ be a standard spatial mollifier.  Choose
	$\chi\in C^\infty(I_\ast;[0,1])$ such that
	\[
	\chi\equiv1\quad\hbox{on }I_\ast\cap(-\infty,1/4],
	\qquad
	\chi\equiv0\quad\hbox{on }I_\ast\cap[3/4,\infty).
	\]

	Let $\mathcal R$ be a symmetric inverse-divergence operator on the
	torus, normalized so that
	\[
		\diver\mathcal R f=f
		\qquad\text{whenever}\qquad
		\int_{\T^3}f(x)\,\dd x=0.
	\]

	\begin{proposition}[Smooth initialization]\label{prop:init}
		Define
		\begin{align}
			u_0
			&:=\chi(t)(  u^{(0)}*\rho_\ell)(x)
			+(1-\chi(t))( u^{(1)}*\rho_\ell)(x),\label{eq:init-u}\\
			p_0&:=0,\notag\\ \E_0&:=\mathcal R( \partial_tu_0+\diver( u_0\otimes u_0)
			-\nu\Delta u_0).\label{eq:init-stress}
		\end{align}
		Then $(u_0,p_0,\E_0)$ is a smooth solution of
		\eqref{eq:NSR_0} on \(\T^3\times I_\ast\).  Moreover, for every $\varepsilon>0$, there exists
		$\ell>0$ such that
		\begin{equation}\label{eq:init-endpoints}
			\| u_0(\cdot,0)-u^{(0)}\|_{L^2}
			+\|u_0(\cdot,1)-u^{(1)}\|_{L^2}
			<\frac{\varepsilon}{2}.
		\end{equation}
	\end{proposition}
	
	\begin{proof}
		Spatial convolution preserves both incompressibility and zero mean.
		Consequently,
		\[
		\int_{\T^3}\partial_tu_0+\diver( u_0\otimes u_0)
		-\nu\Delta u_0\dd x=0
		\]
		for every $t$.  Thus
		$\E_0$ is well defined and symmetric, and
		\eqref{eq:init-u}--\eqref{eq:init-stress} give \eqref{eq:NSR_0}.
		Finally, $u_0(\cdot,0)=u^{(0)}*\rho_\ell$ and
		$u_0(\cdot,1)=u^{(1)}*\rho_\ell$, so
		\eqref{eq:init-endpoints} follows by taking $\ell$ sufficiently small.
	\end{proof}
	After \(\ell\) and \(\chi\) have been fixed, all relevant norms of the
	initial triple are finite.  By increasing \(\lambda_0\), we may arrange
	\begin{equation*}
		\|\mathcal{E}_0\|_{L_t^\infty L_x^1}\le \delta_{1},\quad
		\|u_0\|_{L_t^\infty L_x^2}
		\le \delta_0^{\frac12},\quad
		\|(\nabla_xu_0,\partial_tu_0)\|_{L_t^\infty L_x^4}
		\le \lambda_0^n.
	\end{equation*}
	Thus \((u_0,p_0,\E_0)\) satisfies the induction hypotheses at level
	\(q=0\).
	
	We use the following compactness statement.
	\begin{proposition}[Compactness and convergence]\label{prop:limit}
		Let $(u_q,p_q,\E_q)_{q\ge0}$ be a sequence of
		admissible solutions of
		\eqref{eq:NSR}, with every $u_q$ divergence free and of zero mean.
		Suppose that, for some \(\alpha_0,\vartheta>0\),
		\begin{align}
			\|\E_q\|_{L_t^\infty L_x^1}
			&\le\delta_{q+1},\label{eq:stress-bound}\\
			\|u_{q+1}-u_q\|_{C_tL_x^2}
			&\le C\delta_{q+1}^{1/2},\label{eq:L2-increment}\\
			\|\nabla u_{q+1}\|_{C_t^{\alpha_0}L_x^{\bar p}}
			&\le \|\nabla u_q\|_{C_t^{\alpha_0}L_x^{\bar p}}
			+C\delta_{q+1}^{\vartheta},\label{eq:gradient-recursion}\\
			\|u_{q+1}(\cdot,j)-u_{q}(\cdot,j)\|_{L^2}
			&\le C\lambda_{q}^{-1/5}, \qquad j=0,1, \label{eq:ini bound}
		\end{align}
		where $\lambda_{q+1}=\lambda_q^\sigma$ with $\sigma>1$ and
		$\delta_q=\lambda_1^{2\gamma}\lambda_q^{-\gamma}$ with $\gamma>0$.
		Then there exists
		\[
		u\in C([0,1];L^2_\sigma(\T^3))
		\quad\text{with}\quad
		\nabla u\in {C^{\alpha_0}}([0,1];L^{\bar p}(\T^3))
		\]
		such that $u_q\to u$ strongly in $C_tL_x^2$.  The field $u$ is a weak
		solution of the Navier--Stokes equation in the sense that
		\begin{align}\label{eq:weak-limit}
			&\int_0^1\!\!\int_{\T^3}
			\left[u\cdot\partial_t\varphi
			+(u\otimes u):\nabla\varphi
			+\nu u\cdot\Delta\varphi\right]\dd x\dd t
			+\int_{\T^3}u(x,0)\cdot\varphi(x,0)\dd x=0
		\end{align}
		for every divergence-free
		$\varphi\in C_c^\infty(\T^3\times[0,1);\R^3)$.
		If, in addition, \eqref{eq:init-endpoints} holds and
		\(C\sum_{q=0}^\infty\lambda_q^{-1/5}<\varepsilon/2\), then
		\begin{equation*}
			\|u(\cdot,j)-u^{(j)}\|_{L^2}<\varepsilon,
			\qquad j=0,1.
		\end{equation*}
	\end{proposition}
	
	\begin{proof}
		Since $\lambda_q=\lambda_0^{\sigma^q}$, the three series
		\begin{equation}\label{eq:summability}
			\sum_{q=0}^\infty\delta_{q+1}^{1/2},
			\qquad
			\sum_{q=0}^\infty\delta_{q+1}^{\vartheta},
			\qquad
			\sum_{q=0}^\infty\lambda_q^{-1/5}
		\end{equation}
		converge.  It follows from \eqref{eq:L2-increment} that $(u_q)_q$ is
		Cauchy in $C([0,1];L^2)$, so
		\begin{equation}\label{eq:strong-L2}
			u_q\longrightarrow u
			\qquad\text{strongly in }C([0,1];L^2)
		\end{equation}
		for some $u$ in that space.  Divergence and spatial mean pass to the
		strong $L^2$ limit, hence $u(t)\in L^2_\sigma$ for every $t$.
		
		Iterating \eqref{eq:gradient-recursion} and using
		\eqref{eq:summability} gives
		\begin{equation}\label{eq:uniform-gradient}
			\sup_{q\ge0}\|\nabla u_q\|_{C_t^{\alpha_0}L_x^{\bar p}}
			\le
			\|\nabla u_0\|_{C_t^{\alpha_0}L_x^{\bar p}}
			+C\sum_{q=0}^\infty\delta_{q+1}^{\vartheta}
			=:M<\infty.
		\end{equation}
		Fix $t\in[0,1]$.  Since $\bar p>1$, $L^{\bar p}$ is reflexive, and
		$(\nabla u_q(t))_q$ is weakly precompact in $L^{\bar p}$.
		The strong \(L^2\) convergence implies convergence in
		distributions, so every weakly convergent subsequence of
		\(\nabla u_q(t)\) has the unique limit \(\nabla u(t)\).  Hence the full
		sequence converges weakly, and
		\begin{equation*}
			u(t)\in W^{1,\bar p}(\T^3),
			\qquad
			\nabla u_q(t)\rightharpoonup\nabla u(t)
			\quad\text{weakly in }L^{\bar p}.
		\end{equation*}
		
		For $s,t\in[0,1]$, weak lower semicontinuity and
		\eqref{eq:uniform-gradient} yield
		\begin{align*}
			\|\nabla u(t)-\nabla u(s)\|_{L^{\bar p}}
			&\le\liminf_{q\to\infty}
			\|\nabla u_q(t)-\nabla u_q(s)\|_{L^{\bar p}}\\
			&\le M|t-s|^{\alpha_0}.
		\end{align*}
		Thus $\nabla u\in C_t^{\alpha_0}L_x^{\bar p}$.  
		
		It remains to pass to the equation.  Testing the $q$th Reynolds system
		against a divergence-free $\varphi$ gives
		\begin{align}\label{eq:weak-q}
			&\int_0^1\!\!\int_{\T^3}
			\left[u_q\cdot\partial_t\varphi
			+(u_q\otimes u_q):\nabla\varphi
			+\nu u_q\cdot\Delta\varphi\right]\dd x\dd t
			+\int_{\T^3}u_q(x,0)\cdot\varphi(x,0)\dd x\nonumber\\
			&\hspace{35mm}
			=\int_0^1\!\!\int_{\T^3}\E_q:\nabla\varphi\dd x\dd t.
		\end{align}
		The right-hand side tends to zero by \eqref{eq:stress-bound}.  Moreover,
		\eqref{eq:strong-L2} and the uniform $L^2$ bound imply
		\begin{align*}
			\|u_q\otimes u_q-u\otimes u\|_{C_tL_x^1}
			&\le
			\bigl(\|u_q\|_{C_tL_x^2}+\|u\|_{C_tL_x^2}\bigr)
			\|u_q-u\|_{C_tL_x^2}
			\longrightarrow0.
		\end{align*}
		All the remaining terms in \eqref{eq:weak-q} pass to the limit directly
		by \eqref{eq:strong-L2}, including the initial trace.  We obtain
		\eqref{eq:weak-limit}.  A distributional pressure may then be recovered
		from the periodic de Rham theorem.
		
		Finally, \eqref{eq:ini bound} and \eqref{eq:init-endpoints} give
		\begin{equation*}
			\|u(\cdot,j)-u^{(j)}\|_{L^2}\le \|u^{(j)}-u_0(\cdot,j)\|_{L^2}+\|u(\cdot,j)-u_0(\cdot,j)\|_{L^2}
			< \frac{\varepsilon}{2}+C\sum_{q=0}^\infty\lambda_q^{-1/5} < \varepsilon,
			\qquad j=0,1.
		\end{equation*}
	\end{proof}

	We now complete the proof of Theorem~\ref{thm:main}.  Choose \(\ell\)
	by Proposition~\ref{prop:init}, and then choose \(\lambda_0\) sufficiently
	large that the initial triple satisfies the induction hypotheses and
	\(C\sum_q\lambda_q^{-1/5}<\varepsilon/2\).  Repeated application of
	Proposition~\ref{iteration} gives a sequence satisfying the assumptions
	of Proposition~\ref{prop:limit}.  Its limit has the required regularity,
	solves \eqref{eq:intro-NS} in the sense of \eqref{eq:intro-weak}, and
	satisfies the two endpoint estimates.

	\subsection{Time locality and proof of Theorem~\ref{thm:nonuniqueness}}

	The endpoint-density statement alone does not imply exact
	nonuniqueness.  The extra ingredient is that one iteration step only
	uses the Reynolds flow in a short neighborhood of the current time cell.
	We record this feature in the form needed below.  At stage \(q\), write
	\begin{equation}\label{eq:locality-scale}
		l_q:=c_0\lambda_{q+1}^{-n/\sigma-\gamma}
	\end{equation}
	for the space--time mollification scale used in Section~4.

	\begin{lemma}[Time locality of the iteration]\label{lem:time-locality}
		Suppose that two Reynolds triples satisfying the hypotheses of
		Proposition~\ref{iteration}, with the same parameters, agree on
		\(\T^3\times(I_\ast\cap(-\infty,T_q])\).
		After fixing the same normalization for the
		pressures, the two applications of Proposition~\ref{iteration} may be
		coupled so that the resulting triples agree on
		
		\[
			\T^3\times
			\bigl(I_\ast\cap(-\infty,T_q-l_q-\tau_{q+1}]\bigr),
		\]
		
		provided \(T_q>l_q+\tau_{q+1}\).
	\end{lemma}

	\begin{proof}
		Since the two input triples agree on
		\(I_\ast\cap(-\infty,T_q]\), their
		mollifications, and hence the coefficients \(a_i\), agree on
		\(I_\ast\cap(-\infty,T_q-l_q]\).
		Every complete time cell contained in this interval
		therefore has the same frozen coefficients \(a_i^k\) in the two
		constructions.  On such a cell we choose the same admissible initial
		points for the trajectories.  The amplitudes, radii, trajectories,
		moving Hill blocks, and auxiliary sources then agree on that cell.

		The temporal corrector \eqref{eq:Q} is defined by integration only over
		the current cell.  The Leray projection, the symmetric
		inverse-divergence operator, and all other operators used to define
		\(u_{q+1}\), \(p_{q+1}\), and \(\E_{q+1}\) act only in space.
		Consequently, the two output triples agree on the union of the complete
		cells contained in
		\(I_\ast\cap(-\infty,T_q-l_q]\).  In particular, they
		agree on \([0,T_q-l_q-\tau_{q+1}]\), and the same left-time convention
		is preserved for the next stage.  This proves the claim.
	\end{proof}

	\begin{proof}[Proof of Theorem~\ref{thm:nonuniqueness}]
		Fix an arbitrary \(v\in L^2_\sigma(\T^3)\) and \(\varepsilon>0\).
		Choose \(w\in L^2_\sigma(\T^3)\) with \(\|w\|_{L^2}=1\), set
		\[
			w^{(1)}:=0,
			\qquad
			w^{(2)}:=w,
			\qquad
			\varepsilon_*:=\min\{\varepsilon/2,1/8\},
		\]
		and choose one spatial mollification scale \(\ell>0\), common to both
		constructions, such that
		\[
			\|v*\rho_\ell-v\|_{L^2}
			+
			\sum_{j=1}^2
			\|w^{(j)}*\rho_\ell-w^{(j)}\|_{L^2}
			<\frac{\varepsilon_*}{2}.
		\]
		For \(j=1,2\), apply Proposition~\ref{prop:init} with endpoint
		targets \((v,w^{(j)})\), using this common scale and the same cutoff
		\(\chi\).  Denote the resulting initial Reynolds triples by
		\((u_0^{(j)},p_0^{(j)},\E_0^{(j)})\).  Since
		\(\chi\equiv1\) on
		\(I_\ast\cap(-\infty,1/4]\), the two triples agree on
		\(\T^3\times(I_\ast\cap(-\infty,1/4])\).

		Choose one \(\lambda_0\), common to both constructions, sufficiently
		large that both initial triples satisfy the induction hypotheses,
		\begin{equation*}
			C\sum_{q=0}^{\infty}\lambda_q^{-1/5}
			<\frac{\varepsilon_*}{2},
			\qquad
			\sum_{q=0}^{\infty}(l_q+\tau_{q+1})<\frac18.
		\end{equation*}
		The second requirement is possible because \(l_q\) is given by
		\eqref{eq:locality-scale}, \(\tau_{q+1}=\lambda_{q+1}^{-\eta}\),
		and \(\lambda_q=\lambda_0^{\sigma^q}\).
		Apply Proposition~\ref{iteration} to the two triples with identical
		choices on every time cell on which they agree.  Iterating
		Lemma~\ref{lem:time-locality}, the level-\(q\) triples agree on

		\[
			\T^3\times\bigl(I_\ast\cap(-\infty,T_q]\bigr),
			\qquad
			T_q:=\frac14-
			\sum_{m=0}^{q-1}(l_m+\tau_{m+1})\ge\frac18.
		\]

		Proposition~\ref{prop:limit} now gives two limiting weak solutions
		\(u^{(1)},u^{(2)}\in\mathcal S_{\bar p}\).  Their equality on the
		intervals above and their convergence in \(C_tL_x^2\) imply
		\[
			u^{(1)}(\cdot,t)=u^{(2)}(\cdot,t)
			\quad\text{for }0\le t\le\frac18.
		\]
		In particular, they have a common initial trace, denoted by
		\(\widetilde v\), and the endpoint estimate in
		Proposition~\ref{prop:limit} gives
		\begin{equation}\label{eq:nonuniqueness-initial-density}
			\|\widetilde v-v\|_{L^2}<\varepsilon_*\le\varepsilon.
		\end{equation}
		At the other endpoint,
		\[
			\|u^{(j)}(\cdot,1)-w^{(j)}\|_{L^2}<\varepsilon_*,
			\qquad j=1,2.
		\]
		Therefore,
		\begin{align*}
			\|u^{(1)}(\cdot,1)-u^{(2)}(\cdot,1)\|_{L^2}
			&\ge \|w^{(1)}-w^{(2)}\|_{L^2}-2\varepsilon_*\\
			&\ge 1-\frac14>0,
		\end{align*}
		so the two solutions are distinct.

		Let \(\mathcal D\) be the set of all initial data in
		\(L^2_\sigma(\T^3)\) admitting two such distinct solutions.
		Since \(v\) and \(\varepsilon\) were arbitrary,
		\eqref{eq:nonuniqueness-initial-density} shows that \(\mathcal D\) is
		dense in \(L^2_\sigma(\T^3)\).  This proves the theorem.
	\end{proof}

	\section{Limitations and possible modifications}
	
	We record the limits of the present single-core construction under the
	condition \(\nabla u\in C_tL_x^{\bar p}\).  These calculations are not
	impossibility results for other Hill-vortex schemes.

	\subsection{Can the present method improve the exponent?}
	
	The principal obstruction is already visible in
	\eqref{res:extra 1}.  If
	\[
	\vartheta(p):=\frac53-\frac2p,
	\]
	then the formal principal-gradient estimate would require
	\begin{equation*}
		\vartheta(\bar p)
		<
		\frac{\gamma}{2(\mu+\gamma)}.
	\end{equation*}
	For the parameters used in the construction above,
	\(\gamma=10^{-3}\) and \(\mu=6+2\gamma=6.002\), this gives
	\begin{equation}\label{eq:discussion-pmax}
		\bar p
		<
		p_{\max}
		:=
		\frac{2}{
			\frac53-\frac{\gamma}{2(\mu+\gamma)}
		}
		=\frac{24012}{20009}
		\approx 1.20005997.
	\end{equation}
	Thus \(\bar p=1.20005=6/5+5\times10^{-5}\) is already close to the
	ceiling imposed by the principal-gradient estimate for these parameters.
	
	This value is not claimed to be optimal.  However, \(\gamma\) and \(\mu\)
	have competing roles: increasing \(\gamma\) improves amplitude decay,
	whereas decreasing \(\mu\) enlarges the cores.  The orbit-spacing condition
	\[
	1-\alpha\mu
	+3\alpha\left(\frac n\sigma+\gamma\right)<0
	\]
	keeps \(\mu\) relatively large, while the stress estimates restrict
	\(\gamma\).  Reoptimization may improve the numerical margin but does not
	provide a route toward \(p=3/2\); a larger gain requires a structural
	change in the construction.
	
	\subsection{Why temporal concentration does not improve the present theorem}
	
	Temporal concentration, as used in related convex-integration
	constructions such as \cite{CheskidovLuo2022,CheskidovLuo2023}, gives
	estimates in time-integrated spaces.  Concentrating a stress of size
	\(\delta\) on a set of relative measure \(\rho\) changes the perturbation
	amplitude to \(\delta^{1/2}\rho^{-1/2}\).  Since here \(R=ra\) with
	\(a\sim\delta/\rho\), one obtains
	\begin{equation*}
		\|\nabla V_i^k\|_{L_t^sL_x^p}
		\lesssim
		\rho^{\frac1s-\frac12+\vartheta(p)}
		\delta^{\frac12-\vartheta(p)}
		r^{-\vartheta(p)}.
	\end{equation*}
	For \(p\) slightly above \(6/5\), the factor at \(s=\infty\) is
	\(\rho^{-1/2+\vartheta(p)}\), which diverges as \(\rho\downarrow0\).
	Temporal cutoffs also worsen the transport and corrector errors.  Hence
	this device does not improve the present uniform-in-time estimate, although
	it may be useful for a theorem with finite time integrability.
	
	\subsection{Anisotropic localization}
	
	This direction is motivated by intermittent Beltrami, Mikado, and pipe
	constructions; see \cite{DeLellisSzekelyhidi2013,
	BuckmasterVicol2019,BuckmasterVicolHydrodynamics,DaneriSzekelyhidi2017,
	MS18,BuckmasterMasmoudiNovackVicol2023,GiriKwonNovack2024}.  These schemes
	use anisotropic spatial concentration.  The solutions in
	\cite{BuckmasterVicol2019} belong to
	\(C_t(H_x^\beta\cap W_x^{1,1+\beta})\) for some \(\beta>0\), and hence to
	\(C_tL_x^{2+a}\) for some unspecified \(a>0\).  The checkerboard
	localization of Novack and Vicol~\cite{NovackVicol2023} gives a related
	Euler example.  This motivated the target
	\[
	C_t\bigl(W_x^{1,p}\cap L_x^{2+a}\bigr),
	\qquad p>\frac65,
	\]
	with a larger velocity-integrability gain.  One may test the cutoff
	\[
	\chi_R^{a,b}(x)
	:=
	\chi\left(
	\frac{x_1}{R^a},
	\frac{(x_2,x_3)}{R^b}
	\right),
	\qquad 0<a,b<\frac23,
	\]
	and set
	\[
	H_R^{A}:=H_R-R\nabla(\chi_R^{a,b}\Phi).
	\]
	
	If \(H_R^A=H_R\) on \(B_{cR^{2/3}}\), the unchanged core gives
	\[
	\|H_R^{A}\|_{L^p(\T^3)}
	\ge \|H_R\|_{L^p(B_{cR^{2/3}})}
	\gtrsim R^{\frac2p-1},
	\qquad
	\|\nabla H_R^{A}\|_{L^p(\T^3)}
	\ge \|\nabla H_R\|_{L^p(B_{cR^{2/3}})}
	\gtrsim R^{\frac2p-\frac53}.
	\]
	Thus the cutoff does not change the intrinsic Hill-core loss.  It could
	help only through better packing or a smaller divergence-correction error,
	neither of which follows from the present estimates.
	
	\subsection{Several parallel moving cores}
	
	Another possibility is to split one directional stress among
	\(N=N_{q+1}\) disjoint cores moving in parallel.  The cancellation
	normalization would then give an amplitude of order
	\[
	\eta_N\sim\delta_{q+1}^{1/2}N^{-1/2}
	\]
	for each core.  If the cores have comparable radius \(R\) and disjoint
	supports, their combined fixed-time gradient satisfies schematically
	\begin{equation*}
		\left\|
		\sum_{j=1}^{N}\nabla V_j
		\right\|_{L^p}
		\lesssim
		\delta_{q+1}^{1/2}
		N^{\frac1p-\frac12}
		R^{\frac2p-\frac53}.
	\end{equation*}
	For \(p<2\), the factor \(N^{1/p-1/2}\) grows with \(N\).  Thus splitting
	the amplitude does not improve the fixed-time gradient estimate unless
	the new packing permits sufficiently larger cores.  A multicore
	improvement therefore requires a different packing or averaging mechanism,
	not only more copies of the same block.
	
	\section{Further directions and open problems}

	We conclude with several questions that lie beyond the present
	single-core iteration.
	
	\paragraph{Optimal uniform-in-time spatial regularity.}
	Determine the largest exponent \(p\) for which the endpoint-trace
	density in \eqref{eq:intro-trace-density} can hold with
	\(\nabla u\in C_tL_x^p\).  The ceiling in
	\eqref{eq:discussion-pmax} concerns only the chosen parameters.  Can
	optimization, localization, or a different averaging mechanism give a
	fixed gain above \(6/5\)?  The scale-invariant value is \(p=3/2\), for which
	\(W^{1,3/2}(\T^3)\hookrightarrow L^3(\T^3)\), but the present estimates
	do not approach it.

	\paragraph{Energy profiles and admissibility.}
	Prescribed energy profiles and admissibility criteria have been studied in
	Euler convex integration; see
	\cite{DeLellisSzekelyhidi2010,DaneriRunaSzekelyhidi2021,Isett2022,
	BuckmasterDeLellisSzekelyhidiVicol2019}.
	Can the construction prescribe
	\(E(t)=\frac12\|u(\cdot,t)\|_{L^2}^2\) while retaining
	\(\nabla u\in C_tL_x^{\bar p}\), or give flexibility among
	energy-compatible endpoints in the Leray--Hopf or suitable class?
	Arbitrary endpoint density is incompatible with nonincreasing energy, and
	the present estimates give neither \(L_t^2H_x^1\) control nor a local
	energy inequality.
	
	\paragraph{Coherent vortices beyond the Hill profile.}
	The Hill energy--gradient scaling gives the threshold \(6/5\).  Possible
	alternatives include Fraenkel's rings of small
	cross-section~\cite{Fraenkel1972}, Norbury's one-parameter
	family~\cite{Norbury1973} interpolating between thin rings and Hill's
	vortex, and the variational constructions of Fraenkel and
	Berger~\cite{FraenkelBerger1974} and Friedman and
	Turkington~\cite{FriedmanTurkington1981}.  Their aspect ratio provides an
	additional parameter, although the induced velocity may offset the apparent
	\(L^p\)-gain.  Hill's vortex is also the unique extreme member of the
	family~\cite{AmickFraenkel1986}.  Any replacement must retain exact
	traveling cancellation, a controllable impulse~\cite[Ch.~3]{Saffman1992},
	localizability, and small viscous and modulation stresses.  Whether a
	vortex-ring profile can meet these requirements and improve the exponent is
	open.

	\paragraph{Numerically assisted search for new building blocks.}
	Numerical and computer-assisted work on Navier--Stokes
	nonuniqueness~\cite{HouWangYang2026,IonescuJiaPalasek2026} studies
	candidate self-similar profiles and unstable modes; related tools appear for
	Euler~\cite{ChenHou2022,
	WangLaiGomezSerranoBuckmaster2023} and in rigorous a posteriori error
	estimates~\cite{BressanShen2021}.  For the present method, one would instead
	seek a traveling or pulsating profile with better energy--gradient scaling
	and a certified residual small in the required stress norms.  Numerical
	optimization may help identify candidates, but the cited instability
	mechanisms do not supply the cancellation identities needed here.
	
	\paragraph{Trajectory geometry and nonperiodic domains.}
	Extending the construction to \(\R^3\) requires a substitute for periodic
	recurrence and control of spatial decay; bounded domains also require
	boundary-compatible blocks.  For three-dimensional Euler, Enciso,
	Pe\~nafiel-Tom\'as, and
	Peralta-Salas~\cite{EncisoPenafielPeralta2025} proved that, under explicit
	compatibility conditions, a smooth local solution on a bounded region can
	be extended to an admissible \(C^\beta\) weak solution on \(\R^3\),
	\(\beta<1/3\), with controlled support.  Their result does not provide
	boundary-compatible Hill vortices or treat viscosity.  The inner--outer
	gluing in \cite{AlbrittonBrueColombo2023} is another precedent.  On the
	torus, partial transport by the coarse velocity may reduce the linear
	interaction, but would require averaging along deformed trajectories.
	
	These questions ask whether the gain above \(6/5\) extends beyond the
	present moving-Hill-vortex construction.

	\subsection*{Acknowledgment}

	Quoc-Hung Nguyen's research was supported by the CAS Project for Young
	Scientists in Basic Research (Grant No.~YSBR-031) and by the National
	Natural Science Foundation of China under Grant Nos.~1251101538 and
	12595282.  He is grateful to Elia Bru\`e for suggesting that he consider
	this problem.

	\subsection*{AI use disclosure}

	This paper was written by the authors and was not generated by artificial
	intelligence.  OpenAI's ChatGPT~5.6 was used to assist with language
	editing, organization, bibliographic checks, \LaTeX{} troubleshooting, and
	verification of mathematical calculations.  It was not used as an
	independent source of mathematical results.  All mathematical ideas,
	arguments, proofs, and conclusions are those of the authors, who reviewed
	the AI-assisted revisions and take full responsibility for the contents of
	the paper.

\end{document}